\documentclass[reqno]{amsart}

\usepackage{amsmath,amsfonts,amssymb,amscd,amsthm,amsbsy,bm,latexsym}
\usepackage{gensymb}
\usepackage[pdftex,bookmarksnumbered=true,bookmarksopen=true,          
            colorlinks=true,pdfborder=001,citecolor=blue,
            linkcolor=red,anchorcolor=green,urlcolor=blue]{hyperref}

\allowdisplaybreaks

\usepackage{color}

\newtheorem{lemma}{Lemma}[section]
\newtheorem{theorem}{Theorem}[section]

\newtheorem{remark}{Remark}[section]

\numberwithin{equation}{section}

\begin{document}

\title[The Vlasov-Maxwell-Boltzmann system with very soft potentials]{The Vlasov-Maxwell-Boltzmann system near Maxwellians in the whole space with very soft potentials}

\author[R.-J. Duan]{Renjun Duan}
\address[RJD]{Department of Mathematics, The Chinese University of Hong Kong,
Shatin, Hong Kong}
\email{rjduan@math.cuhk.edu.hk}

\author[Y.-J. Lei]{Yuanjie Lei}
\address[YJL]{School of Mathematics and Statistics, Wuhan University,
Wuhan 430072, P.R.~China}
\email{leiyuanjie@whu.edu.cn}

\author[T. Yang]{Tong Yang}
\address[TY]{Department of mathematics, City University of Hong Kong,
Tat Chee Avenue, Kowloon, Hong Kong}
\email{matyang@cityu.edu.hk}

\author[H.-J. Zhao]{Huijiang Zhao}
\address[HJZ]{School of Mathematics and Statistics, Wuhan University,
Wuhan 430072, P.R.~China}
\email{hhjjzhao@whu.edu.cn}


\begin{abstract}
Since the work \cite{Guo-Invent-03} by Guo [Invent. Math. \textbf{153} (2003), no. 3, 593--630], how to establish the global existence of perturbative classical solutions around a global Maxwellian to the Vlasov-Maxwell-Boltzmann  system
with the whole range of soft potentials has been an open problem.  This is mainly due to  the complex structure of the system, in particular, the degenerate
dissipation at large velocity, the velocity-growth of the nonlinear term induced by the Lorentz force, and the regularity-loss of the electromagnetic fields.
This  paper aims to resolve this problem
 in  the whole space provided that initial
perturbation has sufficient  regularity and velocity-integrability.
\end{abstract}

\maketitle
\thispagestyle{empty}
\tableofcontents

\section{Introduction}

\setcounter{equation}{0}
The motion of dilute ionized plasmas consisting of two-species particles (e.g., electrons and ions) under the influence of binary collisions and the self-consistent electromagnetic field can be modelled by the Vlasov-Maxwell-Boltzmann system (cf. \cite[Chapter 19]{Chapman-Cowling-1970} as well as \cite[Chapter 6.6]{Krall-Trivelpiece})
\begin{eqnarray}
 \partial_tF_++ v  \cdot\nabla_xF_++\frac{e_+}{m_+}\left(E+\frac vc\times B\right)\cdot\nabla_{ v  }F_+&=&Q(F_+,F_+)+Q(F_+,F_-),\nonumber\\
 \partial_tF_-+ v  \cdot\nabla_xF_--\frac{e_-}{m_-}\left(E+\frac vc\times B\right)\cdot\nabla_{ v  }F_-&=&Q(F_-,F_+)+Q(F_-,F_-).\label{VMB}
\end{eqnarray}
The  electromagnetic field $[E,B]=[E(t,x), B(t,x)]$ satisfies the  Maxwell equations
\begin{eqnarray}
 \partial_tE-c\nabla_x\times B&=&-4\pi{\displaystyle\int_{\mathbb{R}^3}}v\left(e_+F_+-e_-F_-\right)dv,\nonumber\\
 \partial_tB+c\nabla_x\times E&=&0,\label{Maxwell}\\
 \nabla_x\cdot E&=&4\pi{\displaystyle\int_{\mathbb{R}^3}}\left(e_+F_+-e_-F_-\right)dv,\nonumber\\
 \nabla_x\cdot B&=&0.\nonumber
\end{eqnarray}
Here $\nabla_x=\left(\partial_{x_1}, \partial_{x_2},\partial_{x_3}\right), \nabla_v=\left(\partial_{v_1}, \partial_{v_2},\partial_{v_3}\right)$.  The unknown functions $F_\pm= F_\pm(t,x, v) \geq  0$ are the number density functions for the ions ($+$) and electrons ($-$) with position $x = (x_1, x_2, x_3)\in {\mathbb{R}}^3$ and velocity $ v=( v_1,  v_2,  v_3) \in {\mathbb{R}}^3$ at time $t\geq 0$, respectively, $e_\pm$ and $m_\pm$ the magnitudes of their charges and masses, and $c$ the speed of light.

Let $F(v)$, $G(v)$ be two number density functions for two types of particles with masses $m_\pm$
and diameters $\sigma_\pm$, then $Q(F,G)(v)$ is
defined as (cf.~\cite{Chapman-Cowling-1970})
\begin{equation*}
\begin{split}
  Q(F,G)(v)
  =&\frac{(\sigma_++\sigma_-)^2}{4}\int_{\mathbb{R}^3\times \mathbb{S}^2}|u-v|^{\gamma}{\bf b}\left(\frac{\omega\cdot(v-u)}{|u-v|}\right)\left\{F(v')G(u')-F(v)G(u)\right\}
  d\omega du\\
  \equiv&Q_{gain}(F,G)-Q_{loss}(F,G).
  \end{split}
\end{equation*}
Here $\omega\in\mathbb{S}^2$ and ${\bf b}$, the angular part of the collision kernel, satisfies the Grad cutoff assumption (cf.~\cite{Grad-1963})
\begin{equation}\label{cutoff-assump}
0\leq{\bf b}(\cos \theta)\leq C|\cos\theta|
\end{equation}
for some positive constant $C>0$. The deviation angle $\pi-2\theta$ satisfies $\cos\theta =\omega\cdot(v-u)/{|v-u|}$.
Moreover, for $m_{1}, m_2\in \{m_+,m_-\}$,
\begin{equation*}
v'=v-\frac{2m_2}{m_1+m_2}[(v-u)\cdot\omega]\omega,\quad u'=u+\frac{2m_1}{m_1+m_2}[(v-u)\cdot\omega]\omega,
\end{equation*}
which denote velocities $(v',u')$ after a collision of particles having velocities $(v, u)$ before the collision and vice versa. Notice that the above identities follow from the conservation of momentum $m_1v+m_2u$ and energy $\frac 12
m_1|v|^2+\frac 12m_2|u|^2$.

The exponent $\gamma\in(-3,1]$ in the kinetic part of the collision kernel is determined by the potential of intermolecular force, which is classified into the soft potential case when $-3<\gamma<0$, the Maxwell molecular case when $\gamma=0$, and the hard potential case when $0<\gamma\leq 1$ which includes the hard sphere model with $\gamma=1$ and ${\bf b}(\cos\theta)=C|\cos\theta|$ for some positive constant $C>0$. For the soft potentials, the case $-2\leq \gamma<0$ is called the moderately soft potentials while $-3<\gamma<-2$ is called the very soft potentials, cf. \cite{Villani-02} by Villani. The importance and the difficulty in studying
the very soft potentials can be also found in that review paper.

The main purpose of this work is to construct global classical solutions to the Vlasov-Maxwell-Boltzmann system \eqref{VMB}, \eqref{Maxwell}
for the whole range of soft potentials, in particular, the very soft
case when $-3<\gamma<-2$, near global Maxwellians
\begin{eqnarray*}
\mu_+(v)&=&\frac{n_0}{e_+}\left(\frac{m_+}{2\pi\kappa T_0}\right)^{\frac 32}\exp\left(-\frac{m_+|v|^2}{2\kappa T_0}\right),\\
\mu_-(v)&=&\frac{n_0}{e_-}\left(\frac{m_-}{2\pi\kappa T_0}\right)^{\frac 32}\exp\left(-\frac{m_-|v|^2}{2\kappa T_0}\right),
\end{eqnarray*}
in the whole space $\mathbb{R}^3$ with prescribed initial data
\begin{equation}\label{VMB-in}
  F_\pm (0,x,v)=F_{0,\pm}(v,x), \quad E(0,x)=E_0(x), \quad B(0,x)=B_0(x),
\end{equation}
which satisfy the compatibility conditions
\begin{equation*}
  \nabla_x\cdot E_0=\int_{\mathbb{R}^3}(F_{0,+}-F_{0,-})dv, \quad \nabla_x\cdot B_0=0.
\end{equation*}

We assume in the paper that all the physical constants are chosen to be one. Under such assumption,
accordingly we normalize the above Maxwellians as (with $E(t,x)\equiv B(t,x)\equiv 0$)
$$
\mu=\mu_-(v)=\mu_+(v)=(2\pi)^{-\frac 32}e^{-\frac{|v|^2}{2}}.
$$
To study the stability problem around $\mu$, we define the perturbation $f_\pm=f_\pm(t,x, v)$ by
$$
F_\pm(t, x,  v  ) = \mu+ \mu^{1/2}f_\pm(t, x,  v).
$$
Then, the Cauchy problem (\ref{VMB}), \eqref{Maxwell}, \eqref{VMB-in} is reformulated as
\begin{equation} \label{f}
\begin{cases}
  \partial_tf_\pm+ v  \cdot\nabla_xf_\pm\pm(E+v\times B)\cdot\nabla_{ v  }f_\pm\mp E \cdot v \mu^{1/2}\mp \frac12 E\cdot v f_\pm+{ L}_\pm f={\Gamma}_\pm(f,f),\\
\partial_tE-\nabla_x\times B=-{\displaystyle\int_{\mathbb{R}^3}}v\mu^{1/2}(f_+-f_-)dv,\\
\partial_tB+\nabla_x\times E=0,\\
\nabla_x\cdot E={\displaystyle\int_{\mathbb{R}^3}}\mu^{1/2}(f_+-f_-)dv,\quad \nabla_x\cdot B=0
\end{cases}
\end{equation}
with initial data
\begin{equation}\label{f-initial}
f_\pm(0,x,v)=f_{0,\pm}(x,v),  \quad E(0,x)=E_0(x), \quad B(0,x)=B_0(x)
\end{equation}
satisfying the compatibility conditions
\begin{equation}\label{compatibility conditions}
  \nabla_x\cdot E_0=\int_{\mathbb{R}^3}\mu^{1/2}(f_{0,+}-f_{0,-})dv, \quad \nabla_x\cdot B_0=0.
\end{equation}
Here, as in \cite{Guo-Invent-03}, for later use, setting $f=\left[f_+,f_-\right]$, the first equation of $(\ref{f})$ can be also written as
\begin{equation}\label{f gn}
\partial_t f+v\cdot\nabla_xf+q_0(E+v\times B)\cdot\nabla_v f-E\cdot v\mu^{1/2}q_1+Lf=\frac{q_0}{2}E\cdot vf+\Gamma(f,f),
\end{equation}
where $q_0={\textrm{diag}}(1,-1)$, $q_1=[1,-1]$, and the linearized collision operator $L=[L_+,L_-]$ and the nonlinear collision operator $\Gamma=[\Gamma_+,\Gamma_-]$ are to be given later on.

We are now ready to state the main theorem in this paper.

\begin{theorem}\label{thm.mr}
Let $-3<\gamma<-1$ and \eqref{cutoff-assump} hold. Assume $F_0(x,v)=\mu+\sqrt{\mu}f_0(x,v)\geq0$. Take $1/2\leq \varrho<3/2$ and $0<q\ll 1$.  Let $N$ be an appropriately chosen integer and $l_0^\ast$ be a large enough constant to be specified in the proof. If
\begin{equation*}
\sum_{|\alpha|+|\beta|\leq N}\left\|\langle v\rangle^{l_0^\ast-|\beta|}
e^{q\langle v\rangle^2}
\partial_\beta^\alpha f_0\right\|+\|f_0\|_{L^2_v(\dot{H}^{-\varrho})}+\|(E_0,B_0)\|_{H^N\bigcap \dot{H}^{-\varrho}}
\end{equation*}
is sufficiently small,  then  the Cauchy problem (\ref{f}), (\ref{f-initial}), \eqref{compatibility conditions} admits a unique global solution $[f(t,x,v),$ $E(t,x),$ $ B(t,x)]$ satisfying $F(t,x,v)=\mu+\sqrt{\mu}f(t,x,v)\geq0$.

\end{theorem}

In the next section, the statement of the above theorem
will be given  more precisely in Theorem \ref{Th1.1} as well as Theorem \ref{Th1.2} for the time decay property.  Basically the result shows that {as long as initial data is small with enough regularity, one can establish the global existence of small amplitude classical solutions for the full range of cutoff intermolecular interactions with $-3<\gamma\leq 1$. Note that the case $-1\leq \gamma\leq 1$ is  a  trivial consequence of \cite{Duan_Liu-Yang_Zhao-VMB-2013}; details for that case will  be briefly discussed in Section \ref{exp-h}.} Here,
the bound in the  Sobolev space of negative index is {used  for
obtaining the time decay of solutions that is needed to close the a priori estimates.}

The proof of Theorem \ref{thm.mr} is based on a subtle
 time-weighted energy method.
For this, in addition
to the  existing analytic
 techniques used in \cite{Guo-JAMS-11} and \cite{Duan_Liu-Yang_Zhao-VMB-2013}, we  develop a new approach to deal with the weighted estimates involving both the negative power time-weight  and  the time-velocity dependent $w_{\ell-|\beta|,\kappa}(t,v)$ weight.

The rest of this paper is organized as follows. In Section 2, we  explain
the difficulty in studying {the case when $-3<\gamma<-1$, particularly including the very soft potential case,} and give
a {complete} statement of the main results.
 In Section $3$, we list some basic lemmas for later use.
The proof of the main results will be given in Section 4. For
clear presentation, the  proofs of several technical
 lemmas and estimates used in Section 4 will be given in the appendix.

\section{{Main results}}

In this section, we will first review the previous approaches for studying
the global existence of classical solutions to Valsov-Maxwell-Boltzmann
equations, and then we will point out the difficulties in studying the
very soft potentials and give the {complete} statements of the main results.

First of all, we recall some basic facts concerning the collision operators and the macro-micro decomposition.  $L,\Gamma$ in \eqref{f gn} are respectively defined by
\begin{equation*}
Lf=[L_+f,L_-f],\quad\quad\quad\Gamma(f,g)=[\Gamma_+(f,g),\Gamma_-(f,g)]
\end{equation*}
with
\begin{equation*}
\begin{aligned}
{ L}_\pm f =& -{\bf \mu}^{-1/2}
\left\{{Q\left(\mu,{\bf \mu}^{1/2}(f_\pm+f_\mp)\right)+ 2Q\left(\mu^{1/2}f_\pm, \mu\right)}\right\},\\{}
{ \Gamma}_{\pm}(f,g) =&{\bf \mu}^{-1/2}\left\{Q\left({\bf \mu}^{1/2}f_{\pm},{\bf \mu}^{1/2}g_\pm\right)+Q\left({\bf \mu}^{1/2}f_{\pm},{\bf \mu}^{1/2}g_\mp\right)\right\}.
\end{aligned}
\end{equation*}
For the linearized  collision operator $ L $, it is well known (cf.~\cite{Guo-Invent-03}) that it is non-negative and the null space $\mathcal{N}$ of ${ L}$ is spanned by
\begin{equation*}
{\mathcal{ N}}={\textrm{span}}\left\{[1,0]\mu^{1/2} , [0,1]\mu^{1/2}, [v_i,v_i]{\mu}^{1/2} (1\leq i\leq3),[|v|^2,|v|^2]{\bf \mu}^{1/2}\right\}.
\end{equation*}
Moreover,  under Grad's angular cutoff assumption \eqref{cutoff-assump}, it is easy to see that $L$ can be decomposed as
\begin{equation}\label{decomposition-L}
{L}f=\nu f-Kf
\end{equation}
with the collision frequency $\nu(v)$ and the nonlocal integral operator $K=[K_+,K_-]$ being defined by
\begin{equation}\label{collision-frequency}
\nu(v)=2Q_{loss}(1,\mu)=2\int_{\mathbb{R}^3\times {\mathbb{S}}^2}|v-u|^\gamma {\bf b}\left(\frac{\omega\cdot(v-u)}{|v-u|}\right)\mu(u)d\omega du\thicksim(1+|v|)^\gamma,
\end{equation}
and
\begin{multline}\label{Operator-K}
\left(K_\pm f\right)(v)
={\mu}^{-\frac 12}
\left\{2Q_{gain}\left(\mu^{\frac 12}f_\pm, \mu\right)-Q\left(\mu,{\mu}^{\frac 12}(f_\pm+f_\mp)\right)\right\}\\
=\int_{\mathbb{R}^3\times \mathbb{S}^2}|u-v|^\gamma {\bf b}\left(\frac{\omega\cdot(v-u)}{|v-u|}\right)\mu^{\frac 12}(u)\qquad\qquad\qquad\\
\left\{2\mu^{\frac 12}(u')f_\pm(v')-\mu^{\frac 12}(v') (f_\pm+f_\mp)(u')
+\mu^{\frac 12}(v)(f_\pm+f_\mp)(u)\right\}d\omega du,
\end{multline}
respectively.

Define ${\bf P}$ as the orthogonal projection from $L^2({\mathbb{R}}^3_ v)\times L^2({\mathbb{R}}^3_ v)$ to $\mathcal{N}$. Then for any given function $f(t, x, v )\in L^2({\mathbb{R}}^3_ v)\times L^2({\mathbb{R}}^3_ v)$, one has
\begin{equation*}
  {\bf P}f ={a_+(t, x)[1,0]\mu^{1/2}+a_-(t, x)[0,1]\mu^{1/2}+\sum_{i=1}^{3}b_i(t, x) [1,1]v_i{\mu}^{1/2}+c(t, x)[1,1](| v|^2-3)}{\bf \mu}^{1/2}
\end{equation*}
with
\begin{equation*}
  a_\pm=\int_{{\mathbb{R}}^3}{\bf \mu}^{1/2}f_\pm d v,\quad
  b_i=\frac12\int_{{\mathbb{R}}^3} v  _i {\bf \mu}^{1/2}(f_++f_-)d v,\quad
  c=\frac{1}{12}\int_{{\mathbb{R}}^3}(| v|^2-3){\bf \mu}^{1/2}(f_++f_-) d v.
\end{equation*}
Therefore, we have the following macro-micro decomposition with respect to the given global Maxwellian $\mu(v)$, cf.  \cite{Guo-IUMJ-04},
\begin{equation*}
 f(t,x, v)={\bf P}f(t,x, v)+\{{\bf I}-{\bf P}\}f(t, x, v),
\end{equation*}
where ${\bf I}$ denotes the identity operator, and ${\bf P}f$ and $\{{\bf I}-{\bf P}\}f$ are called the macroscopic and the microscopic component of $f(t,x,v)$, respectively.

Under the Grad's angular cutoff assumption \eqref{cutoff-assump},
by  \cite[Lemma 1]{Guo-Invent-03},  $L$ is locally coercive in the sense that
\begin{equation}\label{coercive-estimates}
-\left\langle f, Lf\right\rangle\geq \sigma_0\left|\{{\bf I}-{\bf P}\}f\right|_\nu^2\equiv \sigma_0\left\|\sqrt{\nu}\{{\bf I}-{\bf P}\}f\right\|^2_{L^2(\mathbb{R}^3_v)},\quad \nu(v)\sim(1+|v|)^{\gamma}
\end{equation}
holds for some positive constant $\sigma_0>0$. Here $\langle\cdot, \cdot\rangle$ denotes the   inner product in $L^2(\mathbb{R}^3_v)\times L^2(\mathbb{R}^3_v)$.

\subsection{Existing approaches}

For the problem on the construction of solutions to the Cauchy problem \eqref{f}, \eqref{f-initial}, \eqref{compatibility conditions}, the local existence and uniqueness of solution $[f_+(t,x,v), f_-(t,x,v), E(t,x), B(t,x)]$ in certain weighted Sobolev space to be specified later can be obtained by combining the arguments used in \cite{Guo-Invent-03} and \cite{Guo-JAMS-11}.
To extend the local solution $[f_+(t,x,v), f_-(t,x,v), E(t,x), B(t,x)]$
to be global in time, one needs to deduce certain a priori estimates in
some function spaces.
In general,   the main difficulties in this step lies in:
\begin{itemize}
\item How to control the possible velocity-growth  induced by the nonlinearity
of the system \eqref{f gn}?
\item How to control the convection  term $v\cdot\nabla_xf$ in the weighted
energy estimates?
\end{itemize}

The nonlinear energy method developed in \cite{Guo-CMP-02,Guo-IUMJ-04, Liu-Yang-Yu, Liu-Yu} for the Boltzmann equation provides an effective approach
 in the perturbative framework. The main idea in those work is to decompose the solution into the macroscopic part and the microscopic part and then rewrite the original equation as the combination of an equation satisfied by the microscopic part which contains the macroscopic part as source term and a system satisfied by the macroscopic part with the microscopic part as source term. In the perturbative framework, the dissipative mechanism  on the microscopic part is the coercive estimate \eqref{coercive-estimates} of the linearized Boltzmann collision operator or its weighted variants, while for the macroscopic part, the corresponding mechanism comes
from the dissipation of the compressible Navier-Stokes type system.
The corresponding approach to treat the case of non-cutoff cross sections was developed in \cite{AMUXY-JFA-2012} and \cite{Gressman_Strain-JAMS-2011}.

However, as pointed out in \cite{Guo-JAMS-11} and \cite{Duan_Liu-Yang_Zhao-VMB-2013},
when one applies the energy method to some complex systems such as the Vlasov-Maxwell-Boltzmann system \eqref{VMB}, \eqref{Maxwell},  in addition to the difficulty caused by the nonlinear collision operator mentioned above, additional difficulties are encountered:

\begin{itemize}
\item How to control the corresponding nonlinear terms induced by the Lorentz force, such as the terms $(E+v\times B)\cdot \nabla_v f$ and $E\cdot vf$, that can lead to  velocity growth at the rate of the first order $|v|$?

\item How to cope with the regularity loss of the electromagnetic field $[E(t,x), B(t,x)]$?
\end{itemize}

For the hard sphere model, the coercive estimate \eqref{coercive-estimates} of $L$ is sufficient to control the nonlinear terms related to the Lorentz force provided that the electromagnetic field $[E(t,x), B(t,x)]$ is suitably small and thus   satisfactory global well-posedness theory for the Vlasov-Maxwell-Boltzmann system \eqref{VMB}, \eqref{Maxwell} for the hard sphere model
has been established, cf. \cite{DS-CPAM-11, Guo-Invent-03, GS-12,Jang-ARMA-2009, Strain-CMP-2006} and the references therein. But for the corresponding problem involving cutoff non-hard sphere intermolecular interactions with $\gamma<1$, the story is quite different. One can not use the coercive estimate \eqref{coercive-estimates} of $L$ to absorb the nonlinear terms related to the Lorentz force which yield the velocity growth at the rate of the first order $|v|$.  Thus it is  interesting and important to find out how to construct global classical solutions near Maxwellians to the Vlasov-Maxwell-Boltzmann system \eqref{VMB}, \eqref{Maxwell} for cutoff non-hard sphere cases. Certainly, the same applies to the Vlasov-Poisson-Landau system and the Vlasov-Maxwell-Landau system containing the Coulomb potential,  cf. \cite{Guo-JAMS-11, S-Z-12, Wang-12} and \cite{Duan-VML, S-G-04}, respectively;  and the Vlasov-Poisson-Boltzmann system for non-hard sphere interactions cf. \cite{Duan_Liu-2012, Duan_Yang_Zhao-MMMA-2012, Xiao-Xiong-Zhao-2014}.

Particularly, a breakthrough was made in
 Guo's work  \cite{Guo-JAMS-11} on the two-species Vlasov-Poisson-Landau system in a periodic box, that leads to  the subsequent works for the Vlasov-Poisson-Landau system in the whole space mentioned above.
The main ideas  can be outlined as follows:
\begin{itemize}
\item An exponential weight of electric potential $e^{\mp\phi}$ is used to cancel the growth of the velocity in the nonlinear term
$\mp\frac 12\nabla_x\phi\cdot v f_\pm$.
\item A velocity weight
$$
\overline{w}_{l-|\alpha|-|\beta|}(v)=\langle v\rangle^{-(\gamma+1)(l-|\alpha|-|\beta|)},\quad \langle v\rangle=\sqrt{1+|v|^2},\quad l\geq |\alpha|+|\beta|
$$
is used to compensate the weak dissipation
of the linearized Landau kernel $\mathcal{L}$ for the case of $-3\leq \gamma<-2$;
\item The decay of the electric field $\phi(t,x)$ is used to close the energy estimate.
\end{itemize}

However,  since the Lorentz force $E+v\times B$ is not of the potential form, the argument in \cite{Guo-JAMS-11} can not be directly adopted to study
the Vlasov-Maxwell-Boltzmann system \eqref{VMB}, \eqref{Maxwell}.
For this, a time-velocity weighted energy method is introduced in \cite{Duan_Yang_Zhao-MMMA-2012}
by using the following weight $\widetilde{w}_{\ell,|\beta|}(t,v)$ function:
\begin{equation}\label{weight}
    \widetilde{w}_{\ell-|\beta|}\equiv\widetilde{w}_{\ell-|\beta|}(t,v)
    =\langle v\rangle^{-\gamma(\ell-|\beta|)}e^{\frac{q\langle v\rangle^2}{(1+t)^{\vartheta}}},\quad 0<q \ll 1, \quad |\beta|\leq \ell,\quad 0<\vartheta\leq\frac 14.
\end{equation}
Here it is worth pointing out that, unlike the weight function $\overline{w}_{l-|\alpha|-|\beta|}(v)$, the  algebraic factor $\widetilde{w}^a_{\ell-|\beta|}(v)=\langle v\rangle^{-\gamma(\ell-|\beta|)}$ in \eqref{weight} varies only with the order of the $v-$derivatives to represent  the fact that the dissipative effect of the cutoff linearized Boltzmann collision operator $L$ is  ``weaker'' than that of the linearized Landau collision operator $\mathcal{L}$.

\subsection{Difficulties for very soft potentials}\label{exp-h}
To illustrate the main ideas used in \cite{Duan_Liu-Yang_Zhao-VMB-2013, Duan_Yang_Zhao-MMMA-2012}  for $-1\le\gamma\le 1$, and the problem to be studied in this paper, we first introduce the following general weight function 
\begin{equation}\label{our-weight}
w_{\ell-|\beta|,\kappa}\equiv w_{\ell-|\beta|,\kappa}(t,v)=\langle v\rangle^{\kappa(\ell-|\beta|)}e^{\frac{q\langle v\rangle^2}{(1+t)^\vartheta}},\quad \kappa\geq 0,\quad 0<q\ll 1,
\end{equation}
where the precise range of the parameter $\vartheta$ will be specified later. It is easy to see that
$$
w_{\ell-|\beta|,-\gamma}(t,v)\equiv \widetilde{w}_{\ell-|\beta|}(t,v).
$$
Since for cutoff non-hard sphere intermolecular interactions, the macroscopic part can be controlled as for the case of hard sphere model, the main difficulty for the case of non-hard sphere model is to control the microscopic component $\{{\bf I}-{\bf P}\}f(t,x,v)$ suitably. The idea  for that purpose is to use the following two types of dissipative mechanisms:
\begin{itemize}
\item The first one is the dissipative term
\begin{equation*}
D^L_{|\alpha|,\ell-|\beta|,\kappa}\equiv\left\|w_{\ell-|\beta|,\kappa}\partial^\alpha_\beta\{{\bf I}-{\bf P}\}f\right\|^2_\nu\equiv\left\|\sqrt{\nu}w_{\ell-|\beta|,\kappa}\partial^\alpha_\beta\{
{\bf I}-{\bf P}\}f\right\|^2_{L^2(\mathbb{R}^3_v\times\mathbb{R}^3_x)}
\end{equation*}
from the coercive estimate of the linearized collision operator $L$;

\item The second type is the extra dissipative term
\begin{eqnarray}\label{dissipative-weight}
D^W_{|\alpha|,\ell-|\beta|,\kappa}
&\equiv& (1+t)^{-1-\vartheta}\left\| w_{\ell-|\beta|,\kappa}\partial^\alpha_\beta\{{\bf I}-{\bf P}\}f\langle v\rangle \right \|^2_{L^2(\mathbb{R}^3_v\times\mathbb{R}^3_x)}\nonumber
\end{eqnarray}
induced by  the weight function $w_{\ell-|\beta|,\kappa}(t,v)$.
\end{itemize}
The most difficult terms to be studied are:
\begin{itemize}
\item The term
\begin{equation}\label{I^{lt}}
I^{lt}_{|\alpha|,\ell-|\beta|,\kappa}\equiv\left(\partial^{\alpha+e_i}_{\beta-e_i}\{{\bf I-P}\}f,w_{\ell-|\beta|,\kappa}^2\partial^\alpha_\beta\{{\bf I-P}\}f\right)
\end{equation}
related to the linear transport term $v\cdot\nabla_x f$;

\item  The terms containing the electromagnetic field $[E(t,x), B(t,x)]$, i.e.
\begin{equation}\label{difficulty-1}
I^E_{|\alpha|,\ell-|\beta|,\kappa}\equiv\sum_{|\alpha_1|\geq1}\left(\partial^{\alpha_1}E\cdot \nabla_v\partial^{\alpha-\alpha_1}_\beta\{{\bf I-P}\}f,w_{\ell-|\beta|,\kappa}^2\partial^\alpha_\beta\{{\bf I-P}\}f\right),
\end{equation}
and
\begin{equation}\label{difficulty-2}
I_{|\alpha|,\ell-|\beta|,\kappa}^B\equiv\sum_{|\alpha_1|\geq1}\left(\left(v\times\partial^{\alpha_1}B\right)\cdot \nabla_v\partial^{\alpha-\alpha_1}_\beta\{{\bf I-P}\}f,w_{\ell-|\beta|,\kappa}^2\partial^\alpha_\beta\{{\bf I-P}\}f\right).
\end{equation}
Here $(\cdot,\cdot)$ denotes the standard $L^2(\mathbb{R}^3_v\times\mathbb{R}^3_x)\times L^2(\mathbb{R}^3_v\times\mathbb{R}^3_x)$ inner product in $\mathbb{R}^3_v\times\mathbb{R}^3_x$.
\end{itemize}
To deduce the desired estimates on the above terms, the main ingredients used in \cite{Duan_Liu-Yang_Zhao-VMB-2013, Duan_Yang_Zhao-MMMA-2012}  can be summarized as follows:
\begin{itemize}
\item A time-velocity weighted energy method is introduced basing on the weight function $\widetilde{w}_{\ell-|\beta|}(t,v)=w_{\ell-|\beta|,-\gamma}(t,v)$. An advantage of this weight function is that the term $I^{lt}_{|\alpha|,\ell-|\beta|,-\gamma}$ related to the linear transport term $v\cdot\nabla_x f$ can be controlled suitably. In fact,
$$
\widetilde{w}_{\ell-|\beta|}^2=\widetilde{w}_{\ell-|\beta|}\times\widetilde{w}_{\ell-|\beta-e_i|}
\times \langle v\rangle^\gamma,
$$
then
\begin{eqnarray*}
\left|I^{lt}_{|\alpha|,\ell-|\beta|,-\gamma}\right|&\leq& \varepsilon \left\|\widetilde{w}_{\ell-|\beta|}\partial^\alpha_\beta\{{\bf I}-{\bf P}\}f\right\|^2_\nu+
C(\varepsilon)\left\|\widetilde{w}_{\ell-|\beta-e_i|}\partial^{\alpha+e_i}_{\beta-e_i}\{{\bf I}-{\bf P}\}f\right\|^2_\nu,
\end{eqnarray*}
that, by a suitable linear combination with respect to $|\alpha|$,  can be further controlled by the dissipation
$$
\sum\limits_{|\alpha|+|\beta|\leq N} D^L_{|\alpha|,\ell-|\beta|,-\gamma}
$$
induced by the linearized Boltzmann collision operator $L$.

On the other hand, this approach leads to an additional difficulty on estimating the nonlinear term $(E+v\times B)\cdot\nabla_vf_\pm$ that
requires a restriction on the range of the parameter $\gamma$. In fact, to control the term $I^B_{|\alpha|,\ell-|\beta|,-\gamma}$, by
$$
\widetilde{w}^2_{\ell-|\beta|}(t,v)=\widetilde{w}_{\ell-|\beta|}(t,v)\times \widetilde{w}_{\ell-|\beta|-1}(t,v) \times\langle v\rangle^{-\gamma},
$$
we can have
\begin{eqnarray}\label{restriction-gamma}
&&\left|I^B_{\alpha|,\ell-|\beta|,-\gamma}\right|\\
&\leq&
\sum\limits_{0<\alpha_1\leq \alpha}\int_{\mathbb{R}^3_x\times\mathbb{R}^3_v}\langle v\rangle ^{1-\gamma}\left|\partial^{\alpha_1}B\right|\left|\widetilde{w}_{\ell-|\beta|}\partial^\alpha_\beta\{{\bf I}-{\bf P}\}f_\pm\right|\left|\widetilde{w}_{\ell-|\beta|-1}\nabla_v\partial^{\alpha-\alpha_1}_\beta\{{\bf I}-{\bf P}\}f_\pm\right| dvdx,\nonumber
\end{eqnarray}
which can be  controlled by the dissipation
$$
\sum\limits_{|\alpha|+|\beta|\leq N} D^W_{|\alpha|,\ell-|\beta|,-\gamma}
$$
induced by the exponential factor of the weight function $w_{\ell-|\beta|,-\gamma}(t,v)$
{\it only when}
$$
1-\gamma\leq 2, \quad i.e.\quad \gamma\geq -1,
$$
and $\partial^{\alpha_1}B(t,x)$ decays sufficiently fast.

\end{itemize}

Thus, up to now, the existing approaches for the construction of global classical solutions to the Vlasov-Maxwell-Boltzmann system \eqref{VMB}, \eqref{Maxwell} near Maxwellians is limited to the case when $-1\leq \gamma\leq 1$. And
the purpose of this paper is to introduce a new
approach for the  whole range soft potential, that is, to include the case when
$-3< \gamma < -1$.


To continue, we first introduce some notations used throughout the paper.
\begin{itemize}
\item $C$ and $O(1)$ denote some positive constants (generally large) and $\kappa$, $\delta$ and $\lambda$ are used to denote
some positive constants (generally small), where  $C$, $O(1)$, $\kappa$, $\delta$, and $\lambda$ may take different values in different places;
\item $A\lesssim B$ means that there is a generic constant $C> 0$ such that $A \leq   CB$. $A \sim B$ means $A\lesssim B$ and $B\lesssim A$;
\item The multi-indices $ \alpha= [\alpha_1,\alpha_2, \alpha_3]$ and $\beta = [\beta_1, \beta_2, \beta_3]$ will be used to record spatial and velocity derivatives, respectively. And $\partial^{\alpha}_{\beta}=\partial^{\alpha_1}_{x_1}\partial^{\alpha_2}_{x_2}\partial^{\alpha_3}_{x_3} \partial^{\beta_1}_{ v_1}\partial^{\beta_2}_{ v_2}\partial^{\beta_3}_{ v_3}$. Similarly, the notation $\partial^{\alpha}$ will be used when $\beta=0$ and likewise for $\partial_{\beta}$. The length of $\alpha$ is denoted by $|\alpha|=\alpha_1 +\alpha_2 +\alpha_3$. $\alpha'\leq  \alpha$ means that no component of $\alpha'$ is greater than the corresponding component of $\alpha$, and $\alpha'<\alpha$ means that $\alpha'\leq  \alpha$ and $|\alpha'|<|\alpha|$. And it is convenient to write $\nabla_x^k=\partial^{\alpha}$ with $|\alpha|=k$;
\item $\langle\cdot,\cdot\rangle$ is used to denote the ${L^2_{ v}}\times L^2_{ v}$ inner product in ${\mathbb{ R}}^3_{ v}$, with the ${L^2}$ norm $|\cdot|_{L^2}$. For notational simplicity, $(\cdot, \cdot)$ denotes the ${L^2}\times L^2$ inner product either in ${\mathbb{ R}}^3_{x}\times{\mathbb{ R}}^3_{ v }$ or in ${\mathbb{ R}}^3_{x}$ with the ${L^2}\times L^2$ norm $\|\cdot\|$;
\item $\chi_{\Omega}$ is the standard indicator function of the set $\Omega$;
\item $\|f(t,\cdot,\cdot)\|_{L^p_xL^q_v}=\left({\displaystyle\int_{\mathbb{R}^3_x}}
    \left({\displaystyle\int_{\mathbb{R}^3_v}}|f(t,x,v)|^qdv\right)^{\frac pq}dx\right)^{\frac 1p}$, and others
like $\|f(t,\cdot,\cdot)\|_{L^p_xH^q_v}$ can be defined similarly;
\item $B_C \subset \mathbb{R}^3$ denotes the ball of radius $C$ centered at the origin, and $L^2(B_C)\times L^2(B_C)$ stands for the space $L^2\times L^2$ over $B_C$ and likewise for other spaces. Recall that $\nu(v)\sim(1+|v|^2)^{\frac{\gamma}{2}}$, we set
    $|f|_{\nu}^2\equiv\int_{\mathbb{R}^3}|f|^2\nu(v) dv$ and for each $l\in \mathbb{R}$ , $L_l^2(\mathbb{R}_v^3)\times L_l^2(\mathbb{R}_v^3)$ denotes the weighted
function space with norm
$$
|f|_{L^2_{l}}^2\equiv\int_{\mathbb{R}^3_v}|f(v)|^2\langle v\rangle^{2l}dv,\quad \langle v\rangle=\sqrt{1+| v|^2}.
$$
$H^k_l(\mathbb{R}_v^3)\times H^k_l(\mathbb{R}_v^3)$ with the norm $|f|_{H^k_l}$ etc. can be defined similarly;
\item For $s\in\mathbb{R}$,
\[
\left(\Lambda^sg\right)(t,x,v)=\int_{\mathbb{R}^3}|\xi|^{s}\hat{g}(t,\xi,v)e^{2\pi ix\cdot\xi}d\xi
=\int_{\mathbb{R}^3}|\xi|^{s}\mathcal{F}[g](t,\xi,v)e^{2\pi ix\cdot\xi}d\xi
\]
with $\hat{g}(t,\xi,v)\equiv\mathcal{F}[g](t,\xi,v)$ being the Fourier transform of $g(t,x,v)$ with respect to $x$.  The homogeneous Sobolev space $\dot{H}^s\times \dot{H}^s$ is the Banach space consisting of all $g$ satisfying  $\|g\|_{\dot{H}^s}<+\infty$, where
\[
\|g(t)\|_{\dot{H}^s}\equiv\left\|\left(\Lambda^s g\right)(t,x,v)\right\|_{L^2_{x,v}}=\left\||\xi|^s\hat{g}(t,\xi,v)\right\|_{L^2_{\xi,v}}.
\]
\end{itemize}

For an integer $N\geq0$ and $\ell\in \mathbb{R}$, the parameter $\vartheta$ is suitably chosen so that \begin{equation}\label{range-vartheta}
 \left\{
 \begin{array}{ll}
0<\vartheta\leq \min\left\{\frac{\gamma-2\varrho\gamma+4\varrho+2}{4-4\gamma},
\frac{2\varrho\gamma-3\gamma-4\varrho-6}{8\gamma-4}\right\}, \quad & when\ \  \varrho\in[\frac12,\frac32)\  and\  N_0\geq 5,\\[2mm]
 0<\vartheta\leq \min\left\{\frac{\gamma-2\varrho\gamma+4\varrho+2}{4-4\gamma},
\frac{\varrho\gamma-2\gamma-2\varrho-2}{4\gamma-2}\right\},\quad & when \ \ \varrho\in(1,\frac32)\  and\ N_0=4.
 \end{array}
 \right.
\end{equation}
Define the energy functional $\mathcal{\overline{E}}_{N,\ell,\kappa}(t)$ and the corresponding energy dissipation rate functional $\mathcal{\overline{D}}_{N,\ell,\kappa}(t)$ of a given function $f(t,x, v)$ with respect to the weight function $w_{\ell-|\beta|,\kappa}(t,v)$ defined by \eqref{our-weight} as follows:
\begin{equation*}
\mathcal{\overline{E}}_{N,\ell,\kappa}(t)\sim{\mathcal{E}}_{N,\ell,\kappa}(t)+\left\|\Lambda^{-\varrho}(f,E,B)\right\|^2,
\end{equation*}
and
\begin{equation*}
\mathcal{\overline{D}}_{N,\ell,\kappa}(t)\sim{\mathcal{D}}_{N,\ell,\kappa}(t)
+\left\|\Lambda^{1-\varrho}(a,b,c,E,B)\right\|^2+\left\|\Lambda^{-\varrho}
(a_+-a_-,E)\right\|^2+\left\|\Lambda^{-\varrho}{\bf\{I-P\}}f\right\|^2_\nu,
\end{equation*}
respectively. Here
\begin{equation*}
{\mathcal{E}}_{N,\ell,\kappa}(t)\sim\sum_{|\alpha|+|\beta|\leq N}\left\|w_{\ell-|\beta|,\kappa}
\partial^{\alpha}_{\beta}f\right\|^2+\|(E,B)\|_{H^{N}_x}^2,
\end{equation*}
and
\begin{equation*}
\begin{split}
{\mathcal{D}}_{N,\ell,\kappa}(t)\sim&\sum_{1\leq|\alpha|\leq N}\left\|
\partial^{\alpha}(a_{\pm},b,c)\right\|^2+\sum_{|\alpha|+|\beta|\leq N}\left\|w_{\ell-|\beta|,\kappa}
\partial^{\alpha}_{\beta}{\bf\{I-P\}}f\right\|^2_{\nu}
+\left\|a_+-a_-\right\|^2\\{}
&+\|E\|_{H^{N-1}_x}^2+\left\|\nabla_x B\right\|_{H^{N-2}_x}^2
+(1+t)^{-1-\vartheta}\sum_{|\alpha|+|\beta|\leq N}\left\| w_{\ell-|\beta|,\kappa}
\partial^{\alpha}_{\beta}{\bf\{I-P\}}f\langle v\rangle \right\|^2.
\end{split}
\end{equation*}
Moreover, we also need to define $\mathcal{E}_{N}(t)$, the energy functional without weight, $\mathcal{E}_{N_0}^{k}(t)$, the high order energy functional without weight, and $\mathcal{E}^k_{N_0,\ell,\kappa}(t)$, the high order energy functional with respect to the weight function $w_{\ell-|\beta|,\kappa}(t,v)$, as follows:
\begin{equation*}
\mathcal{E}_{N}(t)\sim\sum_{k=0}^{N}\left\|\nabla^k(f,E,B)\right\|^2,
\end{equation*}
\begin{equation*}
\mathcal{E}_{N_0}^{k}(t)\sim\sum_{|\alpha|=k}^{N_0}\left\|\partial^\alpha(f,E,B)\right\|^2,
\end{equation*}
and
\begin{equation*}
\mathcal{E}^k_{N_0,\ell,\kappa}(t)\sim\sum_{|\alpha|+|\beta\leq N_0,\atop{|\alpha|\geq k}}\left\|w_{\ell-|\beta|,\kappa}\partial^\alpha_\beta f\right\|^2+\sum_{|\alpha|=k}^{N_0}\left\|\partial^\alpha(E,B)\right\|^2,
\end{equation*}
respectively. The corresponding energy dissipation rate functionals $\mathcal{D}_{N}(t)$, $\mathcal{D}_{N_0}^{k}(t)$, and $\mathcal{D}_{N_0,\ell,\kappa}^{k}(t)$ are given by
\begin{equation*}
\begin{split}
\mathcal{D}_{N}(t)\sim&
\left\|(E,a_+-a_-)\right\|^2+\sum_{1\leq|\alpha|\leq N-1}\left\|
\partial^\alpha({\bf P}f,E,B)\right\|^2
+\sum_{|\alpha|=N}\left\|\partial^\alpha {\bf P}f\right\|^2+\sum_{ |\alpha| \leq N}\left\|\partial^\alpha{\bf\{I-P\}}f\right\|^2_{\nu},
\end{split}
\end{equation*}
\begin{equation*}
\begin{split}
\mathcal{D}_{N_0}^{k}(t)\sim&\left\|\nabla^{k}(E,a_+-a_-)\right\|^2+\sum_{k+1\leq|\alpha|\leq N_0-1}
\left\|\partial^\alpha({\bf P}f,E,B)\right\|^2+\sum_{|\alpha|=N_0}\left\|\partial^\alpha {\bf P}f\right\|^2+\sum_{k\leq |\alpha|\leq N_0}\left\|\partial^\alpha{\bf\{I-P\}}f\right\|^2_{\nu},
\end{split}
\end{equation*}
and
\begin{equation*}
\begin{split}
\mathcal{D}_{N_0,\ell,\kappa}^{k}(t)\sim&\left\|\nabla^{k}(E,a_+-a_-)\right\|^2+\sum_{k+1\leq |\alpha|\leq N_0-1}
\left\|\partial^\alpha({\bf P}f,E,B)\right\|^2+\sum_{|\alpha|+|\beta\leq N_0,\atop{|\alpha|\geq k}}\left\|w_{\ell-|\beta|,\kappa}\partial^\alpha_\beta{\bf\{I-P\}}f\right\|^2_{\nu}\\{}
&+\sum_{|\alpha|=N_0}\left\|\partial^\alpha {\bf P}f\right\|^2+(1+t)^{-1-\vartheta}\sum_{|\alpha|+|\beta\leq N_0,\atop{|\alpha|\geq k}}\left\| w_{\ell-|\beta|,\kappa}
\partial^{\alpha}_{\beta}{\bf\{I-P\}}f\langle v\rangle \right\|^2,
\end{split}
\end{equation*}
respectively.

\subsection{Main results and ideas}

With the above preparation, the precise statement  concerning the global in time
 solvability of the Cauchy problem (\ref{f}), (\ref{f-initial}), \eqref{compatibility conditions} can be stated as follows.

\begin{theorem}\label{Th1.1}
Suppose that
\begin{itemize}
\item[(i)] $F_0(x,v)=\mu+\sqrt{\mu}f_0(x,v)\geq0$, $\frac 12\leq \varrho<\frac 32$, $-3<\gamma<-1$. Let
\begin{equation}\label{N-N_0}
 \left\{
 \begin{array}{rl}
 N_0\geq 5, N=2N_0-1, \quad & when\ \  \varrho\in[\frac12,1],\\[2mm]
 N_0\geq 4, N=2N_0,\quad & when \ \ \varrho\in(1,\frac32);
 \end{array}
 \right.
\end{equation}
\item[(ii)] The parameter $\vartheta$ is chosen to satisfy \eqref{range-vartheta} and we take {$\sigma_{N,0}=\frac{1+\epsilon_0}{2}$, $\sigma_{n,0}=0$ with $n\leq N-1$,  $\sigma_{n,j}-\sigma_{n,j-1}=\frac{2(1+\gamma)}{\gamma-2}(1+\vartheta)$ when $0\leq j\leq n$ and $1\leq n\leq N$;}
\item[(iii)] There exists a positive constants $\widetilde{l}$ which depends only on $\gamma$ and $N_0$ such that
\begin{itemize}
\item[(a)] $\widetilde{l}_1\geq \frac\gamma2+\frac{(1-2\gamma)\sigma_{N_0,N_0}}{2+\varrho},$ $\widetilde{\ell}_2\geq\frac\gamma2+\frac{2(1-2\gamma)\sigma_{N,N}}{3+2\varrho}$, and $\widetilde{\ell}_3\geq\frac\gamma2+\frac{(1-2\gamma)\sigma_{ N-1, N-1}}{2+\varrho}$,
\item[(b)] $l_1\geq N$, $l_1^*\geq\max\left\{\widetilde{\ell}_2-\frac\gamma2,\  \widetilde{\ell}_3-\frac\gamma2-\gamma l_1\right\}$, {$l_0\geq \ {l_1^*}+\frac{5}{2}$},
$l_0^*\geq \widetilde{\ell}_1-\frac\gamma2-\gamma(l_0+l^*)$ with $l^*=\frac32-\frac{\widetilde{l}}\gamma$.
\end{itemize}
\end{itemize}
If we assume further that
\begin{equation*}
Y_0=\sum_{|\alpha|+|\beta|\leq N_0}\left\|w_{l_0^*-|\beta|,1}\partial^\alpha_\beta f_0\right\| +\sum_{N_0+1\leq|\alpha|+|\beta|\leq N}\left\|w_{l_1^*-|\beta|,1}\partial^\alpha_\beta f_0\right\|
+\|(E_0,B_0)\|_{H^N\bigcap \dot{H}^{-\varrho}}+\|f_0\|_{\dot{H}^{-\varrho}}
\end{equation*}
is sufficiently small, then the Cauchy problem (\ref{f}), (\ref{f-initial}), \eqref{compatibility conditions} admits a unique global solution $[f(t,x,v),$ $E(t,x),$ $ B(t,x)]$ satisfying $F(t,x,v)=\mu+\sqrt{\mu}f(t,x,v)\geq0$.
\end{theorem}

\begin{remark} Several remarks concerning Theorem \ref{Th1.1} are given.
\begin{itemize}
\item {As mentioned before, although only the case of $-3<\gamma<-1$ is studied in this paper, the case of $-1\leq\gamma\leq 1$ is much simpler and similar result holds. Thus, this
work together with \cite{Duan_Liu-Yang_Zhao-VMB-2013} provide a satisfactory well-posedness theory for the Cauchy problem of the two-species Vlasov-Maxwell-Boltzmann system (\ref{f}), (\ref{f-initial}), \eqref{compatibility conditions} in the perturbative framework for both cutoff and non-cutoff intermolecular interactions.}
\item Since in the proof of Lemma \ref{lemma4.3}, $N$ is assumed to satisfy $N>\frac53 N_0-\frac53$, while in the proof of Lemma \ref{Lemma1}, $N$ is further required to satisfy $N\geq 2N_0-2+\varrho$. Putting these assumptions together, we can take $N=2N_0-1$ for $\varrho\in[\frac12,1]$ and $N=2N_0$ for $\varrho\in(1,\frac32)$.
\item The minimal regularity index, i.e., the lower bound on the parameter $N$, we imposed on the initial data is $N=9$, $N_0=5$ for $\varrho\in[\frac12,1]$ and $N=8$, $N_0=4$ for $\varrho\in(1,\frac32)$.
\item The precise value of the parameter $\widetilde{l}$ will be specified in the proof of Lemma \ref{lemma4.3}.
\end{itemize}
\end{remark}

Note that Theorem \ref{thm.mr} is an immediate consequence of Theorem  \ref{Th1.1}. The next result is concerned with the temporal decay estimates on the global solution $[f(t,x,v),$ $E(t,x),$ $B(t,x)]$ obtained in Theorem \ref{Th1.1}.

\begin{theorem}\label{Th1.2}
Under the assumptions of Theorem \ref{Th1.1},
we have
\begin{itemize}
\item[(1)]
Taking $k=0,1,2,\cdots, N_0-2$, it follows that
\begin{equation}\label{TH2-1}
\mathcal{E}^k_{N_0}(t)\lesssim Y^2_0(1+t)^{-(\varrho+k)}.
\end{equation}
\item[(2)]Let $0\leq i\leq k\leq N_0-3$ be an integer. Take {$l_{0,k}\geq N_0$ with $l_{0,k-1}\geq l_{0,k}+3$ for $2\leq k\leq N_0-3$. Further take $l_0$ and $l^*$ respectively as $l_0= l_{0,0}=l_{0,1}\geq\max\left\{\chi_{k\geq2}(l_{0,k}+3k-3), {l_1^*}+\frac{5}{2}\right\}$
and $l^*=\frac{k+2}{2}-\frac{\widetilde{l}}\gamma$ in Theorem 1.1.  Then it follows that
\begin{equation}\label{TH2-2}
 \begin{split}
\mathcal{E}^k_{N_0,l_{0,k}+\frac{i}{2},-\gamma}(t)
\lesssim Y^2_0(1+t)^{-k-\varrho+i}, \quad \quad i=0,1,\cdots,k+[\varrho].
\end{split}
\end{equation}}
Here and in the sequel $[\varrho]$ denotes the greatest integer less than $\varrho$.
\item[(3)]When $N_0+1\leq|\alpha|\leq N-1$,
\begin{equation}\label{TH2-4}
\begin{split}
\|\partial^\alpha f\|^2
\lesssim Y^2_0(1+t)^{-\frac{(N-|\alpha|)(N_0-2+\varrho)}{N-N_0}}.
\end{split}
\end{equation}
\end{itemize}
\end{theorem}

\begin{remark}
In Theorem \ref{Th1.2}, we notice that the highest index k of $\mathcal{E}^k_{N_0}(t)$ is $N_0-2$ while  the highest index of  $\mathcal{E}^k_{N_0,\ell,-\gamma}(t)$ is $N_0-3$. The reason is that the highest order $\|\partial^\alpha E\|^2$ appearing in (\ref{lemma2-1}) does not belong to the corresponding dissipation rate $\mathcal{D}^k_{N_0,\ell,-\gamma}(t)$.
\end{remark}


Now we present the main ideas in the proof.
To overcome the difficulties pointed out
before for the case when $-3<\gamma<-1$, the  main observation
is that two sets of time-velocity weighted energy estimates should be
 performed simultaneously as explained in the following.

\begin{itemize}
\item [(i).]
First of all, when estimating $I^B_{|\alpha|,\ell-|\beta|,\kappa}$ defined by \eqref{difficulty-2} for $\kappa=-\gamma$,
there are some error terms with higher weight
when
$-3<\gamma<-1$, cf. \eqref{restriction-gamma} that can not be
controlled. However, as long as the solution $[f(t,x,v), E(t,x), B(t,x)]$ constructed  up to  $t=T>0$ satisfies the a priori assumption
\begin{eqnarray}\label{a-priori-assumption}
X(t)&=&\sup_{0\leq s\leq t}\left\{\mathcal{E}_N(s)+\mathcal{\overline{E}}_{N_0,l_0+l^*,-\gamma}(s)
+\mathcal{E}_{ N-1,l_1,-\gamma}(s)\right\}\nonumber\\
&&+\sup_{0\leq s\leq t}\left\{\sum_{N_0+1\leq n\leq N}\sum_{|\alpha|+|\beta|=n,\atop|\beta|=j,1\leq j\leq n}(1+s)^{-\sigma_{n,j}}\left\|w_{l_1^*-j,1}\partial_\beta^\alpha\{{\bf I-P}\} f\right\|^2\right.\nonumber\\
&&+\sum_{N_0+1\leq n\leq N-1}\sum_{|\alpha|=n}\left\|w_{l_1^*,1}\partial^\alpha f\right\|^2+\sum_{|\alpha|=N}(1+s)^{-\frac{1+\epsilon_0}{2}}\left\|w_{l_1^*,1}\partial^\alpha f\right\|^2\\
&&+\sum_{1\leq n\leq N_0}\sum_{|\alpha|+|\beta|=n,\atop|\beta|=j,1\leq j\leq n}(1+s)^{-\sigma_{n,j}}\left\|w_{l_0^*-j,1}\partial_\beta^\alpha\{{\bf I-P}\} f\right\|^2\nonumber\\
&&\left.+\sum_{1\leq n\leq N_0}\sum_{|\alpha|=n}\left\|w_{l_0^*,1}\partial^\alpha f\right\|^2+\left\|w_{l_0^*,1}\{{\bf I-P}\} f\right\|^2\right\}
\leq M,\nonumber
\end{eqnarray}
where $M>0$ is sufficiently small, then one can obtain
\begin{equation*}
\frac{d}{dt}\mathcal{\overline{E}}_{N_0,l_0+l^*,-\gamma}(t)+\mathcal{\overline{D}}_{N_0,l_0+l^*,-\gamma}(t)\lesssim
\left\|\nabla^2(E,B)\right\|_{H^{ N_0-2}_x}^{\frac1{\theta_1}}\mathcal{\widetilde{D}}_{N_0,l_0^*,1}(t)+\sum_{|\alpha|=N_0}\varepsilon\|\partial^\alpha E\|^2,
\end{equation*}
\begin{equation*}
\frac{d}{dt}\mathcal{E}_{N}(t)+\mathcal{D}_{N}(t)
\lesssim \left(\| E\|_{L^{\infty}_x}+\left\|\nabla^2(E,B)\right\|_{H^{ N_0-2}_x}\right)^{\frac1{\theta_2}}\mathcal{\widetilde{D}}_{N,l_1^*,1}(t)
+\mathcal{E}_{N}(t)\mathcal{E}^1_{N_0,l_0,-\gamma}(t),
\end{equation*}
and
\begin{equation*}
\begin{aligned}
&\frac{d}{dt}\mathcal{E}_{ N-1,l_1,-\gamma}(t)+\mathcal{D}_{ N-1,l_1,-\gamma}(t)\\{}
\lesssim& \left\|\nabla^2(E,B)\right\|_{H^{ N_0-2}_x}^{\frac1{\theta_3}}\mathcal{\widetilde{D}}_{ N-1,l_1^*,1}(t)
+\mathcal{E}_{N}(t)\mathcal{E}^1_{N_0,l_0,-\gamma}(t)+\sum_{|\alpha|= N-1}\left\|\partial^\alpha E\right\|\left\|\mu^\delta\partial^\alpha f\right\|,
\end{aligned}
\end{equation*}
where $\mathcal{\widetilde{D}}_{N_0,l_0^*,1}(t)$, $\mathcal{\widetilde{D}}_{ N-1,l_1^*,1}(t),$ and $\mathcal{\widetilde{D}}_{ N,l_1^*,1}(t)$ are defined by \eqref{D_{N_0,l^*_0,1}} and \eqref{D_{m,l^*_1,1}} respectively.

Notice that $\theta_i\, (i=1,2,3)$ can be chosen sufficiently small as long as $l^*_j(j=0,1)$ is taken sufficiently large. Thus, one deduce some  uniform-in-time estimates based on the above three differential inequalities provided that
\begin{itemize}
\item [(i1).] The electromagnetic field $[E(t,x), B(t,x)]$ has certain
temporal decay estimate and $\mathcal{E}^1_{N_0,l_0,-\gamma}(t)\in L^1(\mathbb{R}^+)$;
\item [(i2).] There are some upper bound estimates on $\mathcal{\widetilde{D}}_{N_0,l_0^*,1}(t)$, $\mathcal{\widetilde{D}}_{ N-1,l_1^*,1}(t),$ and $\mathcal{\widetilde{D}}_{ N,l_1^*,1}(t)$. For example, even if we can not deduce uniform-in-time bounds on $\mathcal{\widetilde{D}}_{N_0,l_0^*,1}(t)$, $\mathcal{\widetilde{D}}_{ N-1,l_1^*,1}(t),$ and $\mathcal{\widetilde{D}}_{ N_,l_1^*,1}(t)$, it suffices to show that the possible time increasing upper bounds on $\mathcal{\widetilde{D}}_{N_0,l_0^*,1}(t)$, $\mathcal{\widetilde{D}}_{ N-1,l_1^*,1}(t),$ and $\mathcal{\widetilde{D}}_{ N,l_1^*,1}(t)$ are   independent of the choices of the parameters $l^*_j$ $(j=0,1)$ but depend only on $N$ and $N_0$.
\end{itemize}
To achieve  (i1), first of all,  under the assumption of \eqref{a-priori-assumption} with $M>0$ sufficiently small, we can deduce that
\begin{equation*}
 \frac{d}{dt}\mathcal{E}_{N_0}^{k}(t)+\mathcal{D}_{N_0}^{k}(t)\leq 0,\quad k=0,1,\cdots, N_0-2
 \end{equation*}
 and
 \begin{equation*}
\begin{aligned}
&\frac{d}{dt}\mathcal{E}^k_{N_0,\ell,-\gamma}(t)+\mathcal{D}^k_{N_0,\ell,-\gamma}(t)
\lesssim \sum_{|\alpha|=N_0}\|\partial^\alpha E\|^2,
&\quad k=0,1
\end{aligned}
\end{equation*}
hold for any $0\leq t\leq T$.

From these two differential inequalities, by using the interpolation technique as in \cite{Guo-CPDE-12, Wang-12}, we can deduce a temporal decay rate of $\mathcal{E}^k_{N_0}(t)$, from which one can further obtain the temporal decay rates of $\mathcal{E}^k_{N_0,l_0,-\gamma}(t)$ with $\mathcal{E}^1_{N_0,l_0,-\gamma}(t)\in L^1({\mathbb{R}}^+)$.
\item [(ii).] To deduce the estimates stated in (i2),
we need the second set of time-velocity
weighted energy estimates  with  the weight function $w_{\ell-|\beta|,1}(t,v)$ for some $\ell$ that is sufficiently large. In this case, since
\begin{eqnarray*}
&& w^2_{\ell-|\beta|,1}(t,v)=w_{\ell-|\beta|,1}(t,v)\times w_{\ell-|\beta|-1,1}(t,v)\times \langle v\rangle,\\
&& w^2_{\ell-|\beta|,1}(t,v)=w_{\ell-|\beta|,1}(t,v)\times w_{\ell-|\beta|+1,1}(t,v)\times \langle v\rangle^{-1},
\end{eqnarray*}
we can deduce that for all $-3<\gamma<-1$, the terms \eqref{difficulty-1} and \eqref{difficulty-2} can be controlled by the extra dissipative term \eqref{dissipative-weight} provided that the electromagnetic field $[E(t,x), B(t,x)]$ has certain temporal decay estimates. On the other hand,  the term \eqref{I^{lt}} related to the linear transport term $v\cdot\nabla_xf$  can only be bounded as
\begin{eqnarray*}
I^{lt}_{|\alpha|,\ell-|\beta|,1}
&\lesssim& \eta\left\|w_{\ell-|\beta|, 1}\partial^{\alpha}_\beta{\bf\{I-P\}}f\right\|_\nu^2+C_\eta\left\|w_{\ell-|\beta-e_i|, 1}\partial^{\alpha+e_i}_{\beta-e_i}{\bf\{I-P\}}f\langle v\rangle^{-\frac\gamma2-1}\right\|^2.
\end{eqnarray*}
Hence, it leads to how to control
\begin{equation}\label{diffi-linear-term}
\left\|w_{\ell-|\beta-e_i|, 1}\partial^{\alpha+e_i}_{\beta-e_i}{\bf\{I-P\}}f\langle v\rangle^{-\frac\gamma2-1}\right\|^2.
\end{equation}

For \eqref{diffi-linear-term}, observe that
\begin{itemize}
\item Since $\frac\gamma 2<-\frac \gamma 2-1<2$ holds for all $-3<\gamma<-1$, it does not lead to the increase of the weight if we neglect the fact $(1+t)^{-1-\vartheta}$ in the extra dissipative term $D^W_{|\alpha|,\ell-|\beta|,1}$ given by \eqref{dissipative-weight};
\item The order of the derivative with respect to $x$ increases by one in \eqref{diffi-linear-term} so that the corresponding temporal decay rate in $L^2-$norm increases $\frac 12$, cf.  \cite{Duan_Liu-Yang_Zhao-VMB-2013, DS-CPAM-11}.
\end{itemize}

Therefore, motivated in \cite{Hosono-Kawashima-2006} for deducing
 the temporal decay estimates on solutions to some nonlinear equations of regularity-loss type, we set different time increase rate $\sigma_{n,j}$ for
$$
\displaystyle\sum_{|\alpha|+|\beta|=n, |\beta|=j}\left\|w_{l_1^*-j,1}\partial_\beta^\alpha\{{\bf I-P}\} f\right\|^2,
$$
where
\begin{equation*}
\sigma_{n,j}-\sigma_{n,j-1}=\frac{2(1+\gamma)}{\gamma-2}(1+\vartheta).
\end{equation*}
Thus, one can deduce that
 \begin{eqnarray*}
&&\sum_{|\alpha|+|\beta|=n,\atop|\beta|=j,1\leq j\leq n}(1+t)^{-\sigma_{n,j}}
\left\|w_{\ell-j+1,1}\partial^{\alpha+e_i}_{\beta-e_i} \{{\bf I-P}\} f\langle v\rangle^{-\frac\gamma2-1}\right\|
^2\\ \nonumber
&\lesssim&\sum_{|\alpha|+|\beta|=n,\atop|\beta|=j,1\leq j\leq N_0}\left\{(1+t)^{-\sigma_{n,j-1}-1-\vartheta}
\left\|w_{\ell-j+1,1}\partial^{\alpha+e_i}_{\beta-e_i} \{{\bf I-P}\} f\langle v\rangle\right\|^{2}\right.\\ \nonumber
&&\left.+(1+t)^{-\sigma_{n,j-1}}\left\|w_{\ell-j+1,1}\partial^{\alpha+e_i}_{\beta-e_i} \{{\bf I-P}\} f\right\|^2_\nu\right\}.
\end{eqnarray*}
\end{itemize}

Once the above argument is substantiated,
 we can then close the a priori assumption \eqref{a-priori-assumption} and the global solvability result follows. And this will be given in detail  in the following sections.


\section{Proofs of the main results}
The proofs of Theorem \ref{Th1.1} and Theorem \ref{Th1.2}
will be given in this section. To illustrate the main ideas of the proof clearly and to make the presentation easy to follow, we will just state some key estimates first and then use them to prove our main results. The complete proofs of these key estimates will be given in the next section. To simplify the presentation, we divide this section into a few parts.

\subsection{Preliminaries}

In this subsection, for later use we collect several basic estimates on the linearized Boltzmann collision operator $L$ and the nonlinear term $\Gamma$ for cutoff potentials, whose  one-species version can be found in \cite{Duan_Yang_Zhao-MMMA-2012, S-G-08}.

The first lemma concerns the coercivity estimate \eqref{coercive-estimates} on the linearized collision operators $L$ together with its weighted version with respect to the weight $w_{\ell,\kappa}(t,v)$ given by \eqref{our-weight}.

\begin{lemma}\label{Lemma L}
Let $-3<\gamma<0$, one has
\begin{equation}\label{L_0}
  \langle L f, f\rangle\geq |{\bf \{I-P\}}f|^2_\nu.
\end{equation}
Moreover, let $|\beta|>0$, for $\eta>0$ small enough and any $\ell\in\mathbb{R}, \kappa\geq0, 0<q\ll 1, \vartheta\in\mathbb{R}$, there exists $C_{\eta}>0$ such that
\begin{equation}\label{L_v}
  \left\langle w_{\ell,\kappa}^2\partial_{\beta}{ L}f,\partial_{\beta}f\right\rangle\geq \left|w_{\ell,\kappa}\partial_{\beta}f\right|^2_\nu
  -\eta\sum_{|\beta'|<|\beta|}\left|w_{\ell,\kappa}\partial_{\beta'}\{{\bf I-P}\}f\right|_\nu^2-C_{\eta}\left|\chi_{\{| v|\leq2C_{\eta}\}}f\right|^2
\end{equation}
holds.
\end{lemma}
\begin{proof} For the estimate \eqref{L_0}, the case for the hard sphere model has been proved in \cite{Guo-Invent-03}, while for general cutoff soft potentials, recall that $L$ can be decomposed as in \eqref{decomposition-L} with the collision frequency $\nu(v)$ and the nonlocal integral operator $K$ being defined by \eqref{collision-frequency} and \eqref{Operator-K} respectively, one can deduce by using the argument employed in Lemma 2 of \cite{Guo-ARMA-03} for one-species linearized Boltzmann collision operator with cutoff that the operator $K$ can be decomposed into a ``small part'' $K_s$ and a ``compact part'' $K_c$, therefore \eqref{L_0} follows by repeating the argument used in Lemma 3 of \cite{Guo-ARMA-03}.

As to \eqref{L_v}, it can be proved by a straightforward modification of the argument used in Lemma 2 of \cite{S-G-08}, we thus omit the details for brevity.
\end{proof}
The second lemma is concerned with the corresponding weighted estimates on the nonlinear term $\Gamma$.
For this purpose, similar to that of \cite{S-G-08}, we can get that
\begin{eqnarray}\label{gamma-0}
\partial^\alpha_\beta\Gamma_\pm(g_1,g_2)&\equiv&\sum C_\beta^{\beta_0\beta_1\beta_2}C_\alpha^{\alpha_1\alpha_2} {\Gamma}^0_\pm\left(\partial^{\alpha_1}_{\beta_1}g_1,\partial^{\alpha_2}_{\beta_2}g_2\right)\\ \nonumber
&\equiv& \sum C_\beta^{\beta_0\beta_1\beta_2}C_\alpha^{\alpha_1\alpha_2}\int_{\mathbb{R}^3\times\mathbb{S}^2}|v-u|^\gamma{\bf b}(\cos\theta)\partial_{\beta_0}[\mu(u)^\frac 12]\left\{\partial^{\alpha_1}_{\beta_1}g_{1\pm}(v')\partial^{\alpha_2}_{\beta_2}g_{2\pm}(u')\right.\\ \nonumber
&&\left.
+\partial^{\alpha_1}_{\beta_1}g_{1\pm}(v')\partial^{\alpha_2}_{\beta_2}g_{2\mp}(u')
-\partial^{\alpha_1}_{\beta_1}g_{1\pm}(v)\partial^{\alpha_2}_{\beta_2}g_{2\pm}(u)
-\partial^{\alpha_1}_{\beta_1}g_{1\pm}(v)\partial^{\alpha_2}_{\beta_2}g_{2\mp}(u)\right\}d\omega du,
\end{eqnarray}
where $g_i(t,x,v)=[g_{i+}(t,x,v), g_{i-}(t,x,v)]$ $(i=1,2)$ and the summations are taken for all $\beta_0+\beta_1+\beta_2=\beta, \alpha_1+\alpha_2=\alpha$.
From which one can deduce that

\begin{lemma}\label{lemma-nonlinear}
Assume $\kappa\geq 0, \ell\geq 0$.
Let $-3<\gamma<0$, $N\geq4$, $g_i=g_i(t,x,v)=[g_{i+}(t,x,v),g_{i-}(t,x,v)]\ (i=1,2,3)$, $\beta_0+\beta_1+\beta_2=\beta$ and $\alpha_1+\alpha_2=\alpha$, we have
the following results:
\begin{itemize}
\item[(i).] When $|\alpha_1|+|\beta_1|\leq N$, we have
\begin{equation}\label{nonlinear-1}
\begin{array}{rl}
\left\langle w_{\ell,\kappa}^2{\Gamma}^0_\pm\left(\partial^{\alpha_1}_{\beta_1}g_1,\partial^{\alpha_2}_{\beta_2}g_2\right), \partial^{\alpha}_{\beta}g_3\right\rangle
\lesssim\sum\limits_{m\leq2}\left\{
\left|\nabla^m_{v}\left\{\mu^\delta\partial^{\alpha_1}_{\beta_1}g_1\right\}\right|+\left|w_{\ell,\kappa}\partial^{\alpha_1}_{\beta_1}g_1\right|\right\}
\left|w_{\ell,\kappa}\partial^{\alpha_2}_{\beta_2}g_2\right|_{L^2_{\nu}}
\left|w_{\ell,\kappa}\partial^{\alpha}_{\beta}g_3\right|_{L^2_{\nu}}
\end{array}
\end{equation}
or
\begin{equation}\label{nonlinear-2}
\begin{array}{rl}
\left\langle w_{\ell,\kappa}^2{\Gamma}^0_\pm\left(\partial^{\alpha_1}_{\beta_1}g_1,\partial^{\alpha_2}_{\beta_2}g_2\right), \partial^{\alpha}_{\beta}g_3\right\rangle
\lesssim\sum\limits_{m\leq2}\left\{
\left|\nabla^m_{v}\left\{\mu^\delta\partial^{\alpha_2}_{\beta_2}g_2\right\}\right|+\left|w_{\ell,\kappa}\partial^{\alpha_2}_{\beta_2}g_2\right|\right\}
\left|w_{\ell,\kappa}\partial^{\alpha_1}_{\beta_1}g_1\right|_{L^2_{\nu}}
\left|w_{\ell,\kappa}\partial^{\alpha}_{\beta}g_3\right|_{L^2_{\nu}}.
\end{array}
\end{equation}
\item[(ii).]
Set $\varsigma(v)=\langle v\rangle^{-\gamma}\equiv \nu(v)^{-1},~l\geq0$, it holds that
\begin{equation}\label{n-3}
\begin{split}
\left|\varsigma^{l}{\Gamma}(g_1,g_2)\right|^2_{L^2_v}
&\lesssim\displaystyle\sum_{|\beta|\leq2}\left|\varsigma^{l-|\beta|}\partial_{\beta}g_1\right|^2_{L^2_{\nu}}
\left|\varsigma^{l}g_2\right|^2_{L^2_{\nu}},\\
\left|\varsigma^{l}{\Gamma}(g_1,g_2)\right|^2_{L^2_v}
&\lesssim\sum_{|\beta|\leq2}\left|\varsigma^{l}g_1\right|^2_{L^2_{\nu}}
\left|\varsigma^{l-|\beta|}\partial_{\beta}g_2\right|^2_{L^2_{\nu}}.
\end{split}
\end{equation}
\end{itemize}
\end{lemma}
\begin{proof}
Although the definition of ${\Gamma}^0_\pm(g_1,g_2)$ in \eqref{gamma-0} is a little different from ${\Gamma}^0(g_1,g_2)$ of \cite{S-G-08}, one can still deduce \eqref{nonlinear-1} and \eqref{nonlinear-2} by employing the similar argument used to yield the estimates stated in Lemma 3 of \cite{S-G-08}, we thus omit its proof for simplicity. As for \eqref{n-3}, it can also be proved by repeating the argument used in Lemma 2.4 of \cite{Xiao-Xiong-Zhao-JDE-2013}. This completes the proof of Lemma \ref{lemma-nonlinear}.
\end{proof}

In what follows, we will collect some analytic tools which will be used in this paper. The first one is on the Sobolev interpolation among the spatial regularity.
\begin{lemma}\label{lemma2.2}(cf. \cite{R. Adams, Guo-CPDE-12})
Let $2\leq p<\infty$ and $k,\ell, m\in\mathbb{R}$, then we have
\begin{equation*}
\left\|\nabla^k f\right\|_{L^p}\lesssim\left\|\nabla^\ell f\right\|^{\theta}_{L^2}\left\|\nabla^m f\right\|^{1-\theta}_{L^2}.
\end{equation*}
Here $0\leq \theta\leq1$ and $\ell$ satisfy
\begin{equation*}
\frac{1}{p}-\frac k3=\left(\frac12-\frac\ell3\right)\theta+\left(\frac12-\frac m3\right)(1-\theta).
\end{equation*}
Moreover, we have that
\begin{equation*}
\left\|\nabla^k f\right\|_{L^\infty}\lesssim\left\|\nabla^\ell f\right\|^{\theta}_{L^2}\left\|\nabla^m f\right\|^{1-\theta}_{L^2},
\end{equation*}
where $0\leq \theta\leq1$ and $\ell$ satisfy
\begin{equation*}
-\frac k3=\left(\frac12-\frac\ell3\right)\theta+\left(\frac12-\frac m3\right)(1-\theta),\quad \ell\leq k+1, \quad m\geq k+2.
\end{equation*}
\end{lemma}
The second one is concerned with the $L^p-L^q$ estimate on the operator $\Lambda^{-\varrho}$.
\begin{lemma}\label{lemma2.3}
Let $0<\varrho<3$, $1<p<q<\infty$, $\frac1q+\frac \varrho3=\frac1p$, then
\begin{equation*}
\|\Lambda^{-\varrho}f\|_{L^q}\lesssim\|f\|_{L^p}.
\end{equation*}
\end{lemma}

\subsection{Some a priori estimates}
In this subsection, we will deduce
  some a priori estimates on the solutions $[f(t,x,v), E(t,x), B(t,x)]$ to the Cauchy problem (\ref{f}) and (\ref{f-initial}) under some additional assumptions imposed on $[f(t,x,v), E(t,x), B(t,x)]$. For this purpose, we suppose that the Cauchy problem (\ref{f}) and (\ref{f-initial}) admits a unique local solution $[f(t,x,v), E(t,x), B(t,x)]$ defined on the time interval $ 0\leq t\leq T$ for some $0<T<\infty$. We now turn to deduce certain a priori estimates on $[f(t,x,v), E(t,x), B(t,x)]$. The first result is concerned with the temporal decay estimates on the energy functional $\mathcal{E}^k_{N_0}(t)$ for $k=0,1,2,\cdots, N_0-2$:
\begin{lemma}\label{Lemma1}
Let $N_0$ and $N$ satisfy \eqref{N-N_0}, $n\geq \frac23 N_0-\frac{5}3$, and take $k=0,1,2,\cdots, N_0-2$, then one has
\begin{equation}\label{Lemma1-1}
\frac{d}{dt}\mathcal{E}^k_{N_0}(t)+\mathcal{D}^k_{N_0}(t)\leq 0,\quad 0\leq t\leq T
\end{equation}
provided that there exists a positive constant $\widetilde{l}$ whose precise range will be specified in the proof of Lemma \ref{lemma4.3} such that
$$
\max\left\{\displaystyle\sup_{0\leq\tau\leq T}\mathcal{E}_{N_0+n}(\tau),
\sup_{0\leq\tau\leq T}\mathcal{E}_{N-1,N-1,-\gamma}(\tau),
\sup_{0\leq\tau\leq T}\mathcal{\overline{E}}_{N_0,N_0-\frac{\widetilde{l}}\gamma,-\gamma}(\tau)\right\}
\ \ \text{is sufficiently small}.
\leqno(H_1)
$$

Furthermore, as a consequence of \eqref{Lemma1-1}, we can get that
\begin{equation}\label{Lemma1-2}
\mathcal{E}^k_{N_0}(t)\lesssim\max\left\{\sup_{0\leq \tau\leq t}\mathcal{\overline{E}}_{N_0,N_0+\frac{k+\varrho}{2},-\gamma}(\tau),\sup_{0\leq \tau\leq t}\mathcal{E}_{N_0+k+\varrho}(\tau)\right\}(1+t)^{-(k+\varrho)}
\end{equation}
holds for $0\leq t\leq T$.
\end{lemma}
\begin{proof} First notice that under the smallness assumption $(H_1)$, one can deduce that
\begin{equation*}
\frac{d}{dt}\mathcal{E}^k_{N_0}(t)+\mathcal{D}^k_{N_0}(t)\leq 0,
\end{equation*}
which is an immediate consequence of Lemma \ref{lemma4.3} and Lemma \ref{Lemma4.4} whose proofs are complicated and thus are postponed to the next section.

Now we turn to compare the difference between $\mathcal{E}^k_{N_0}(t)$ and $\mathcal{D}^k_{N_0}(t)$. To this end, for the macroscopic component ${\bf P}f(t,x,v)$ and the electromagnetic field $[E(t,x), B(t,x)]$
one has by Lemma \ref{lemma2.2} that
\begin{equation*}
\begin{aligned}
\left\|\nabla^k({\bf P}f,B)\right\|\leq \left\|\nabla^{k+1}({\bf P}f,B)\right\|^{\frac{k+\varrho}{k+\varrho+1}}\left\|\Lambda^{-\varrho}({\bf P}f,B)\right\|^{\frac{1}{k+\varrho+1}}
\end{aligned}
\end{equation*}
and
\[
\left\|\nabla^{N_0}(E,B)\right\|\lesssim\left\|\nabla^{N_0-1}(E,B)\right\|^\frac{k+\varrho}{k+\varrho+1}
\left\|\nabla^{N_0+k+\varrho}(E,B)\right\|^\frac{1}{k+\varrho+1},
\]
while for the microscopic component $\{{\bf I}-{\bf P}\}f(t,x,v)$, we have by employing the H\"older inequality that
\begin{equation*}
\begin{aligned}
\sum_{k\leq|\alpha|\leq N_0}\left\|\partial^\alpha {\bf\{I-P\}}f\right\|\leq & \left\|\partial^\alpha {\bf\{I-P\}}f\langle v\rangle^{\frac{\gamma}2}\right\|^{\frac{k+\varrho}{k+\varrho+1}}
\left\|\partial^\alpha {\bf\{I-P\}}f\langle v\rangle^{-\frac{\gamma(k+\varrho)}{2}}\right\|^{\frac{1}{k+\varrho+1}}\\
\leq&\left\|\partial^\alpha {\bf\{I-P\}}f\right\|_\nu^{\frac{k+\varrho}{k+\varrho+1}}
\left\|w_{\frac{k+\varrho}{2},-\gamma}\partial^\alpha {\bf\{I-P\}}f\right\|^{\frac{1}{k+\varrho+1}}.
\end{aligned}
\end{equation*}
Therefore, we arrive at
\begin{equation*}
\begin{aligned}
\mathcal{E}^k_{N_0}(t)\leq \left\{\mathcal{D}^k_{N_0}(t)\right\}^\frac{k+\varrho}{k+\varrho+1}\left\{\max\left\{\sup_{0\leq \tau\leq t}\mathcal{\overline{E}}_{N_0,N_0+\frac{k+\varrho}{2},-\gamma}(\tau),\sup_{0\leq \tau\leq t}\mathcal{E}_{N_0+k+\varrho}(\tau)\right\}\right\}^\frac{1}{k+\varrho+1},
\end{aligned}
\end{equation*}
which combing with \eqref{Lemma1-1} yields that
\begin{equation*}
\frac{d}{dt}\mathcal{E}^k_{N_0}(t)+\left\{\max\left\{\sup_{0\leq \tau\leq t}\mathcal{\overline{E}}_{N_0,N_0+\frac{k+\varrho}{2},-\gamma}(\tau),\sup_{0\leq \tau\leq t}\mathcal{E}_{N_0+k+\varrho}(\tau)\right\}\right\}^{-\frac{1}{k+\varrho}}\left\{\mathcal{E}^k_{N_0}(t)\right\}^{1+\frac{1}{k+\varrho}}\leq 0.
\end{equation*}
Solving the above inequality directly gives
\begin{equation*}
\mathcal{E}^k_{N_0}(t)\lesssim\max\left\{\sup_{0\leq \tau\leq t}\mathcal{\overline{E}}_{N_0,N_0+\frac{k+\varrho}{2},-\gamma}(\tau),\sup_{0\leq \tau\leq t}\mathcal{E}_{N_0+k+\varrho}(\tau)\right\}(1+t)^{-(k+\varrho)}.
\end{equation*}
Here we have used the fact that
$$
\mathcal{E}^k_{N_0}(0)\lesssim \displaystyle\sup_{0\leq \tau\leq t}\left\{\mathcal{\overline{E}}_{N_0,N_0+\frac{k+\varrho}{2},-\gamma}(\tau)\right\}.
$$
This completes the proof of Lemma \ref{Lemma1}.
\end{proof}
Based on the above lemma, we can further obtain the temporal time decay of $\mathcal{E}^k_{N_0,\ell,-\gamma}(t)$ as in the following lemma.

\begin{lemma}\label{lemma2}
 Let $\ell\geq N_0$, $n\geq \frac23 N_0-\frac{5}3$ and suppose that
$$
\max\left\{\sup\limits_{0\leq\tau\leq t}\mathcal{E}_{N_0+n}(\tau), \sup\limits_{0\leq \tau\leq t}\mathcal{E}_{N_0,\ell-\frac{\widetilde{l}}\gamma}(\tau)\right\}\ \ \text{is sufficiently small}
\leqno(H_2)
$$
with $\widetilde{l}$ being given in Lemma \ref{Lemma1},
then the following estimates
\begin{equation}\label{lemma2-1}
\begin{aligned}
&\frac{d}{dt}\mathcal{E}^k_{N_0,\ell,-\gamma}(t)+\mathcal{D}^k_{N_0,\ell,-\gamma}(t)\\{}
\lesssim& \sum_{|\alpha|=N_0}\|\partial^\alpha E\|^2+\chi_{k\geq 2}\sum_{1\leq |\alpha'|\leq k-1,\atop
|\alpha|+|\beta|=N_0}\left\|\partial^{\alpha'} (E,B)\right\|_{L^\infty_x}^2\left\| w_{\ell-|\beta|-1,-\gamma}\partial^{\alpha-\alpha'}_{\beta+e_i}\{{\bf I-P}\}f\langle v\rangle^{1-\frac{3\gamma}2} \right\|^2
\end{aligned}
\end{equation}
hold for any $0\leq t\leq T$ and $k=0,1,\cdots,N_0-3$. Therefore, letting {$l_{0,k}\geq N_0$ with  $l_0= l_{0,0}=l_{0,1}$ and $l_{0,k-1}\geq l_{0,k}+ 3$ for  $2\leq k\leq N_0-3$},
one has
\begin{equation}\label{lemma3-1}
 \begin{split}
\mathcal{E}^k_{N_0,{l_{0,k}}+\frac{i}{2},-\gamma}(t)
\lesssim \max\left\{\sup_{0\leq \tau\leq t}\mathcal{\overline{E}}_{N_0,N_0+\frac{{k+1}+\varrho}{2},-\gamma}(\tau),\sup_{0\leq \tau\leq t}\mathcal{E}_{N_0+{k+1}+\varrho}(\tau)\right\}(1+t)^{-k-\varrho+i}, \quad   i=0,1,\cdots,{k+[\varrho]}.
\end{split}
\end{equation}
\end{lemma}

\begin{proof}
We omit the proof of $\eqref{lemma2-1}$ as it is similar to the one of \eqref{Lemma1-1}.  Here, we point out that the main difference for proving \eqref{Lemma1-1} and $\eqref{lemma2-1}$:
\begin{itemize}
\item The term $\displaystyle\sum_{|\alpha|=N_0}\|\partial^\alpha E\|^2$ appears when we deal with the term
$\displaystyle\sum_{|\alpha|=N_0}\left(\partial^\alpha E\cdot v\mu^{\frac12}, w^2_{\ell,-\gamma}\partial^\alpha f\right)$;

\item To deduce the desired estimates on
$$
\displaystyle{\sum_{1\leq |\alpha_1|\leq k-1,\atop|\alpha|=k,
|\alpha|+|\beta|=N_0}}
\left(\partial^{\alpha_1}E\cdot\nabla_v\partial_\beta^{\alpha-\alpha_1}\{{\bf I-P}\}f,w^2_{\ell-|\beta|,-\gamma}\{{\bf I-P}\}f\right)
$$
and
$$
\displaystyle{\sum_{1\leq |\alpha_1|\leq k-1,\atop|\alpha|=k,
|\alpha|+|\beta|=N_0}}
\left(\left(v\times\partial^{\alpha_1}B\right)\cdot\nabla_v\partial_\beta^{\alpha-\alpha_1}\{{\bf I-P}\}f,w^2_{\ell-|\beta|,-\gamma}\{{\bf I-P}\}f\right),
$$
one has to encounter the term
$$
\displaystyle{\sum_{1\leq |\alpha_1|\leq k-1,\atop|\alpha|=k,
|\alpha|+|\beta|=N_0}}
\left\|\partial^{\alpha_1} (E,B)\right\|_{L^\infty_x}^2\left\|\langle v\rangle^{1-\frac{3\gamma}2} w_{\ell-|\beta|-1,-\gamma}\partial^{\alpha-\alpha_1}_{\beta+e_i}\{{\bf I-P}\}f\right\|^2.
$$
\end{itemize}

With \eqref{lemma2-1} in hand, we now turn to prove \eqref{lemma3-1}.
{For the case $k=0,1$, the lase term on the right hand side of \eqref{lemma2-1} disappears, }we have by replacing the parameter $\ell$ in \eqref{lemma2-1} by ${l_0}+\frac{i}{2} (i=0,1,\cdots,{k+[\varrho]})$ and then by multiplying the resulting inequality by $(1+t)^{k+\varrho-i+\epsilon}$ that
\begin{equation}\label{E-i}
 \begin{split}
&\frac{d}{dt}\left\{(1+t)^{k+\varrho-i+\epsilon}\mathcal{E}^k_{N_0,{l_0}+\frac{i}{2},-\gamma}(t)\right\}+(1+t)^{k+\varrho-i+\epsilon}\mathcal{D}^k_{N_0,{l_0}+\frac{i}{2},-\gamma}(t)\\{}
\lesssim& \sum_{|\alpha|=N_0}(1+t)^{k+\varrho-i+\epsilon}\|\partial^\alpha E\|^2+(1+t)^{k+\varrho-i-1+\epsilon}\mathcal{E}^k_{N_0,{l_0}+\frac{i}{2},-\gamma}(t).
\end{split}
\end{equation}
Here
$\epsilon$ is taken as a sufficiently small positive constant.

{By replacing the parameter
$\ell$ in \eqref{lemma2-1} by $l_0+\frac{k+[\varrho]+1}2$,} it holds that
\begin{equation}\label{E_1}
\frac{d}{dt}\mathcal{E}^k_{N_0,{l_0+\frac{k+[\varrho]+1}2},-\gamma}(t)+\mathcal{D}^k_{N_0,{l_0+\frac{k+[\varrho]+1}2},-\gamma}(t)\lesssim \sum_{|\alpha|=N_0}\|\partial^\alpha E\|^2.
\end{equation}
By using the relation between the energy functional $\mathcal{E}^k_{N_0,{l_0},-\gamma}(t)$ and its corresponding dissipation functional $\mathcal{D}^k_{N_0,{l_0},-\gamma}(t)$, we deduce by a proper linear combination of (\ref{E-i}) and (\ref{E_1}) that
\begin{equation}\label{E um}
 \begin{split}
&\frac{d}{dt}\left\{\sum_{i=0}^{k+[\varrho]}C_i(1+t)^{k+\varrho-i+\epsilon}\mathcal{E}^k_{N_0,{l_0}+\frac{i}2,-\gamma}(t)+C_{{k+[\varrho]}+1}\mathcal{E}^k_{N_0,{l_0}+\frac{{k+[\varrho]}+1}{2},-\gamma}(t)\right\}\\{}
&+\sum_{i=0}^{k+[\varrho]}(1+t)^{k+\varrho-i+\epsilon}\mathcal{D}^k_{N_0,{l_0}+\frac{i}{2},-\gamma}(t)+\mathcal{D}^k_{N_0,{l_0}+\frac{{k+[\varrho]}+1}{2},-\gamma}(t)\\{}
\lesssim& \sum_{|\alpha|=N_0}(1+t)^{k+\varrho+\epsilon}\|\partial^\alpha E\|^2+(1+t)^{k+\varrho-1+\epsilon}\left\{\left\|\nabla^k({\bf P}f, B)\right\|^2
+\left\|\nabla^{N_0}B\right\|^2\right\}.
\end{split}
\end{equation}
On the other hand, Lemma \ref{Lemma1} tells us that
\begin{equation}\label{E-decay}
 \begin{split}
&\sum_{|\alpha|=N_0}(1+t)^{k+\varrho+\epsilon}\left\|\partial^\alpha E\right\|^2+(1+t)^{k+\varrho-1+\epsilon}\left\{\left\|\nabla^k({\bf P}f,E,B)\right\|^2+\left\|\nabla^{N_0}B\right\|^2\right\}\\
&\quad\lesssim \max\left\{\sup_{0\leq \tau\leq t}\mathcal{\overline{E}}_{N_0,N_0+\frac{{k+1}+\varrho}{2},-\gamma}(\tau),\sup_{0\leq \tau\leq t}\mathcal{E}_{N_0+{k+1}+\varrho}(\tau)\right\}(1+t)^{-1+\epsilon}.
\end{split}
\end{equation}
Plugging (\ref{E-decay}) into (\ref{E um}) and taking the time integration, one can get that
\begin{equation*}
 \begin{split}
&\sum_{i=0}^{k+[\varrho]}(1+t)^{k+\varrho-i+\epsilon}\mathcal{E}^k_{N_0,{l_0}+\frac{i}2,-\gamma}(t)+\mathcal{E}^k_{N_0,{l_0}+\frac{{k+[\varrho]}+1}{2},-\gamma}(t)\\{}
&\quad+\int_0^t\left\{\sum_{i=0}^{k+[\varrho]}(1+\tau)^{k+\varrho-i+\epsilon}\mathcal{D}^k_{N_0,{l_0}+\frac{i}{2},-\gamma}(\tau)+\mathcal{D}^k_{N_0,{l_0+\frac{k+[\varrho]+1}{2}},-\gamma}(\tau)\right\}d\tau
\\{}&\quad\quad\lesssim \max\left\{\sup_{0\leq \tau\leq t}\mathcal{\overline{E}}_{N_0,N_0+\frac{{k+1}+\varrho}{2},-\gamma}(\tau),\sup_{0\leq \tau\leq t}\mathcal{E}_{N_0+{k+1}+\varrho}(\tau)\right\}(1+t)^{\epsilon},
\end{split}
\end{equation*}
and the estimate (\ref{lemma3-1}) { with the case $k=0,1$ follows by multiplying the above inequality by $(1+t)^{-\epsilon}$ where we take $l_0=l_{0,1}=l_{0,1}$}.

As to the case of {$2\leq k\leq N_0-3$, noticing that $\gamma\in(-3,-1)$, let $l_{0,k}\geq N_0$ and $l_{0,k-1}\geq l_{0,k}+3$, $l_0= l_{0,0}=l_{0,1}$}, 
\eqref{lemma3-1} { with the case $2\leq k\leq N_0-3$} follows by using induction in $k$.
Thus the proof of Lemma \ref{lemma2} is complete.
\end{proof}

The above two lemmas are for the temporal time decay estimates on $\mathcal{E}^k_{N_0}(t)$ and $\mathcal{E}^k_{N_0,\ell,-\gamma}(t)$ respectively which are based on the following two assumptions:
\begin{itemize}
\item $n>\frac23N_0-\frac53$ and $N_0+n\leq N$. It is easy to see that if $N_0$ and $N$ are suitably chosen such that \eqref{N-N_0} holds, one can be able to find such an index $n$;

\item The assumptions $(H_1)$ and $(H_2)$ hold, that is, both
$$\max\left\{\displaystyle\sup_{0\leq\tau\leq t}\mathcal{E}_{N}(\tau),
\sup_{0\leq\tau\leq t}\mathcal{E}_{N-1,N-1,-\gamma}(\tau),
\sup_{0\leq\tau\leq t}\mathcal{\overline{E}}_{N_0,N_0-\frac{\widetilde{l}}\gamma,-\gamma}(\tau)\right\}
$$
and
$$
\max\left\{\sup_{0\leq\tau\leq t}\mathcal{\overline{E}}_{N_0,{l_0+\frac{k+2}2}-\frac{\widetilde{l}}\gamma,-\gamma}(\tau), \sup_{0\leq\tau\leq t}\mathcal{E}_{N}(\tau)\right\}
$$
are assumed to be small.
\end{itemize}
Set
\begin{equation}\label{l*}
l^*={\frac{k+2}2}-\frac{\widetilde{l}}\gamma,
\end{equation}
the above computation tells us that to guarantee the validity of the assumptions imposed in Lemma \ref{Lemma1} and Lemma \ref{lemma2}, we need to control $\mathcal{\overline{E}}_{N_0,{l_0}+l^*,-\gamma}(t)$, $\mathcal{E}_{N}(t)$, and $\mathcal{E}_{N-1,N-1,-\gamma}(t)$ suitably. To this end, we only outline the main ideas to yield these estimates and since the proofs are quite complicated, we leave the details to the next section. In fact, as pointed out in the introduction, if we perform the weighted energy estimate with respect to the weight function $w_{\ell-|\beta|,-\gamma}$, it is easy to see that the corresponding term
$I^{lt}_{|\alpha|,\ell-|\beta|,-\gamma}$ defined by \eqref{I^{lt}}
related to the linear transport term $v\cdot\nabla_x f$ can be controlled suitably. In fact, due to
$$
w^2_{\ell-|\beta|,-\gamma}(t,v)=w_{\ell-|\beta|,-\gamma}(t,v)\times w_{\ell-|\beta-e_i|,-\gamma}(t,v)\times\langle v\rangle^\gamma,
$$
the above term can be controlled by
$$
I^{lt}_{|\alpha|,\ell-|\beta|,-\gamma}\lesssim\left\|w_{\ell-|\beta-e_i|,-\gamma}\partial^{\alpha+e_i}_{\beta-e_i}\{{\bf I-P}\}f\right\|_\nu^2+
\varepsilon\left\|w_{\ell-|\beta|,-\gamma}\partial^{\alpha}_{\beta}\{{\bf I-P}\}f\right\|_\nu^2.
$$
On the other hand, since
$$
w^2_{\ell-|\beta|,-\gamma}(t,v)=w_{\ell-|\beta|,-\gamma}(t,v)\times w_{\ell-|\beta+e_i|,-\gamma}(t,v)\times\langle v\rangle^{-\gamma},
$$
one can deduce that for $\gamma<-1$, the terms \eqref{difficulty-1} and \eqref{difficulty-2} containing the electromagnetic field $[E(t,x), B(t,x)]$ can not be controlled by the extra dissipation term
$$
(1+t)^{-1-\vartheta}\left\| w_{\ell-|\beta|,-\gamma}(t,v)\partial^\alpha_\beta \{{\bf I}-{\bf P}\}f\langle v\rangle\right\|^2
$$
induced by the weight $w_{\ell-|\beta|,-\gamma}$.

To overcome such a difficulty, our main trick is to use the interpolation method for $v$ to bound these terms by $\left\|\nabla^2(E,B)\right\|_{H^{ N_0-2}_x}^{\frac {1}{\theta_1}} \mathcal{\widetilde{D}}_{N_0,l_0^*,1}(t)$ with
\begin{equation}\label{D_{N_0,l^*_0,1}}
\begin{split}
\mathcal{\widetilde{D}}_{N_0,l_0^*,1}(t)
\thicksim\sum_{1\leq n\leq N_0}\sum_{|\alpha|+|\beta|=n,\atop|\beta|=j,1\leq j\leq n}\left\|w_{l_0^*-j,1}\partial_\beta^\alpha\{{\bf I-P}\} f\right\|^2_\nu
+\sum_{1\leq n\leq N}\sum_{|\alpha|=n}\left\|w_{l_0^*,1}\partial^\alpha f\right\|^2_\nu+\left\|w_{l_0^*,1}\{{\bf I-P}\} f\right\|^2_\nu
\end{split}
\end{equation}
and some other similar terms. In fact, for $\mathcal{\overline{E}}_{N_0,\ell,-\gamma}(t)$, we can deduce that
\begin{lemma}\label{lemma4}
Let $N_0\geq 3$, $\ell\geq N_0$, {$\widetilde{\ell}_1>\frac12-\frac12\gamma$},
$\theta_1=\frac{1-2\gamma}{2\widetilde{l}_1-\gamma}$ and $l_0^*\geq \widetilde{\ell}_1-\frac\gamma2-\gamma\ell$,
then one has
\begin{multline}\label{lemma4-1}
\frac{d}{dt}\mathcal{\overline{E}}_{N_0,\ell,-\gamma}(t)+\mathcal{\overline{D}}_{N_0,\ell,-\gamma}(t)\lesssim
\|E\|_{L^\infty_x}^{\frac{2-\gamma}{1-\gamma}}\sum_{|\alpha|+|\beta|\leq N_0}\left\|w_{\ell-|\beta|,-\gamma}\partial^{\alpha}_\beta\{{\bf I-P}\}f\langle v\rangle\right\|^2\\
+\left\|\nabla^2(E,B)\right\|_{H^{ N_0-2}_x}^{\frac1{\theta_1}}\mathcal{\widetilde{D}}_{N_0,l_0^*,1}(t)+\sum_{|\alpha|=N_0}\varepsilon\|\partial^\alpha E\|^2
\end{multline}
provided that
$$
\mathcal{\overline{E}}_{N_0,\ell}(t)\  \text{is sufficiently small}.
\leqno(H_3)
$$
\end{lemma}

Note that $\varepsilon>0$ is an arbitrary small constant, and for brevity of presentation, here and in the sequel the dependence of coefficient constants on $\varepsilon$ similarly as on the right of \eqref{lemma4-1} is skipped, since the order of those terms are strictly higher than that of the quadratic term.

Similar to the definition of $\mathcal{\widetilde{D}}_{N_0,l_0^*,1}(t)$ given in Lemma \ref{lemma4}, for $m=N-1$ or $N$,
$\mathcal{\widetilde{D}}_{m,l_1^*,1}(t)$ is defined by
\begin{equation}\label{D_{m,l^*_1,1}}
\begin{split}
\mathcal{\widetilde{D}}_{m,l_1^*,1}(t)
\thicksim\sum_{N_0+1\leq n\leq m}\sum_{|\alpha|+|\beta|=n,\atop|\beta|=j,1\leq j\leq n}\left\|w_{l_1^*-j,1}\partial_\beta^\alpha\{{\bf I-P}\} f\right\|^2_\nu
+\sum_{N_0+1\leq n\leq m}\sum_{|\alpha|=n}\left\|w_{l_1^*,1}\partial^\alpha f\right\|^2_\nu.
\end{split}
\end{equation}
Here we emphasize that for the functional $\mathcal{\widetilde{D}}_{N-1,l_1^*,1}(t)$ or $\mathcal{\widetilde{D}}_{N,l_1^*,1}(t)$, the differentiation order in $x$ and $v$ starts from $N_0+1$, i.e.~$|\alpha|+|\beta|\geq N_0+1$.
We have the following two lemmas for $\mathcal{E}_{N}(t)$ and $\mathcal{E}_{N-1,\ell,-\gamma}(t)$ respectively:
\begin{lemma}\label{lemma5}
Assume $N_0\geq 3$, $N_0+1\leq N\leq 2N_0$,
{$\widetilde{l}_2 >\frac12-\frac12\gamma$}, $\theta_2=\frac{1-2\gamma}{2\widetilde{\ell}_2-\gamma}$, $l_0\geq \frac32-\frac1\gamma$, and $l_1^*\geq\widetilde{\ell}_2-\frac\gamma2$, we can deduce that
\begin{equation}\label{lemma5-1}
\frac{d}{dt}\mathcal{E}_{N}(t)+\mathcal{D}_{N}(t)
\lesssim \left(\| E\|_{L^{\infty}_x}+\left\|\nabla^2(E,B)\right\|_{H^{ N_0-2}_x}\right)^{\frac1{\theta_2}}\mathcal{\widetilde{D}}_{N,l_1^*,1}(t)
+\mathcal{E}_{N}(t)\mathcal{E}^1_{N_0,l_0,-\gamma}(t)
\end{equation}
provided that
$$
\mathcal{E}_{N}(t)\ \text{is sufficiently small.}
\leqno(H_4)
$$
\end{lemma}
\begin{lemma}\label{lemma6}
Take $N_0\geq 3$, $N_0+1\leq N\leq 2N_0$, {$\widetilde{l}_3>\frac12-\frac12\gamma$}, $\theta_3=\frac{1-2\gamma}{2\widetilde{\ell}_3-\gamma}$, $l_1\geq N$, {$l_0\geq l_1+\frac{5}{2}$}, and $l_1^*\geq \widetilde{\ell}_3-\frac\gamma2-\gamma l_1$, one has
\begin{equation}\label{lemma6-1}
\begin{aligned}
&\frac{d}{dt}\mathcal{E}_{N-1,l_1,-\gamma}(t)+\mathcal{D}_{N-1,l_1,-\gamma}(t)\\
\lesssim& \|E\|_{L^\infty_x}^{\frac{2-\gamma}{1-\gamma}}\sum_{|\alpha|+|\beta|\leq N-1}\left\|w_{l_1-|\beta|,-\gamma}\partial_\beta^{\alpha}\{{\bf I-P}\}f\langle v\rangle\right\|^2\\
&+\left\|\nabla^2(E,B)\right\|_{H^{ N_0-2}_x}^{\frac1{\theta_3}}\mathcal{\widetilde{D}}_{N-1,l_1^*,1}(t)+\mathcal{E}_{N}(t)\mathcal{E}^1_{N_0,l_0,-\gamma}(t)+\sum_{|\alpha|= N-1}\left\|\partial^\alpha E\right\|\left\|\mu^\delta\partial^\alpha f\right\|,
\end{aligned}
\end{equation}
where we have used the assumption that
$$
\mathcal{E}_{N-1,l_1}(t)\ \text{is sufficiently small.}
\leqno(H_5)
$$
\end{lemma}

Lemmas \ref{lemma4}-\ref{lemma6} together with the fact $\mathcal{E}^1_{N_0,l_0,-\gamma}(t)\in L^1({\mathbb{R}^+})$ which is a direct consequence of the estimates \eqref{lemma3-1} imply that to deduce the desired estimates on
$\mathcal{\overline{E}}_{N_0,\ell,-\gamma}(t)$, $\mathcal{E}_{N}(t)$, and $\mathcal{E}_{N-1,l_1,-\gamma}(t)$, one needs to bound $\mathcal{\widetilde{D}}_{N_0,l_0^*,1}(t)$ and $\mathcal{\widetilde{D}}_{N,l_1^*,1}(t)$ suitably. To this end, we have to perform the weighted energy estimates by replacing the weight $w_{\ell-|\beta|,-\gamma}$ by $w_{\ell-|\beta|,1}$ and in such a case, as explained in the introduction,
the terms $I^E_{|\alpha|,\ell-|\beta|,1}$ and $I^B_{|\alpha|, \ell-|\beta|,1}$ corresponding to \eqref{difficulty-1} and \eqref{difficulty-2} can be controlled by the corresponding extra dissipation rate $D^W_{|\alpha|,\ell-|\beta|,1}$ given by \eqref{dissipative-weight}
induced by the exponential factor of the weight $w_{\ell-|\beta|,1}(t,v)$ provided that the electromagnetic field $[E(t,x),$ $B(t,x)]$ enjoys certain temporal decay estimates. However, compared with the weighted energy estimate with respect to the weight $w_{\ell,-\gamma}$, the linear term $I^{lt}_{|\alpha|,\ell-|\beta|,1}$ defined by \eqref{I^{lt}} leads to a new difficult term
$$
\left\|w_{\ell-|\beta-e_i|, 1}\partial^{\alpha+e_i}_{\beta-e_i}{\bf\{I-P\}}f\langle v\rangle^{-\frac\gamma2-1}\right\|^2,
$$
which can not be controlled directly by combining the dissipative effects $D^L_{|\alpha|,\ell-|\beta|,1}$ induced by the linearized collision operator $L$.

Motivated by the argument developed in \cite{Hosono-Kawashima-2006} to deduce the temporal decay estimates on solutions to some nonlinear equations of regularity-loss type, we want to design different time increase rate $\sigma_{n,j}$ for
$$
\displaystyle\sum_{|\alpha|+|\beta|=n, |\beta|=j}\left\|w_{l_1^*-j,1}\partial_\beta^\alpha\{{\bf I-P}\} f\right\|^2,
$$
where $\sigma_{n,j}-\sigma_{n,j-1}=\frac{2(1+\gamma)}{\gamma-2}(1+\vartheta)$. For result in this direction, we have the following two lemmas whose proof will be given in the next section. The first one is concerned with the case of $N_0+1\leq n\leq N$.
\begin{lemma}\label{lemma7}
Assume $N_0\geq 4$,
$\sigma_{n,j}-\sigma_{n,j-1}=\frac{2(1+\gamma)}{\gamma-2}(1+\vartheta)$, $l_1^*\geq N$, and {$l_0\geq {l_1^*}+\frac{5}{2}$},
one can get that
\begin{eqnarray}\label{lemma7-1}
&&\sum_{N_0+1\leq n\leq N}\frac{d}{dt}\left\{\sum_{|\alpha|+|\beta|=n,\atop|\beta|=j,1\leq j\leq n}(1+t)^{-\sigma_{n,j}}\left\|w_{l_1^*-j,1}\partial_\beta^\alpha\{{\bf I-P}\} f\right\|^2\right.\left.+\sum_{|\alpha|=n}(1+t)^{-\sigma_{n,0}}\left\|w_{l_1^*,1}\partial^\alpha f\right\|^2\right\}\nonumber\\
&&
+\sum_{N_0+1\leq n\leq N}\left\{\sum_{|\alpha|+|\beta|=n,\atop|\beta|=j,1\leq j\leq n}(1+t)^{-\sigma_{n,j}}\left\|w_{l_1^*-j,1}\partial_\beta^\alpha\{{\bf I-P}\} f\right\|^2_\nu+\sum_{|\alpha|=n}(1+t)^{-\sigma_{n,0}}\left\|w_{l_1^*,1}\partial^\alpha f\right\|^2_\nu\right\}
\nonumber\\
&&+\sum_{N_0+1\leq n\leq N}\sum_{|\alpha|+|\beta|=n,\atop|\beta|=j,1\leq j\leq n}(1+t)^{-1-\vartheta-\sigma_{n,j}}\left\|w_{l_1^*-j,1}\partial_\beta^\alpha\{{\bf I-P}\} f\langle v\rangle\right\|^2\\  \nonumber
&&+\sum_{N_0+1\leq n\leq N}\sum_{|\alpha|=n}(1+t)^{-1-\vartheta-\sigma_{n,0}}\left\|w_{l_1^*,1}\partial^\alpha f\langle v\rangle\right\|^2\\  \nonumber
&\lesssim&\sum_{|\alpha|\leq N-1}\left\{\left\|\nabla^{|\alpha|+1}f\right\|^2_\nu
+\left\|\{{\bf I-P}\} f\right\|^2_\nu+\left\|\nabla^{|\alpha|}E\right\|^2\right\}
+(1+t)^{-2\sigma_{N,0}}\left\|\nabla^NE\right\|^2+\mathcal{E}_{N}(t)\mathcal{E}^1_{N_0,l_0,-\gamma}(t)
\nonumber\\
&&+\mathcal{E}_{N}(t)\mathcal{D}_{N}(t)+{\sum_{N_0+1\leq n\leq N,\atop 0\leq j\leq n}}(1+t)^{-\sigma_{n,j}}E^n_{tri,j}(t)+\eta{\sum_{N_0+1\leq n\leq N ,\atop
1\leq j\leq n}}\sum_{|\alpha|+|\beta|=n,\atop |\beta|=j, |\beta'|<j}(1+t)^{-\sigma_{n,j}}
\left\|w_{l_1^*-|\beta'|,1}\partial^\alpha_{\beta'}\{{\bf I-P}\} f\right\|^2_\nu,\nonumber
\end{eqnarray}
where $E^n_{tri,j}(t)$ is defined by
\begin{eqnarray}
E^{n}_{tri,j}(t)&\sim&
\sum_{|\alpha|+|\beta|=n,\atop |\beta|=j
}\|E\|_{L^\infty}^{\frac{2-\gamma}{1-\gamma}}
\left\|w_{l_1^*-j,1}\partial^\alpha_\beta \{{\bf I-P}\}f\langle v\rangle\right\|^2\notag\\ \nonumber
&&+{\sum_{|\alpha|+|\beta|=n,\atop |\beta|=j
}\sum_{|\alpha-\alpha_1|+j+m\geq N_0+1,\atop 1\leq|\alpha_1|\leq N_0-2,m\leq1}}(1+t)^{1+\vartheta}\left\|\partial^{\alpha_1} B\right\|^2_{L^\infty_x} \left\|w_{l_1^*-j-m,1}\nabla^m_v \partial_\beta^{\alpha-\alpha_1}{\bf \{I-P\}}f\langle v\rangle\right\|^2\nonumber\\
&&+{\sum_{|\alpha|+|\beta|=n,\atop |\beta|=j
}\sum_{|\alpha-\alpha_1|+j+m\geq N_0+1,\atop N_0-1\leq|\alpha_1|\leq N_0,m\leq1}}
(1+t)^{1+\vartheta}\left\|\partial^{\alpha_1} B\right\|^2 \left\|w_{l_1^*-j-m,1}
\nabla^m_v \partial_\beta^{\alpha-\alpha_1} {\bf \{I-P\}}f\langle v\rangle\right\|^2_{L^2_vL^\infty_x}\nonumber\\
&&+{\sum_{|\alpha|+|\beta|=n,\atop |\beta|=j
}\sum_{|\alpha-\alpha_1|+j+m\geq N_0+1,\atop { 1\leq|\alpha_1|\leq N_0-2},m\leq 1}}(1+t)^{1+\vartheta}
\left\|\partial^{\alpha_1} E\right\|^2_{L^\infty_x}\left\|w_{l_1^*-j-m,1}\nabla^m_v \partial_\beta^{\alpha-\alpha_1}{\bf \{I-P\}}f\right\|^2 \nonumber\\
&&+{\sum_{|\alpha|+|\beta|=n,\atop |\beta|=j
}\sum_{|\alpha-\alpha_1|+j+m\geq N_0+1,\atop{ N_0-1\leq|\alpha_1|\leq N_0}, m\leq 1}}
(1+t)^{1+\vartheta}\left\|\partial^{\alpha_1} E\right\|^2\left\|w_{l_1^*-j-m,1}
\nabla^m_v \partial_\beta^{\alpha-\alpha_1}{\bf \{I-P\}}f\right\|^2_{L^2_vL^\infty_x}\nonumber\\
&&+\max\left\{\mathcal{E}_{N_0,l_0,-\gamma}(t), \mathcal{E}_{\overline{m},\overline{m},-\gamma}(t)\right\}
\sum_{|\alpha'|+|\beta'|\leq n,\atop
|\beta'|\leq |\beta|=j}\left\|w_{l_1^*-|\beta'|,1}\partial_{\beta'}^{\alpha'}{\bf\{I-P\}}f\right\|^2_{\nu }.\label{Enj}
\end{eqnarray}
\end{lemma}

Similar to Lemma \ref{lemma7}, we can also get for the case of $1\leq n\leq N_0$ that
\begin{lemma}\label{lemma8}
Under the assumptions of Lemma \ref{lemma7},
for $l^*_0\geq N_0$,
we have
\begin{eqnarray}\label{lemma8-1}
&&\sum_{1\leq n\leq N_0}\frac{d}{dt}\Bigg\{\sum_{|\alpha|+|\beta|=n,\atop|\beta|=j,1\leq j\leq n}(1+t)^{-\sigma_{n,j}}\left\|w_{l_0^*-j,1}\partial_\beta^\alpha\{{\bf I-P}\} f\right\|^2+\sum_{|\alpha|=n}(1+t)^{-\sigma_{n,0}}\left\|w_{l_0^*,1}\partial^\alpha f\right\|^2\nonumber\\
&&+(1+t)^{-\sigma_{0,0}}\left\|w_{l_0^*,1}\{{\bf I-P}\} f\right\|^2\Bigg\}
+\sum_{1\leq n\leq N_0}\left\{\sum_{|\alpha|+|\beta|=n,\atop|\beta|=j,1\leq j\leq n}(1+t)^{-\sigma_{n,j}}\left\|w_{l_0^*-j,1}\partial_\beta^\alpha\{{\bf I-P}\} f\right\|^2_\nu\right.\\
&&\left.+\sum_{|\alpha|=n}(1+t)^{-\sigma_{n,0}}\left\|w_{l_0^*,1}\partial^\alpha f\right\|_\nu^2\right\}
+(1+t)^{-\sigma_{0,0}}\left\|w_{l_0^*,1}\{{\bf I-P}\} f\right\|^2_\nu\nonumber\\
&&+\sum_{1\leq n\leq N_0}\left\{\sum_{|\alpha|+|\beta|=n,\atop|\beta|=j,1\leq j\leq n}(1+t)^{-1-\vartheta-\sigma_{n,j}}\left\|w_{l_0^*-j,1}\partial_\beta^\alpha\{{\bf I-P}\} f\langle v\rangle\right\|^2\right.\nonumber\\
&& \left.+\sum_{|\alpha|=n}(1+t)^{-1-\vartheta-\sigma_{n,0}}\left\|w_{l_0^*,1}\partial^\alpha f\langle v\rangle\right\|^2\right\}+(1+t)^{-1-\vartheta-\sigma_{0,0}}\left\|w_{l_0^*,1}\{{\bf I-P}\} f\langle v\rangle\right\|^2\nonumber\\
&\lesssim&\sum_{|\alpha|\leq N_0-1}\left\{\left\|\nabla^{|\alpha|+1}f\right\|_\nu^2+\left\|\{{\bf I-P}\} f\right\|^2_\nu
+\left\|\nabla^{|\alpha|}E\right\|^2+\left\|\nabla^{N_0}E\right\|^2\right\}+\mathcal{E}_{N_0}(t)\mathcal{D}_{N_0}(t)\nonumber\\
&&+{\sum_{0\leq n\leq N_0,\atop0\leq j\leq n}}(1+t)^{-\sigma_{n,j}}F^n_{tri,j}(t)+\eta\sum_{1\leq n\leq N_0,\atop 1\leq j\leq n,}{\sum_{|\alpha|+|\beta|=n,\atop|\beta|=j,|\beta'|<j}}(1+t)^{-\sigma_{n,j}}
\left\|w_{l_0^*-|\beta'|,1}\partial^\alpha_{\beta'}\{{\bf I-P}\} f\right\|^2_\nu,\nonumber
\end{eqnarray}
where $F^{n}_{tri,j}(t)$ is defined by
\begin{eqnarray}
F^{n}_{tri,j}(t)&\thicksim&\sum_{|\alpha|+|\beta|=n,\atop |\beta|=j
}\|E\|_{L^\infty}^{\frac{2-\gamma}{1-\gamma}}
\left\|w_{l_0^*-j,1}\partial^\alpha_\beta \{{\bf I-P}\}f\langle v\rangle\right\|^2\notag\\ \nonumber
&&+{\sum_{|\alpha|+|\beta|=n,\atop |\beta|=j
}
\sum_{1\leq|\alpha_1|\leq \min\{n-j, N_0-2\},\atop m\leq1}}(1+t)^{1+\vartheta}
\left\|\partial^{\alpha_1} B\right\|^2_{L^\infty_x}\left\|w_{l_0^*-m-j,1}\nabla^m_v\partial_\beta^{\alpha-\alpha_1}{\bf \{I-P\}}f\langle v\rangle\right\|^2\nonumber\\
&&+{\sum_{|\alpha|+|\beta|=n,\atop |\beta|=j
}\sum_{N_0-1\leq|\alpha_1|\leq N_0,\atop m\leq1}}(1+t)^{1+\vartheta}\left\|\partial^{\alpha_1} B\right\|^2\left\|w_{l_0^*-m-j,1}\nabla^m_v\partial_\beta^{\alpha-\alpha_1}{\bf \{I-P\}}f\langle v\rangle\right\|^2_{L^2_vL^\infty_x}\nonumber\\ \notag
&&+{\sum_{|\alpha|+|\beta|=n,\atop |\beta|=j
}
\sum_{1\leq|\alpha_1|\leq \min\{n-j, N_0-2\},\atop m\leq1}}(1+t)^{1+\vartheta}\left\|\partial^{\alpha_1} E\right\|^2_{L^\infty_x} \left\|w_{l_0^*-m-j,1}\nabla^m_v\partial_\beta^{\alpha-\alpha_1}{\bf \{I-P\}}f\right\|^2\\ \nonumber
&&+{\sum_{|\alpha|+|\beta|=n,\atop |\beta|=j
}\sum_{N_0-1\leq|\alpha_1|\leq N_0,\atop m\leq1}}(1+t)^{1+\vartheta}\left\|\partial^{\alpha_1} E\right\|^2\left\|w_{l_0^*-m-j,1}\nabla^m_v\partial^{\alpha-\alpha_1}{\bf \{I-P\}}f\right\|^2_{L^2_vL^\infty_x}\nonumber\\
&&+\left(\mathcal{E}_{N_0,0}(t)+\left\|w_{l_0^*,1}f\right\|_{L^2_vH^2_x}^2\right)
\sum_{|\alpha'|+|\beta'|\leq n,\atop|\alpha'|\geq 1,|\beta'|\leq j}\left\|w_{l_0^*-|\beta'|,1}\partial^{\alpha'}_{\beta'}f\right\|^2_\nu\nonumber\\
&&
+\sum_{|\alpha|+|\beta|=n,\atop |\beta|=j}
\sum_{1\leq|\alpha_1|+|\beta_1|\leq n-1}\left\|w_{l_0^*,1}\partial^{\alpha_1}_{\beta_1}f\right\|^2_{L^2_vL^3_x}
\left\|w_{l_0^*,1}\partial^{\alpha-\alpha_1}_{\beta-\beta_1}f\right\|^2_{L^2_\nu L^6_x}.\label{Fnj}
\end{eqnarray}
\end{lemma}

\subsection{The proof of Theorem \ref{Th1.1}}
We now prove Theorem \ref{Th1.1} in this subsection. For this purpose, suppose that the Cauchy problem (\ref{f}) and (\ref{f-initial}) admits a unique local solution $[f(t,x,v), E(t,x), B(t,x)]$ defined on the time interval $ 0\leq t\leq T$ for some $0<T<\infty$ and $f(t,x,v)$ satisfies the a priori assumption \eqref{a-priori-assumption},
where the parameters $\overline{m},N_0,N,l_0,l_1,$ and $l^*, l^*_0,l^*_1,\sigma_{n,j}$ are given in Theorem \ref{Th1.1} and $M$ is a sufficiently small positive constant. Then to use the continuity argument to extend such a solution step by step to a global one, one only need to deduce certain uniform-in-time energy type estimates on $f(t,x,v)$ such that the a priori assumption \eqref{a-priori-assumption} can be closed, which is the main result of the following lemma.
\begin{lemma}\label{lemma9} Assume that
\begin{itemize}
\item The assumptions of Lemma \ref{lemma7} hold;
\item $\vartheta$ is chosen to satisfy \eqref{range-vartheta}, $N_0$ and $N$ satisfy \eqref{N-N_0};
\item $\sigma_{N,0}=\frac{1+\epsilon_0}{2}$, $\sigma_{n,0}=0$ for $n\leq N-1$;
\item $\widetilde{\ell}_1\geq \frac\gamma2+ \frac{(1-2\gamma)\sigma_{N_0,N_0}}{2+\varrho},$
$\widetilde{\ell}_2\geq\frac\gamma2+\frac{2(1-2\gamma)\sigma_{N,N}}{3+2\varrho}$ and $\widetilde{\ell}_3\geq\frac\gamma2+\frac{(1-2\gamma)\sigma_{N-1,N-1}}{2+\varrho}$;
\item $l_1\geq N$, $l_1^*\geq\max\left\{\widetilde{\ell}_2-\frac\gamma2,\  \widetilde{\ell}_3-\frac\gamma2-\gamma l_1\right\}$, {$l_0\geq {l_1^*}+\frac{5}{2}$},
$l_0^*\geq \widetilde{\ell}_1-\frac\gamma2-\gamma(l_0+l^*)$ with {$l^*=\frac32-\frac{\widetilde{l}}\gamma$} with $\widetilde{l}$ being given in Lemma \ref{Lemma1};
\item The a priori assumption \eqref{a-priori-assumption} holds for some sufficiently small $M>0$.
\end{itemize}
Then it holds that
\begin{equation}\label{lemma9-1}
\begin{aligned}
&\mathcal{E}_N(t)+\mathcal{\overline{E}}_{N_0,l_0+l^*,-\gamma}(t)
+\mathcal{E}_{N-1,l_1,-\gamma}(t)\\{}
&+\sum_{N_0+1\leq n\leq N}\sum_{|\alpha|+|\beta|=n,\atop|\beta|=j,1\leq j\leq n}(1+t)^{-\sigma_{n,j}}\left\|w_{l_1^*-j,1}\partial_\beta^\alpha\{{\bf I-P}\} f\right\|^2
+\sum_{N_0+1\leq n\leq N-1}\sum_{|\alpha|=n}\left\|w_{l_1^*,1}\partial^\alpha f\right\|^2\\
&+\sum_{|\alpha|=N}(1+t)^{-(1+\epsilon_0)/2}\left\|w_{l_1^*,1}\partial^\alpha f\right\|^2
+\sum_{1\leq n\leq N_0}\sum_{|\alpha|+|\beta|=n,\atop|\beta|=j,1\leq j\leq n}(1+t)^{-\sigma_{n,j}}\left\|w_{l_0^*-j,1}\partial_\beta^\alpha\{{\bf I-P}\} f\right\|^2\\
&+\sum_{1\leq n\leq N_0}\sum_{|\alpha|=n}\left\|w_{l_0^*,1}\partial^\alpha f\right\|^2+\left\|w_{l_0^*,1}\{{\bf I-P}\}f\right\|^2\\
\lesssim& Y_0^2
\end{aligned}
\end{equation}
for all $0\leq t\leq T$.
\end{lemma}
\begin{proof} Before proving \eqref{lemma9-1}, we first point out that if the assumptions stated in Lemma \ref{lemma9} hold, especially the a priori assumption \eqref{a-priori-assumption} is satisfied and the parameters such as $\vartheta, \varrho, N_0, N, \sigma_{n,j}, \widetilde{l}_1, \widetilde{l}_2, \widetilde{l}_3, l_1,\l^*_1, l_0,$ and $l_*$ satisfy the conditions listed in Lemma \ref{lemma9}, then all the conditions listed in Lemma \ref{Lemma1}, Lemma \ref{lemma2}, Lemma \ref{lemma4}, Lemma \ref{lemma5}, Lemma \ref{lemma6}, Lemma \ref{lemma7}, and Lemma \ref{lemma8} are satisfied, and based on the results obtained in these lemmas, we can deduce that:
\begin{itemize}
\item[(i).] If we take
$$
\sigma_{n,0}=\left\{
\begin{array}{cl}
\frac{1+\epsilon_0}{2}, & n=N,\\{}
0,& n\leq N-1
\end{array}
\right.
$$
and notice that
$$
\sigma_{n,j}-\sigma_{n,j-1}=\frac{2(1+\gamma)}{\gamma-2}(1+\vartheta),
$$
we can deduce that
$$
\max_{ N_0+1\leq n\leq N,  0\leq j\leq n}\{\sigma_{n,j}\}=\sigma_{N,N},\quad \displaystyle\max_{ N_0+1\leq n\leq N-1,  0\leq j\leq n}\{\sigma_{n,j}\}=\sigma_{N-1,N-1},\quad
\max_{ 0\leq n\leq N_0,  0\leq j\leq n}\{\sigma_{n,j}\}=\sigma_{N_0,N_0};
$$
\item[(ii).] If we choose $\widetilde{\ell}_2\geq\frac\gamma2+\frac{2(1-2\gamma)\sigma_{N,N}}{3+2\varrho}$ and $l_1^*\geq\widetilde{\ell}_2-\frac\gamma2$, then we can deduce that
$\theta_2=\frac{1-2\gamma}{2\widetilde{\ell}_2-\gamma}\leq\frac{3+2\varrho}{4\sigma_{N,N}}$. Consequently, we have from Lemma \ref{Lemma1} that
\begin{equation}\label{1-key}
\begin{split}
&\left(\| E\|_{L^{\infty}}+\left\|\nabla^2(E,B)\right\|_{H^{ N_0-2}_x}\right)^{\frac1{\theta_2}}\mathcal{\widetilde{D}}_{N,l_1^*,1}(t)\\
\lesssim&\left(\| E\|_{L^{\infty}}+\left\|\nabla^2(E,B)\right\|_{H^{ N_0-2}_x}\right)^{\frac1{\theta_2}}(1+t)^{\sigma_{N,N}}(1+t)^{-\sigma_{N,N}}\mathcal{\widetilde{D}}_{N,l_1^*,1}(t)\\
\lesssim&X(t)^{\frac1{2\theta_2}}(1+t)^{-\left(\frac34+\frac\varrho2\right)\frac{1}{\theta_2}}(1+t)^{\sigma_{N,N}}(1+t)^{-\sigma_{N,N}}
\mathcal{\widetilde{D}}_{N,l_1^*,1}(t)\\
\lesssim&X(t)^{\frac1{2\theta_2}}(1+t)^{-\sigma_{N,N}}
\mathcal{\widetilde{D}}_{N,l_1^*,1}(t);
\end{split}
\end{equation}
\item [(iii).] If we take  $\widetilde{\ell}_3\geq\frac\gamma2+\frac{(1-2\gamma)\sigma_{N-1,N-1}}{2+\varrho}$ and $l_1^*\geq \widetilde{\ell}_3-\frac\gamma2-\gamma l_1$,
then $\theta_3=\frac{1-2\gamma}{2\widetilde{l}_3-\gamma}\leq\frac{2+\varrho}{2\sigma_{N-1,N-1}}$ and we have from Lemma \ref{Lemma1} that
\begin{equation}\label{2-key}
\begin{split}
&\left\|\nabla^2(E,B)\right\|_{H^{ N_0-2}_x}^{\frac1{\theta_3}}\mathcal{\widetilde{D}}_{N-1,l_1^*,1}(t)\\
\lesssim&X(t)^{\frac1{2\theta_3}}
(1+t)^{-\left(1+\frac\varrho2\right)\frac{1}{\theta_1}}
(1+t)^{\sigma_{N-1,N-1}}
(1+t)^{-\sigma_{N-1,N-1}}
\mathcal{\widetilde{D}}_{N-1,l_1^*,1}(t)\\
\lesssim&X(t)^{\frac1{2\theta_3}}
(1+t)^{-\sigma_{N-1,N-1}}
\mathcal{\widetilde{D}}_{N-1,l_1^*,1}(t);
\end{split}
\end{equation}
\item [(iv).] For $\widetilde{\ell}_1\geq \frac\gamma2+ \frac{(1-2\gamma)\sigma_{N_0,N_0}}{2+\varrho}$ and $l_0^*\geq \widetilde{\ell}_1-\frac\gamma2-\gamma(l_0+l^*)$, it is easy to see that $\theta_1=\frac{1-2\gamma}{2\widetilde{\ell}_1-\gamma}\leq\frac{2+\varrho}{2\sigma_{N_0,N_0}}$ and consequently we have from Lemma \ref{Lemma1} that
\begin{equation}\label{3-key}
\begin{split}
\left\|\nabla^2(E,B)\right\|_{H^{ N_0-2}_x}^{\frac1{\theta_1}}\mathcal{\widetilde{D}}_{N_0,l_0^*,1}(t)
\lesssim X(t)^{\frac1{2\theta_1}}(1+t)^{-\sigma_{N_0,N_0}}
\mathcal{\widetilde{D}}_{N_0,l_0^*,1}(t);
\end{split}
\end{equation}
\item [(v).]Since $N_0\geq 4$,
by \eqref{Lemma1-2}, we take
$0<\vartheta\leq \frac{\gamma-2\varrho\gamma+4\varrho+2}{4-4\gamma}$
such that
\begin{eqnarray}\label{theta-def-1}
&&\|E\|_{L^\infty_x}^{\frac{2-\gamma}{1-\gamma}}\sum_{|\alpha|+|\beta|\leq N-1\ or\  N_0}\left\|w_{\ell-|\beta|,-\gamma}\partial_\beta^{\alpha}\{{\bf I-P}\}f\langle v\rangle\right\|^2\\\nonumber
&\lesssim&
\left\|\nabla_xE\right\|^{\frac{2-\gamma}{2(1-\gamma)}}
\left\|\nabla^2_xE\right\|^{\frac{2-\gamma}{2(1-\gamma)}}
\sum_{|\alpha|+|\beta|\leq N-1\ or\  N_0}\left\|w_{\ell-|\beta|,-\gamma}\partial_\beta^{\alpha}\{{\bf I-P}\}f\langle v\rangle\right\|^2\\ \nonumber
&\lesssim&\left\{\sup_{0\leq \tau\leq t}\mathcal{\overline{E}}_{N_0,N_0+\frac{k+\varrho}{2},-\gamma}(\tau),\sup_{0\leq \tau\leq t}\mathcal{E}_{N_0+k+\varrho}(\tau)\right\}^{\frac{2-\gamma}{2(1-\gamma)}}\sum_{|\alpha|=N}(1+t)^{-\left(\frac34+\frac\varrho2\right)\frac{2-\gamma}
{1-\gamma}}
\\ \nonumber
&&\times \sum_{|\alpha|+|\beta|\leq N-1\ or\  N_0}\left\|w_{\ell-|\beta|,-\gamma}\partial_\beta^{\alpha}\{{\bf I-P}\}f\langle v\rangle\right\|^2\\ \nonumber
&\lesssim&
X(t)^{\frac{2-\gamma}{2(1-\gamma)}}\sum_{|\alpha|=N}(1+t)^{-1-\vartheta}
 \sum_{|\alpha|+|\beta|\leq N-1\ or\  N_0}\left\|w_{\ell-|\beta|,-\gamma}\partial_\beta^{\alpha}\{{\bf I-P}\}f\langle v\rangle\right\|^2.\\ \nonumber
\end{eqnarray}
\end{itemize}

With the above preparations in hand, we now turn to prove \eqref{lemma9-1}. To this end, we first
multiply (\ref{lemma5-1}) by $(1+t)^{-\epsilon_0}$ and get by employing \eqref{1-key} that
\begin{equation}\label{e-n-1}
\begin{aligned}
&\frac{d}{dt}\left\{(1+t)^{-\epsilon_0}\mathcal{E}_{N}(t)\right\}
+\epsilon_0(1+t)^{-1-\epsilon_0}\mathcal{E}_{N}(t)+(1+t)^{-\epsilon_0}\mathcal{D}_{N}(t)\\
\lesssim&
(1+t)^{-\epsilon_0}\left(\| E\|_{L^{\infty}}+\left\|\nabla^2(E,B)\right\|_{H^{ N_0-2}}\right)^{\frac1{\theta_2}}\mathcal{\widetilde{D}}_{N,l_1^*,1}(t)
+(1+t)^{-\epsilon_0}\mathcal{E}_{N}(t)\mathcal{E}^1_{N_0,l_0,-\gamma}(t)\\
\lesssim&
(1+t)^{-\epsilon_0}X(t)^{\frac1{2\theta_2}}(1+t)^{-\sigma_{N,N}}
\mathcal{\widetilde{D}}_{N,l_1^*,1}(t)
+(1+t)^{-\epsilon_0}\mathcal{E}_{N}(t)\mathcal{E}^1_{N_0,l_0,-\gamma}(t).
\end{aligned}
\end{equation}
It is worth pointing out that the term $\epsilon_0(1+t)^{-1-\epsilon_0}\mathcal{E}_{N}(t)$ on the left hand side of the above inequality can be used to control the term $\displaystyle\sum_{|\alpha|=N}(1+t)^{-2\sigma_{N,0}}\|\partial^{\alpha}E\|^2$ on the right hand of \eqref{lemma8-1}.

Secondly, plugging \eqref{1-key} into \eqref{lemma5-1} gives
\begin{equation}\label{X(t)-1}
\frac{d}{dt}\mathcal{E}_{N}(t)+\mathcal{D}_{N}(t)
\lesssim X(t)^{\frac1{2\theta_2}}(1+t)^{-\sigma_{N,N}}
\mathcal{\widetilde{D}}_{N,l_1^*,1}(t)
+\mathcal{E}_{N}(t)\mathcal{E}^1_{N_0,l_0,-\gamma}(t).
\end{equation}

Thirdly, by combing \eqref{2-key}, \eqref{theta-def-1} with \eqref{lemma6-1}, one has
\begin{equation}\label{X(t)-2}
\begin{aligned}
&\frac{d}{dt}\mathcal{E}_{N-1,l_1,-\gamma}(t)+\mathcal{D}_{N-1,l_1,-\gamma}(t)\\{}
\lesssim& X(t)^{\frac1{2\theta_3}}(1+t)^{-\sigma_{N-1,N-1}}
\mathcal{\widetilde{D}}_{N-1,l_1^*,1}(t)
+\mathcal{E}_{N}(t)\mathcal{E}^1_{N_0,l_0,-\gamma}(t)+\sum_{|\alpha|= N-1}\left\|\partial^\alpha E\right\|\left\|\mu^\delta\partial^\alpha f\right\|.
\end{aligned}
\end{equation}
Thus if $l_1^*$ is suitably chosen such that $l_1^*\geq\max\left\{\widetilde{\ell}_2-\frac\gamma2,\  \widetilde{\ell}_3-\frac\gamma2-\gamma l_1\right\}$, then the estimates \eqref{X(t)-1} and \eqref{X(t)-2} hold and from these we can deduce that
\begin{itemize}
\item If we choose $l_1\geq N$, then once we deduce the estimate on $\mathcal{E}_{N-1,l_1,-\gamma}(t)$, the estimate on $\mathcal{E}_{N-1,N-1,-\gamma}(t)$ follows immediately;
\item A sufficient condition to control the term $\mathcal{E}_{N}(t)\mathcal{E}^1_{N_0,l_0,-\gamma}(t)$ which appears on the right hand side of \eqref{X(t)-1}, \eqref{X(t)-2}, and \eqref{lemma7-1} is to show that $\mathcal{E}^1_{N_0,l_0,-\gamma}(t)\in L^1(\mathbb{R}^3)$. In fact Lemma \ref{lemma2} provides us with such a nice estimate provided that $\displaystyle\sup_{0\leq\tau\leq t}\mathcal{\overline{E}}_{N_0,l_0+l^*,-\gamma}(\tau)$ is sufficiently small.
\end{itemize}
Now we turn to estimate $\mathcal{\overline{E}}_{N_0,l_0+l^*,-\gamma}(t)$ and for this purpose, we first notice from \eqref{l*} that since $k=1$, $l^*$ is now taken as $l^*=\frac32-\frac{\widetilde{l}}\gamma$, then for
{$l_0\geq {l_1^*}+\frac{5}{2}$}, we have by replacing $\ell$ in the estimate \eqref{lemma4-1} with $l_0+l^*$ and the estimate \eqref{3-key} that
\begin{equation}\label{X(t)-3}
\frac{d}{dt}\mathcal{\overline{E}}_{N_0,l_0+l^*,-\gamma}(t)+\mathcal{\overline{D}}_{N_0,l_0+l^*,-\gamma}(t)\lesssim
X(t)^{\frac1{2\theta_1}}(1+t)^{-\sigma_{N_0,N_0}}
\mathcal{\widetilde{D}}_{N_0,l_0^*,1}(t)+\sum_{|\alpha|=N_0}\varepsilon\left\|\partial^\alpha E\right\|^2,
\end{equation}
where we have used the estimate \eqref{theta-def-1}.

Taking a proper linear combination of (\ref{X(t)-1}), (\ref{X(t)-2}), (\ref{X(t)-3}), (\ref{lemma7-1}), (\ref{lemma8-1}), and (\ref{e-n-1}) and by using the smallness of $X(t)$ and $\varepsilon$, we can deduce by taking the time integration  from $0$ to $t$ to the resulting differential inequality that
\begin{equation*}
\begin{aligned}
&\mathcal{E}_N(t)+\mathcal{\overline{E}}_{N_0,l_0+l^*,-\gamma}(t)
+\mathcal{E}_{N-1,l_1,-\gamma}(t)\\{}
&+\sum_{N_0+1\leq n\leq N}\sum_{|\alpha|+|\beta|=n,\atop|\beta|=j,1\leq j\leq n}(1+t)^{-\sigma_{n,j}}\left\|w_{l_1^*-j,1}\partial_\beta^\alpha\{{\bf I-P}\} f\right\|^2+\sum_{N_0+1\leq n\leq N-1}\sum_{|\alpha|=n}\left\|w_{l_1^*,1}\partial^\alpha f\right\|^2\\
&+\sum_{|\alpha|=N}(1+t)^{-\frac{1+\epsilon_0}{2}}\left\|w_{l_1^*,1}\partial^\alpha f\right\|^2 +\sum_{1\leq n\leq N_0}\sum_{|\alpha|+|\beta|=n,\atop|\beta|=j,1\leq j\leq n}(1+t)^{-\sigma_{n,j}}\left\|w_{l_0^*-j,1}\partial_\beta^\alpha\{{\bf I-P}\} f\right\|^2\\
&+\sum_{1\leq n\leq N_0}\sum_{|\alpha|=n}\left\|w_{l_0^*,1}\partial^\alpha f\right\|^2 +\left\|w_{l_0^*,1}\{{\bf I-P}\} f\right\|^2\\
\lesssim& Y_0^2.
\end{aligned}
\end{equation*}
Here we have used the following estimate
\begin{eqnarray}\label{E^n_{tri,j}-F^n_{tri,j}}
&&{\sum_{{N_0+1\leq n\leq N,\atop0\leq j\leq n}}}(1+t)^{-\sigma_{n,j}}E^n_{tri,j}(t)+{\sum_{{0\leq n\leq N_0,\atop0\leq j\leq n}}}(1+t)^{-\sigma_{n,j}}F^n_{tri,j}(t)\nonumber\\
&\lesssim&{\sum_{{N_0+1\leq |\alpha|+|\beta|\leq N}}}X(t)(1+t)^{-\sigma_{|\alpha|+|\beta|,|\beta|}}\mathcal{\widetilde{D}}^{|\alpha|,|\beta|}_{l_1^*,1}(t)\\
\nonumber&&+{\sum_{{0\leq |\alpha|+|\beta|\leq N_0}}}X(t)(1+t)^{-\sigma_{|\alpha|+|\beta|,|\beta|}}\mathcal{\widetilde{D}}^{|\alpha|,|\beta|}_{l_0^*,1}(t)\\ \nonumber
&&
+
{\sum_{{|\alpha|+|\beta|\leq N}}X(t)^{\frac{2-\gamma}{2(1-\gamma)}}(1+t)^{-\sigma_{|\alpha|+|\beta|,|\beta|}}\mathcal{\widetilde{D}}^{|\alpha|,|\beta|}_{l_0^*\ or\ l_1^*,1}(t)}+X(t){D}_{N_0}(t),
\nonumber
\end{eqnarray}
provided that the parameters $\vartheta, \varrho, N,$ and $N_0$ satisfy the conditions listed in Lemma \ref{lemma9}. Here to state briefly, we use $\mathcal{\widetilde{D}}^{|\alpha|,|\beta|}_{\ell,1}(t)$ to denote
$$(1+t)^{-1-\vartheta}\left\|w_{\ell-|\beta|,1}\partial_\beta^\alpha\{{\bf I-P}\} f\langle v\rangle\right\|^2+\left\|{w_{\ell-|\beta|,1}}\partial_\beta^\alpha\{{\bf I-P}\}f\right\|^2_\nu.$$

Without loss of generality, we only verify the estimate \eqref{E^n_{tri,j}-F^n_{tri,j}} for the term
$$(1+t)^{-\sigma_{n,j}}\sum_{|\alpha|+|\beta|=n,\atop|\alpha_1|=1, |\beta|=j,}(1+t)^{1+\vartheta}\left\|\partial^{\alpha_1} B\right\|^2_{L^\infty_x}\left\|w_{l_0^*-1-j,1}\nabla_v\partial_\beta^{\alpha-\alpha_1}{\bf \{I-P\}}f\langle v\rangle\right\|^2
$$
since the other terms can be estimated in a similar way. In such a case, Lemma \ref{Lemma1} tells us that
\[
\sum_{|\alpha_1|=1}\left\|\partial^{\alpha_1}B\right\|^2_{L^\infty_x}\lesssim
\begin{cases} X(t)(1+t)^{-\frac52-\varrho}, N_0\geq 5,\\[2mm]
X(t)(1+t)^{-2-\varrho}, N_0=4\\
\end{cases}
\]
which implies
\begin{equation}\label{3.42}
\sum_{|\alpha_1|=1}\left\|\partial^{\alpha_1}B\right\|^2_{L^\infty_x}\lesssim X(t)(1+t)^{-2-2\vartheta-\frac{2(1+\gamma)}{\gamma-2}(1+\vartheta)}
\end{equation}
if the parameters $\vartheta$ and $\varrho$ are suitably chosen such that
\begin{equation*}
\begin{cases}
0<\vartheta\leq
\frac{2\varrho\gamma-3\gamma-4\varrho-6}{8\gamma-4},\ \ \varrho\in[\frac12,\frac32),\ N_0\geq 5,\\
0<\vartheta
\leq\frac{\varrho\gamma-2\gamma-2\varrho-2}{4\gamma-2},\ \ \varrho\in(1,\frac32),\ N_0=4.
\end{cases}
\end{equation*}
Now due to
$$
\sigma_{n,j}-\sigma_{n,j-1}=\frac{2(1+\gamma)}{\gamma-2}(1+\vartheta),
$$
we can get from the estimate \eqref{3.42} that
\begin{eqnarray*}
&&(1+t)^{-\sigma_{n,j}}\sum_{|\alpha|+|\beta|=n,\atop|\alpha_1|=1, |\beta|=j,}(1+t)^{1+\vartheta}\left\|\partial^{\alpha_1} B\right\|^2_{L^\infty_x} \left\|w_{l_0^*-1-j,1}\nabla_v\partial_\beta^{\alpha-\alpha_1}{\bf \{I-P\}}f\langle v\rangle\right\|^2\\
&\lesssim&X(t)(1+t)^{-\sigma_{n,j+1}-1-\vartheta}\sum_{|\alpha|+|\beta|=n,\atop |\beta|=j+1,}\left\|w_{l_0^*-1-j,1}\partial_\beta^{\alpha}{\bf \{I-P\}}f\langle v\rangle\right\|^2,
\end{eqnarray*}
that is exactly what we wanted.

Finally, Lemma \ref{Lemma1} implies that
$$
 \left\{
 \begin{array}{rl}
 N=2N_0-1, \quad & when\ \  \varrho\in[\frac12,1],\\[2mm]
 N=2N_0,\quad & when \ \ \varrho\in(1,\frac32).
 \end{array}
 \right.
$$
Thus the proof of Lemma \ref{lemma9} is complete.
\end{proof}

Now we turn to prove Theorem \ref{Th1.1}. To this end, recall the definition of the $X(t)-$norm. Lemma \ref{lemma9} tells that for the local solution $[f(t,x,v), E(t,x), B(t,x)]$ to the Cauchy problem (\ref{f}) and (\ref{f-initial}) defined on the time interval $[0,T]$ for some $0<T\leq +\infty$, if
\begin{equation*}
X(t)\leq M,\quad \forall\, t\in[0,T],
\end{equation*}
then there exists a sufficiently small positive constant $\delta_0>0$ such that if
\begin{equation*}
M\leq \delta^2_0,
\end{equation*}
there exists a positive constant $\overline{C}>0$ such that
\begin{equation*}
X(t)\leq \overline{C}^2Y_0^2
\end{equation*}
holds for all $0\leq t\leq T$.

Thus if the initial perturbation $Y_0$ is assumed to be sufficiently small such that
\begin{equation*}
Y_0\leq \frac{\delta_0}{\overline{C}},
\end{equation*}
then the global existence follows by combining the local solvability result with the continuation argument in the usual way. This completes the proof of Theorem \ref{Th1.1}.\qed

\subsection{The proof of Theorem \ref{Th1.2}}

Based on Theorem \ref{Th1.1} and by taking $k=0,1,2,\cdots, N_0-2$, we can get firstly from  Lemma \ref{Lemma1} that
\begin{equation*}
\mathcal{E}^k_{N_0}(t)\lesssim Y^2_0(1+t)^{-(k+\varrho)},
\end{equation*}
that gives (\ref{TH2-1}).

{As to (\ref{TH2-2}), as long as one takes $l_0$ and $l^*$ respectively as $$l_0= l_{0,0}=l_{0,1}\geq\max\left\{\chi_{k\leq2}(l_{0,k}+3k-3), {l_1^*}+\frac{5}{2}\right\},$$
and $l^*=\frac{k+2}{2}-\frac{\widetilde{l}}\gamma$ in Theorem 1.1, then (\ref{TH2-2}) follows from Lemma \ref{lemma2}.}

Finally, to prove (\ref{TH2-4}), we have by the interpolation method with respect to space derivative $x$ for $N_0+1\leq|\alpha|\leq N-1$ and by using the time decay of $\left\|\nabla^{N_0}f\right\|$ and the bound of $\left\|\nabla^{N}f\right\|$ that
\begin{equation*}
\begin{split}
\left\|\partial^\alpha f\right\|^2\lesssim& \left\|\nabla^Nf\right\|^{\frac{|\alpha|-N_0}{N-N_0}}
\left\|\nabla^{N_0}f\right\|^{\frac{N-|\alpha|}{N-N_0}}
\lesssim Y^2_0(1+t)^{-\frac{(N-|\alpha|)(N_0-2+\varrho)}{N-N_0}}.
\end{split}
\end{equation*}
This is \eqref{TH2-4} and the proof of Theorem \ref{Th1.2} is complete.\qed

\section{Appendix}

We will complete the proofs of some lemmas and estimates used in the previous section. \\

\noindent 4.1 {\bf The proof of the key estimate  in Lemma \ref{Lemma1}.} First of all, the following lemmas are for proving \eqref{Lemma1-1}.

\begin{lemma}\label{B-lemma}
Assume $-3<\gamma<-1$, $N_0$ and $N$ satisfying \eqref{N-N_0} and $n\geq \frac23N_0-\frac53$,
there exist a positive integer $\overline{m}$ satisfying $N_0+1\leq\overline{m}\leq N-1$ and a sufficiently large number $\widetilde{l}$, which both depend only $\gamma$ and $N_0$,
such that
when $1\leq k\leq N_0-2$,
\begin{equation}\label{B-k}
\begin{aligned}
&\left|\left(\nabla^k((v\times B)\cdot\nabla_vf),\nabla^kf\right)\right|
\lesssim\max\left\{\mathcal{E}_{\overline{m},\overline{m},-\gamma}(t),
\mathcal{\overline{E}}_{N_0,N_0-\frac{\widetilde{l}}\gamma,-\gamma}(t) \right\}\\
&\times\left(\left\|\nabla^{k+1}B\right\|^2
+\left\|\nabla^k\{{\bf I-P}\}f\right\|^2_\nu+\|\nabla^{k+1}f\|_{\nu}^2\right)+\varepsilon\left(\left\|\nabla^k\{{\bf I-P}\}f\right\|^2_{\nu}+\left\|\nabla^{k+1}f\right\|_{\nu}^2\right),
\end{aligned}
\end{equation}
when $k=N_0-1$, it holds that
\begin{equation}\label{B-N_0-1}
\begin{split}
\left|\left(\nabla^k\left((v\times B)\cdot\nabla_vf\right),\nabla^kf\right)\right|\lesssim&\max\left\{\mathcal{E}_{\overline{m},\overline{m},-\gamma}(t),
\mathcal{\overline{E}}_{N_0,N_0-\frac{\widetilde{l}}\gamma,-\gamma}(t) \right\}\\
&\times\left(\left\|\nabla^{N_0-1}B\right\|^2
+\left\|\nabla^{N_0-1}f\right\|_{\nu}^2\right)
+\varepsilon\left\|\nabla^{N_0-1}f\right\|_{\nu}^2,
\end{split}
\end{equation}
and as for $k=N_0$, it follows that
\begin{equation}\label{B-N_0}
\begin{aligned}
\sum_{k=N_0}\left|\left(\nabla^k((v\times B)\cdot\nabla_vf),\nabla^kf\right)\right|\lesssim&
\max\left\{\mathcal{E}_{N_0+n}(t),
\mathcal{E}_{\overline{m},\overline{m},-\gamma}(t),
\mathcal{\overline{E}}_{N_0,N_0-\frac{\widetilde{l}}\gamma,-\gamma}(t)\right\}\\
&\times
\left(\left\|\nabla^{N_0-1}B\right\|^2+\left\|\nabla^{N_0-2}{\bf \{I-P\}}f\right\|_{\nu}^2+\left\|\nabla^{N_0-1}f\right\|_{\nu}^2
+\left\|\nabla^{N_0}f\right\|_{\nu}^2\right)\\
&
+\varepsilon\left(\left\|\nabla^{N_0-2}{\bf\{I-P\}}f\right\|^2_{H^1_xL^2_{\nu}}+\left\|\nabla^{N_0}f\right\|^2_{\nu}\right).
\end{aligned}
\end{equation}
\end{lemma}
\begin{proof}
To obtain \eqref{B-k}, by using the macro-micro decomposition, one has
\begin{equation*}
\begin{aligned}
&\left|\left(\nabla^k((v\times B)\cdot\nabla_vf),\nabla^kf\right)\right|\lesssim \sum_{1\leq j\leq k}\left|\left(\left(v\times\nabla^{j}B\right)\cdot\nabla_v\nabla^{k-j}f,\nabla^kf\right)\right|\\
&=\underbrace{ \sum_{1\leq j\leq k}\left|\left(\left(v\times\nabla^{j}B\right)\cdot\nabla_v\nabla^{k-j}{\bf P}f,\nabla^kf\right)\right|}_{I_{B,1}}+\underbrace{\sum_{1\leq j\leq k}
\left|\left(\left(v\times\nabla^{j}B\right)\cdot\nabla_v\nabla^{k-j}{\bf\{I- P\}}f,\nabla^k{\bf P}f\right)\right|}_{I_{B,2}}\\
&\qquad+\underbrace{\sum_{1\leq j\leq k}
\left|\left(\left(v\times\nabla^{j}B\right)\cdot\nabla_v\nabla^{k-j}{\bf\{I- P\}}f,\nabla^k{\bf\{I- P\}}f\right)\right|}_{I_{B,3}}.
\end{aligned}
\end{equation*}
Applying the interpolation method with respect to space derivative $x$, so we
deduce from Lemma \ref{lemma2.2}  that
\begin{equation*}
\begin{aligned}
I_{B,1}+I_{B,2}
\lesssim&\sum_{1\leq j\leq k}\left\|\nabla^jB\right\|_{L^3_x}
\left\|\nabla^{k-j}\left(\mu^{\delta}f\right)\right\|\left\|\nabla^{k+1}\left(\mu^{\delta}f\right)\right\|\\
\lesssim&\sum_{1\leq j\leq k}\left\|\Lambda^{-\frac 12}B\right\|^{\frac{2k-2j+1}{2k+3}} \left\|\nabla^{k+1}B\right\|^{\frac{2j+2}{2k+3}}
\left\|\Lambda^{-\frac 12}\left(\mu^{\delta}f\right)\right\|^{\frac{2j+2}{2k+3}} \left\|\nabla^{k+1}\left(\mu^{\delta}f\right)\right\|^{\frac{2k-2j+1}{2k+3}}
\left\|\nabla^{k+1}\left(\mu^{\delta}f\right)\right\|\\
\lesssim&\mathcal{\overline{E}}_{0,0,-\gamma}(t)
\left(\left\|\nabla^{k+1}B\right\|^2+\left\|\nabla^{k+1}f\right\|_{\nu}^2\right)
+\varepsilon\left\|\nabla^{k+1}f\right\|_{\nu}^2.
\end{aligned}
\end{equation*}
As for $I_{3,3}$, when $j=k$, taking $L^6-L^3-L^2$ type inequality and applying Lemma \ref{lemma2.2}, one has
\begin{equation*}
\begin{aligned}
I_{B,3}
\lesssim&\left\|\nabla^{k}B\right\|_{L^6_x}\left\|\nabla_v{\bf\{I- P\}}f\langle v\rangle^{1-\frac{\gamma}2}\right\|_{L^2_v L^3_x}\left\|\nabla^k{\bf\{I- P\}}f\langle v\rangle^{\frac{\gamma}2}\right\|_{L^2_vL^2_x} \\
\lesssim&\mathcal{\overline{E}}_{2,\frac52-\frac {1}\gamma,-\gamma}(t)\left\|\nabla^{k+1}B\right\|^2+\varepsilon\left\|\nabla^k\{{\bf I-P}\}f\right\|^2_\nu.
\end{aligned}
\end{equation*}
While for the case $1\leq j\leq k-1$, by the similar virtue of the estimates on $I_{B,3}$ for $j=k$, one also has
\begin{equation*}
\begin{aligned}
I_{B,3}
\lesssim&\sum_{1\leq j\leq k-1}\left\|\nabla^{j}B\right\|_{L^\infty_x}\left\|\nabla_v\nabla^{k-j}{\bf\{I- P\}}f\right\|\left\|\nabla^k{\bf\{I- P\}}f\langle v\rangle\right\| \\
\lesssim&\sum_{1\leq j\leq k-1}\left\|\Lambda^{-{\varrho}}B\right\|^{\frac{2k-2j-1}{2(k+1+\varrho)}}\left\|\nabla^{k+1}B\right\|^{\frac{2j+2\varrho+3}{2(k+1+\varrho)}}
\left\|\nabla^{m_{1j}+1}_v\nabla^{k-j}{\bf\{I- P\}}f\right\|^{\frac1{m_{1j}+1}}\\{}
&\times\left\|\Lambda^{-{\varrho}}{\bf\{I- P\}}f\right\|^{\frac{m_{1j}j}{(m_{1j}+1)(k+\varrho)}}\left\|\nabla^{k}{\bf\{I- P\}}f\right\|^{\frac{m_{1j}(k-j+\varrho)}{(m_{1j}+1)(k+\varrho)}}\left\|\nabla^k{\bf\{I- P\}}f\langle v\rangle\right\| \\{}
\lesssim&\sum_{1\leq j\leq k-1}\left\|\Lambda^{-{\varrho}}B\right\|^{\frac{2k-2j-1}{2(k+1+\varrho)}}\left\|\nabla^{k+1}B\right\|^{\frac{2j+2\varrho+3}{2(k+1+\varrho)}}
\left\|\nabla^{m_{1j}+1}_v\nabla^{k-j}{\bf\{I- P\}}f\right\|^{\frac1{m_{1j}+1}}\\{}
&\times\left\|\Lambda^{-{\varrho}}{\bf\{I- P\}}f\right\|^{\frac{m_{1j}j}{(m_{1j}+1)(k+\varrho)}}\left\|\nabla^{k}{\bf\{I- P\}}f\right\langle v\rangle^{\frac\gamma2}\|^{\frac{m_{1j}(k-j+\varrho)\beta_j}{(m_{1j}+1)(k+\varrho)}}\left\|\nabla^k{\bf\{I- P\}}f\langle v\rangle^{\frac\gamma2}\right\|^{\beta_j} \\{}
&\times\left\|\nabla^{k}{\bf\{I- P\}}f\right\langle v\rangle^{l_{1j}}\|^{\frac{m_{1j}(k-j+\varrho)(1-\beta_j)}{(m_{1j}+1)(k+\varrho)}}\left\|\nabla^k{\bf\{I- P\}}f\langle v\rangle^{l_{2j}}\right\|^{1-\beta_j} \\{}
\lesssim&\max\left\{\mathcal{\overline{E}}_{k+m_{1},1+m_{1},-\gamma}(t),\mathcal{\overline{E}}_{k,-\frac{\hat{l}_{2}}\gamma,-\gamma}(t)\right\}\left(\left\|\nabla^{k+1}B\right\|^2+\left\|\nabla^k\{{\bf I-P}\}f\right\|^2_\nu\right)+\varepsilon\left\|\nabla^k\{{\bf I-P}\}f\right\|^2_\nu.
\end{aligned}
\end{equation*}
Here we have used the fact that there exists a positive constant $\beta_j\in(0,1)$ such that
$$\frac{2j+2\varrho+3}{2(k+1+\varrho)}+\frac{m_{1j}(k-j+\varrho)\beta_j}{(m_{1j}+1)(k+\varrho)}+\beta_j=2$$
holds for $1\leq j\leq k-1$. A necessary and sufficient condition to guarantee the existence of such $\beta_j$ is
$$
\frac{2j+2\varrho+3}{2(k+1+\varrho)}+\frac{m_{1j}(k-j+\varrho)}{(m_{1j}+1)(k+\varrho)}+1>2,\quad 1\leq j\leq k-1,
$$
from which one can deduce that $m_{1j}> \frac{2k^2+2\varrho k-2jk-k-2j\varrho-\varrho}{2k\varrho+3k+2\varrho^2+3\varrho-2j}$ holds for $1\leq j\leq k-1$. Noticing that $\frac12\leq\varrho<\frac 32$, it is easy to see that we can take
$$
m_1=\max_{1\leq j\leq k-1}m_{1j}=\frac{2k}{2\varrho+3}+\frac{2\varrho-3}{2\varrho+3}>\frac{2k^2+2\varrho k-3k-3\varrho}{2k\varrho+3k+2\varrho^2+3\varrho-2}.
$$
Consequently, $m_{1j}+1+k-j\leq k+m_{1}=\frac{2\varrho+5}{2\varrho+3}k+\frac{2\varrho-3}{2\varrho+3}\leq \frac{2\varrho+5}{2\varrho+3}N_0-\frac{2\varrho+13}{2\varrho+3}$ with $N_0\geq 4$.

Moreover, since $\hat{l}_{1j}$ and $\hat{l}_{2j}$ satisfy respectively $\frac\gamma2\beta_j+\hat{l}_{1j}(1-\beta_j)=0$ and $\frac\gamma2{\beta_j}+\hat{l}_{2j}(1-\beta_j)=1$ with $0<\beta_j<1$, one can deduce that, $\hat{l}_{1j}=\frac\gamma2-\frac{\gamma}{2(1-\beta_j)}$ and $\hat{l}_{2j}=\frac\gamma2-\frac{\gamma-2}{2(1-\beta_j)}$ from which we can see that $\hat{l}_{2j}>\hat{l}_{1j}$
where $$\beta_j=\frac{(4k+1+2\varrho-2j)(m_{1j}+1)(k+\varrho)}{(k+2+2\varrho)(k+\varrho+2km_{1j}+2\varrho m_{1j}-jm_{1j})}.$$
Here we take $\hat{l}_{2}=\displaystyle\max_{1\leq j\leq k-1}\left\{\hat{l}_{2j}\right\}$.

Consequently, if we take $\overline{m}=k+m_1$ and $\widetilde{l}\geq \max\left\{\hat{l}_2,\frac12-\frac1\gamma\right\}$, \eqref{B-k} follows by collecting the above estimates.
As well as case $k\leq N_0-2$,
for $k=N_0-1$, there exist a positive integer $\overline{m}$ and a sufficiently large number $\widetilde{l}$
such that
\begin{equation*}
\begin{split}
\sum_{k=N_0-1}\left|\left(\nabla^k((v\times B)\cdot\nabla_vf),\nabla^kf\right)\right|\lesssim&\max\left\{\mathcal{E}_{\overline{m},\overline{m},-\gamma}(t),
\mathcal{\overline{E}}_{N_0,N_0-\frac{\widetilde{l}}\gamma,-\gamma}(t) \right\}\\
&\times\left(\left\|\nabla^{N_0-1}B\right\|^2
+\left\|\nabla^{N_0-1}f\right\|_{\nu}^2\right)
+\varepsilon\left\|\nabla^{N_0-1}f\right\|_{\nu}^2.
\end{split}
\end{equation*}
With regard to the case $k=N_0$, compared with the above cases, we only notice that if we take $n>\frac 2 3 N_0-\frac53$,

\begin{equation*}
\begin{aligned}
&\left|\left(\left(v\times\nabla^{N_0}B\right)\cdot\nabla_v\{{\bf I-P}\}f, \nabla^{N_0}f\right)\right|\\
\lesssim&\left\|\nabla^{N_0+n}B\right\|^{\frac{1}{1+n}}\left\|\nabla^{N_0-1}B\right\|^{\frac{n}{1+n}}\left\|\nabla_x\nabla^{m_2+1}_v{\bf\{I-P\}}f\right\|^{\frac{1}{1+m_2}}_{H^1_xL^2_v}\left\|\Lambda^{-\varrho}{\bf\{I-P\}}f\right\|^{\frac{m_2(N_0-3)}{(1+m_2)(N_0-2+\varrho)}}_{H^1_xL^2_v}\\{}
&\times\left\|\nabla^{N_0-2}{\bf\{I-P\}}f\langle v\rangle^{\hat{l}_{3}}\right\|
^{\frac{m_2\beta(1+\varrho)}{(1+m_2)(N_0-2+\varrho)}}_{H^1_xL^2_v}\left\|\nabla^{N_0-2}{\bf\{I-P\}}f\langle v\rangle^{\frac\gamma2}\right\|^{\frac{m_2
(1-\beta)(1+\varrho)}{(1+m_2
)(N_0-2+\varrho)}}_{H^1_xL^2_v}\\{}
&\times\left\|\nabla^{N_0} f\langle v\rangle^{\hat{l}_{4}}\right\|^{1-\beta}\left\|\nabla^{N_0} f\langle v\rangle^{\frac\gamma2}\right\|^{\beta}\\{}
\lesssim&\max\left\{\mathcal{E}_{N_0+n}(t),\mathcal{E}_{3+m_2,3+m_2,-\gamma}(t),\mathcal{\overline{E}}_{N_0,N_0-\frac{\hat{l}_{4}}\gamma,-\gamma}(t)\right\}\left(\|\nabla^{N_0-1}B\|^2+\left\|\nabla^{N_0-2}{\bf\{I-P\}}f\right\|^2_{H^1_xL^2_{\nu}}\right)\\{}
&+\varepsilon\left(\left\|\nabla^{N_0-2}{\bf\{I-P\}}f\right\|^2_{H^1_xL^2_{\nu}}+\left\|\nabla^{N_0}f\right\|^2_{\nu}\right).
\end{aligned}
\end{equation*}
Here we need to ask $\frac{n}{1+n}+\frac{m_2(1+\varrho)\beta}{(1+m_2)(N_0-2+\varrho)}+\beta=2$ which deduce that $m_2>\frac{N_0-2+\varrho}{\varrho n+n+3-N_0},$ we can get $\hat{l}_{3}=\frac{\gamma}2-\frac{\gamma}{2(1-\beta)}$ and $\hat{l}_{4}=\frac{\gamma}2-\frac{\gamma-2}{2(1-\beta)}$ from $
\frac{\gamma}2\cdot\beta+\hat{l}_{3}(1-\beta)=0 $ and
$\frac{\gamma}2\cdot\beta+\hat{l}_{4}(1-\beta)=1 $.
 We can choose $m_2$ suitably such that $3+m_2\leq 3+\frac{N_0-2+\varrho}{(\varrho+1)n+3-N_0}$.
 The other terms can be estimated as well as \eqref{B-k}.
Consequently, if we take suitable numbers $\overline{m}$ and $\widetilde{l}$, we also have
\begin{equation*}
\begin{aligned}
\sum_{k=N_0}\left|\left(\nabla^k((v\times B)\cdot\nabla_vf),\nabla^kf\right)\right|\lesssim&
\max\left\{\mathcal{E}_{N_0+n}(t),
\mathcal{E}_{\overline{m},\overline{m},-\gamma}(t),
\mathcal{\overline{E}}_{N_0,N_0-\frac{\widetilde{l}}\gamma,-\gamma}(t)\right\}\\{}
&\times
\left(\|\nabla^{N_0-1}B\|^2+\|\nabla^{N_0-2}{\bf \{I-P\}}f\|_{\nu}^2+\|\nabla^{N_0-1}f\|_{\nu}^2+\|\nabla^{N_0}f\|_{\nu}^2\right)\\{}
&
+\varepsilon\left(\left\|\nabla^{N_0-2}{\bf\{I-P\}}f\right\|^2_{H^1_xL^2_{\nu}}+\left\|\nabla^{N_0}f\right\|^2_{\nu}\right).
\end{aligned}
\end{equation*}
Thus we have completed the proof of this lemma.
\end{proof}
By repeating the argument used to prove Lemma \ref{B-lemma}, we can also obtain that
\begin{lemma}
Under the assumptions of Lemma \ref{B-lemma}, we have estimates on the terms containing $E$ and $\Gamma(f,f)$ as follows. For $k\leq N_0-2$, it holds that
\begin{equation}\label{E-k}
\begin{aligned}
&\left|\left(\nabla^k(v\cdot E f),\nabla^kf\right)\right|+\left|\left(\nabla^k(E\cdot\nabla_vf),\nabla^kf\right)\right|+{\left|\left(\nabla^k\Gamma(f,f),\nabla^kf\right)\right|}\\
\lesssim&\max\left\{\mathcal{E}_{\overline{m},\overline{m},-\gamma}(t),
\mathcal{\overline{E}}_{N_0,N_0-\frac{\widetilde{l}}\gamma,-\gamma}(t) \right\}\left(\left\|\nabla^{k+1}E\right\|^2
+\left\|\nabla^k\{{\bf I-P}\}f\right\|^2_\nu+\|\nabla^{k+1}f\|_{\nu}^2\right)\\{}
&+\varepsilon\left(\left\|\nabla^k\{{\bf I-P}\}f\right\|^2_{\nu}+\left\|\nabla^{k+1}f\right\|_{\nu}^2\right).
\end{aligned}
\end{equation}
For  $k=N_0-1$, it holds that
\begin{equation}\label{E-N_0-1}
\begin{split}
&\left|\left(\nabla^k(v\cdot E f),\nabla^kf\right)\right|+\left|\left(\nabla^k(E\cdot\nabla_vf),\nabla^kf\right)\right|
+{\left|\left(\nabla^k\Gamma(f,f),\nabla^kf\right)\right|}\\
\lesssim&\max\left\{\mathcal{E}_{\overline{m},\overline{m},-\gamma}(t),
\mathcal{\overline{E}}_{N_0,N_0-\frac{\widetilde{l}}\gamma,-\gamma}(t) \right\}\left(\left\|\nabla^{N_0-1}E\right\|^2
+\left\|\nabla^{N_0-1}f\right\|_{\nu}^2\right)
+\varepsilon\left\|\nabla^{N_0-1}f\right\|_{\nu}^2.
\end{split}
\end{equation}
For $k=N_0$, it holds that
\begin{equation}\label{E-N_0}
\begin{aligned}
&\left|\left(\nabla^k(v\cdot E f),\nabla^kf\right)\right|+\left|\left(\nabla^k(E\cdot\nabla_vf),\nabla^kf\right)\right|
+{\left|\left(\nabla^k\Gamma(f,f),\nabla^kf\right)\right|}\\
\lesssim&
\max\left\{\mathcal{E}_{N_0+n}(t),
\mathcal{E}_{\overline{m},\overline{m},-\gamma}(t),
\mathcal{\overline{E}}_{N_0,N_0-\frac{\widetilde{l}}\gamma,-\gamma}(t)\right\}
\left(\left\|\nabla^{N_0-1}E\right\|^2+\left\|\nabla^{N_0-2}{\bf \{I-P\}}f\right\|_{\nu}^2\right.\\
&\left.+\left\|\nabla^{N_0-1}f\right\|_{\nu}^2+\left\|\nabla^{N_0}f\right\|_{\nu}^2\right)
+\varepsilon\left(\left\|\nabla^{N_0-2}{\bf\{I-P\}}f\right\|^2_{H^1_xL^2_{\nu}}+\left\|\nabla^{N_0}f\right\|^2_{\nu}\right).
\end{aligned}
\end{equation}
\end{lemma}

\begin{remark}\label{m-l}
Comparing the proofs of the above two lemmas, we all take suitably numbers $\overline{m}$ satisfying $N_0-1\leq\overline{m}\leq N-1$ and $\widetilde{l}$. In fact, by
complex calculation as well as $\eqref{B-N_0}$, we obtain
 \begin{equation*}
\begin{aligned}
\overline{m}=\max\bigg\{&\chi_{N_0\geq 4}\left\{\frac{2\varrho+5}{2\varrho+3}N_0
-\frac{6}{2\varrho+3}\right\},\ \chi_{N_0\geq 6}\left\{\frac{2\varrho+7}{2\varrho+3}N_0
-\frac{8\varrho+38}{2\varrho+3}\right\},\\ &\chi_{N_0\geq 7}\left\{\frac{2\varrho+9}{2\varrho+5}N_0
-\frac{4\varrho+36}{2\varrho+5}\right\},\
\chi_{N_0\geq 8}\left\{\frac{2\varrho+11}{2\varrho+7}N_0-\frac{4\varrho+44}{2\varrho+9}\right\}\bigg\},
\end{aligned}
\end{equation*}
Thus we can choose $\overline{m}=N-1$ without generality if $N$ satisfies $\eqref{N-N_0}$. Since the computation of accurate value of $\widetilde{l}$ is too complicated but standard, we claim that there exists a finite number
$\widetilde{l}$ satisfying the above three lemmas.
\end{remark}

Based on the above three lemmas and Remark \ref{m-l}, it is straightforward to obtain
\begin{lemma}\label{lemma4.3}
Let $N_0$ and $N$ satisfying \eqref{N-N_0}, then there exists a positive constant $\widetilde{l}$, which depends only on $N_0$, $\varrho$ and $\gamma$,  such that:
\begin{itemize}
\item[(1).] For $k=0,1,\cdots,N_0-2$, it holds that
\begin{equation}\label{Lemma4.3-1}
\begin{aligned}
&\frac{d}{dt}\left(\left\|\nabla^kf\right\|^2+\left\|\nabla^k(E,B)\right\|^2\right)
+\left\|\nabla^k\{{\bf I-P}\}f\right\|^2_{\nu}\\{}
\lesssim&\max\left\{\mathcal{E}_{N-1,N-1,-\gamma}(t),
\mathcal{\overline{E}}_{N_0,N_0-\frac{\widetilde{l}}\gamma,-\gamma}(t)\right\}
\left(\left\|\nabla^{k+1}(E,B)\right\|^2+\left\|\nabla^{k}\{{\bf I-P}\}f\right\|^2_{\nu}\right.\\{}
&\left.+\left\|\nabla^{k+1}f\right\|^2_{\nu}\right)
+\varepsilon\left(\left\|\nabla^k\{{\bf I-P}\}f\right\|^2_{\nu}+\left\|\nabla^{k+1}f\right\|_{\nu}^2\right).
\end{aligned}
\end{equation}
\item [(2).] If $k=N_0-1$, it follows that
\begin{equation}\label{Lemma4.3-2}
\begin{aligned}
&\frac{d}{dt}\left(\left\|\nabla^{N_0-1}f\right\|^2+\left\|\nabla^{N_0-1}(E,B)\right\|^2\right)
+\left\|\nabla^{N_0-1}\{{\bf I-P}\}f\right\|_{\nu}^2\\{}
\lesssim&\max\left\{\mathcal{E}_{N-1,N-1,-\gamma}(t),
\mathcal{\overline{E}}_{N_0,N_0-\frac{\widetilde{l}}\gamma,-\gamma}(t)\right\}
\left(\left\|\nabla^{N_0-1}(E,B)\right\|^2\right.\\{}
&\left.+\left\|\nabla^{N_0-2}\{{\bf I-P}\}f\right\|_{\nu}^2
+\left\|\nabla^{N_0-1}f\right\|_{\nu}^2\right)+\varepsilon\left\|\nabla^{N_0-1}f\right\|_{\nu}^2.
\end{aligned}
\end{equation}
\item[(3).] As for $k=N_0\geq 4$, if we let $n>\frac23N_0-\frac53$, one has
\begin{equation}\label{Lemma4.3-3}
\begin{aligned}
&\frac{d}{dt}\left(\left\|\nabla^{N_0}f\right\|^2+\left\|\nabla^{N_0}(E,B)\right\|^2\right)
+\left\|\nabla^{N_0}\{{\bf I-P}\}f\right\|^2_{\nu}\\{}
\lesssim&\max\left\{\mathcal{E}_{N_0+n}(t),
\mathcal{E}_{N-1,N-1,-\gamma}(t),
\mathcal{\overline{E}}_{N_0,N_0-\frac{\widetilde{l}}\gamma,-\gamma}(t)\right\}
\left(\left\|\nabla^{N_0-1}(E,B)\right\|^2+\left\|\nabla^{N_0-2}{\bf \{I-P\}}f\right\|_{\nu}^2\right.\\{}
&\left.+\left\|\nabla^{N_0-1}f\right\|^2_{H^1_xL^2_{\nu}}\right)
+\varepsilon\left(\left\|\nabla^{N_0-2}{\bf\{I-P\}}f\right\|^2_\nu
+\left\|\nabla^{N_0-1}f\right\|^2_{H^1_xL^2_{\nu}}\right).
\end{aligned}
\end{equation}
\end{itemize}
\end{lemma}

\begin{proof}
To prove (\ref{Lemma4.3-1}), we apply $\nabla^k$ to $(\ref{f})$,
multiply the resulting identity by $\nabla^kf$, and further integrate it with respect to $x$ and $v$ over $\mathbb{R}_x^3\times\mathbb{R}_v^3$. Then, for $k\leq N_0-2$, (\ref{Lemma4.3-1}) follows by recalling \eqref{B-k}, \eqref{E-k} and the coercive property of the linear operator $L$.
Similarly, for $k=N_0-1\geq 2$, \eqref{B-N_0-1} and \eqref{E-N_0-1} imply (\ref{Lemma4.3-2}).
Regarding the last case $k=N_0\geq 4$, one has (\ref{Lemma4.3-3}) by combing \eqref{B-N_0} and \eqref{E-N_0}. Thus the proof of Lemma \ref{lemma4.3} is complete.
\end{proof}

The next lemma is concerned with the macro dissipation $\mathcal{D}_{N,mac}(t)$ defined by
$$
\mathcal{D}_{N,mac}(t)\sim\left\|\nabla_x(a_\pm,b,c)\right\|^2_{H^{N-1}_x}
+\|a_+-a_-\|^2+\|E\|^2_{H^{N-1}_x}+\|\nabla_xB\|^2_{H^{N-2}_x}.
$$

\begin{lemma}\label{Lemma4.4}
For the macro dissipation estimates on $f(t,x,v)$, we have the following results:
\begin{itemize}
\item[(i).] For $k=0,1,2\cdots,N_0-2$, there exist interactive energy functionals $G^k_f(t)$ satisfying
\[
G^k_f(t)\lesssim \left\|\nabla^k(f,E,B)\right\|^2+\left\|\nabla^{k+1}(f,E,B)\right\|^2+\left\|\nabla^{k+2}E\right\|^2
\]
such that
\begin{equation*}
\begin{split}
&\frac{d}{dt}G^k_f(t)+\left\|\nabla^k(E,a_+-a_-)\right\|_{H^1_xL^2_v}^2+\left\|\nabla^{k+1}({\bf P}f,B)\right\|^2\\{}
\lesssim&\mathcal{\overline{E}}_{N_0-1,0,-\gamma}(t)\left(\left\|\nabla^{k+1}(E,B)\right\|^2+\left\|\nabla^{k+1}f\right\|^2_{\nu}\right)
+\left\|\nabla^k\{{\bf I-P}\}f\right\|^2_{\nu}\\{}
&
+\left\|\nabla^{k+1}\{{\bf I-P}\}f\right\|^2_{\nu}
+\left\|\nabla^{k+2}\{{\bf I-P}\}f\right\|^2_{\nu};
\end{split}
\end{equation*}
\item[(ii).] For $k=N_0-1$, there exists an interactive energy functional $G^{N_0-1}_f(t)$ satisfying
$$
G^{N_0-1}_f(t)\lesssim\left\|\nabla^{N_0-2}(f,E,B)\right\|^2+\left\|\nabla^{N_0-1}(f,E,B)\right\|^2
+\left\|\nabla^{N_0}(f,E)\right\|^2
$$
such that
\begin{equation*}
\begin{split}
&\frac{d}{dt}G^{N_0-1}_f(t) +\left\|\nabla^{N_0-2}(E,a_+-a_-)\right\|_{H^1_xL^2_v}^2+\left\|\nabla^{N_0-1}B
\right\|^2+\left\|\nabla^{N_0}{\bf P}f\right\|^2\\{}
\lesssim&\mathcal{\overline{E}}_{N_0,0,-\gamma}(t)\left(\left\|\nabla^{N_0-1}(E,B)\right\|^2
+\left\|\nabla^{N_0-1}f\right\|^2_{\nu}\right)+\left\|\nabla^{N_0-2}\{{\bf I-P}\}f\right\|^2_{\nu}\\{}
&+\left\|\nabla^{N_0-1}\{{\bf I-P}\}f\right\|^2_{\nu}
+\left\|\nabla^{N_0}\{{\bf I-P}\}f\right\|^2_{\nu};
\end{split}
\end{equation*}
\item[(iii).] There exists an interactive energy functional $\mathcal{E}^{int}_N(t)$ satisfying
$$
\mathcal{E}^{int}_N(t)\lesssim\sum_{|\alpha|\leq N}\left\|\partial^\alpha(f,E,B)\right\|^2$$
such that
\begin{equation*}
\frac{d}{dt}\mathcal{E}^{int}_N(t)+\mathcal{D}_{N,mac}(t)\lesssim\sum_{|\alpha|\leq N}\left\|\partial^\alpha\{{\bf I-P}\}f\right\|_{\nu}^2+\mathcal{E}_N(t)\mathcal{D}_N(t)
\end{equation*}
holds for any $t\in[0,T]$.
\end{itemize}
\end{lemma}

\begin{proof}
Since the procedure of the proof is almost the same as the proof of Lemma 3.5 in \cite[page 3742]{Lei-Zhao-JFA-2014},
we omit it for brevity.
\end{proof}

\medskip

\noindent 4.2 {\bf The estimate in the negative indexed space.} Our first result in this subsection is concerned with the estimate on $\|[f,E,B](t)\|_{\dot{H}^{-\varrho}}$.
\begin{lemma}\label{LemmaF}
For $\varrho\in\ [\frac12, \frac32)$, it holds that
\begin{equation}\label{Lemma4.1-1}
\begin{aligned}
&\frac{d}{dt}\left(\left\|\Lambda^{-\varrho}f\right\|^2+\left\|\Lambda^{-\varrho}(E,B)\right\|^2\right)
+\left\|\Lambda^{-\varrho}\{{\bf I-P}\}f\right\|_{\nu}^2\\{}
\lesssim&\left(\mathcal{\overline{E}}_{0,0,-\gamma}(t)\right)^{1/2}\left(\left\|\Lambda^{\frac34-\frac \varrho2}(E,B)\right\|^2
+\left\|\Lambda^{\frac34-\frac \varrho2}f\right\|^2_{\nu}\right)
+\mathcal{\overline{E}}_{2,\frac 32-\frac{1}{\gamma},-\gamma}(t)\left\|\Lambda^{\frac32-\varrho}(f,E,B)\right\|^2.
\end{aligned}
\end{equation}
\end{lemma}
\begin{proof}
We have by taking Fourier transform of (\ref{f}) with respect to $x$,  multiplying the resulting identity by $|\xi|^{-2{s}}\overline{\hat{f}}_\pm$ with $\overline{\hat{f}}_\pm$ being the complex conjugate of $\hat{f}_\pm$, and integrating the final result with respect to $\xi$ and $v$ over $\mathbb{R}^3_\xi\times\mathbb{R}^3_v$ that
\begin{eqnarray}\label{3.3}
&&\left(\partial_t\hat{f}_\pm+ v  \cdot\mathcal{F}[\nabla_xf_\pm]\pm\mathcal{F}[(E+v\times B)\cdot\nabla_{ v  }f_\pm]\mp\frac 1 2 v  \cdot\mathcal{F}[E f_\pm]\mp\hat{E}\cdot v {\bf \mu}^{\frac12}+\mathcal{F}[{ L}_\pm f]-\mathcal{F}[{ \Gamma}_\pm(f,f)]\mid|\xi|^{-2\varrho}\hat{f}\right)\nonumber\\
&=&0.
\end{eqnarray}
Recall that throughout this paper, $\mathcal{F}[g](t,\xi,v)=\hat{g}(t,\xi,v)$ denotes the Fourier transform of $g(t,x,v)$ with respect to $x$.

(\ref{3.3}) together with Lemma \ref{Lemma L} yield
\begin{equation}\label{4.6}
\begin{aligned}
&\frac{d}{dt}\left(\left\|\Lambda^{-\varrho}f\right\|^2+\left\|\Lambda^{-\varrho}(E,B)\right\|^2\right)+\left\|\Lambda^{-\varrho}\{{\bf I-P}\}f\right\|_{\sigma}^2\\{}
\lesssim&\underbrace{\sum_\pm\left|\left(\mathcal{F}[E\cdot\nabla_{ v  }f_\pm]\mid|\xi|^{-2\varrho}\hat{f}_\pm\right)\right|}_{I_1}+\underbrace{\sum_\pm\left|\left(\mathcal{F}[v\times B\cdot\nabla_{ v  }f_\pm]\mid|\xi|^{-2\varrho}\hat{f}_\pm\right)\right|}_{I_2}\\{}
&+\underbrace{\sum_\pm\left|\left(v  \cdot\mathcal{F}[E f_\pm]\mid|\xi|^{-2\varrho}\hat{f}_\pm\right)\right|}_{I_3}+\underbrace{\sum_\pm\left|\left(\mathcal{F}[{ \Gamma}_\pm(f,f)]\mid|\xi|^{-2\varrho}\hat{f}_\pm\right)\right|}_{I_{4}}.
\end{aligned}
\end{equation}
To estimate $I_i$ $(i=1,2,3)$, we have from Lemma \ref{Lemma L}, Lemma \ref{lemma2.2} and Lemma \ref{lemma2.3} that

\begin{equation*}
\begin{aligned}
I_1\lesssim&\left\|\Lambda^{-\varrho}\left(E\cdot \mu^{\delta}f\right)\right\|\left\|\Lambda^{-\varrho}\left(\mu^{\delta}f\right)\right\|
+\left\|\Lambda^{-\varrho}\left(E \cdot\nabla_vf \langle v\rangle^{-\frac{\gamma}{2}}\right)\right\|\left\|\Lambda^{-\varrho}{\bf\{I-P\}}f\right\|_{\nu}\\{}
\lesssim&\left\|E\cdot \mu^{\delta}f\right\|_{L^2_vL_x^{\frac{6}{3+2\varrho}}} \left\|\Lambda^{-\varrho}\left(\mu^{\delta}f\right)\right\|+\left\|E\cdot \nabla_vf \langle v\rangle^{-\frac{\gamma}{2}}\right\|_{L_x^{\frac{6}{3+2\varrho}}L^2_v}
\left\|\Lambda^{-\varrho}{\bf\{I-P\}}f\right\|_{\nu}\\{}
\lesssim&\left\|\Lambda^{\frac34-\frac \varrho2}E\right\|\left\|\Lambda^{\frac34-\frac \varrho2}\left(\mu^{\delta}f\right)\right\|\left\|\Lambda^{-\varrho}\left(\mu^{\delta}f\right)\right\|
+\left\|\Lambda^{\frac32-\varrho}E\right\|^2\left\|\nabla_vf \langle v\rangle^{-\frac{\gamma}{2}}\right\|^2+\varepsilon\left\|\Lambda^{-\varrho}{\bf\{I-P\}}f\right\|^2_{\nu}\\{}
\lesssim&\left(\mathcal{\overline{E}}_{0,0,-\gamma}(t)\right)^{1/2}\left(\left\|\Lambda^{\frac34-\frac \varrho2}E\right\|^2+\left\|\Lambda^{\frac34-\frac \varrho2}f\right\|^2_{\nu}\right)+\mathcal{\overline{E}}_{1,\frac 32,-\gamma}(t) \left\|\Lambda^{\frac32-\varrho}E\right\|^2+\varepsilon\left\|\Lambda^{-\varrho}{\bf\{I-P\}}f\right\|^2_{\nu}.
\end{aligned}
\end{equation*}
For $I_2$ and $I_3$, we have by repeating the argument used in deducing the estimate on $I_1$ that
\begin{equation*}
\begin{aligned}
I_2+I_3\lesssim&\left(\mathcal{\overline{E}}_{0,0,-\gamma}(t)\right)^{1/2}
\left(\left\|\Lambda^{\frac34-\frac \varrho2}(E,B)\right\|^2+\left\|\Lambda^{\frac34-\frac \varrho2}f\right\|^2_{\nu}\right)\\{}
&+\mathcal{\overline{E}}_{1,\frac 32-\frac{1}{\gamma},-\gamma}(t) \left\|\Lambda^{\frac32-\varrho}(E,B)\right\|^2+\varepsilon
\left\|\Lambda^{-\varrho}{\bf\{I-P\}}f\right\|^2_{\nu}.
\end{aligned}
\end{equation*}
$I_{4}$ can be bounded from  Lemma \ref{lemma-nonlinear} by
\begin{equation*}
\begin{aligned}
I_{4}=&\left(\mathcal{F}[{\Gamma}(f,f)],|\xi|^{-2\varrho}\mathcal{F}{{\bf\{I-P\}}f}\right)\\{}
\lesssim&\left\|\Lambda^{-\varrho}(\langle v\rangle^{-\frac\gamma2}{\Gamma}(f,f))\right\|\left\|\Lambda^{-\varrho}(\langle v\rangle^{\frac\gamma2}{\bf\{I-P\}}f)\right\|\\{}
\lesssim&\left\|\Lambda^{\frac32-\varrho}f\right\|\left\|f\right\|_{L^2_xH^2_{v}}\left\|\Lambda^{-\varrho}{\bf\{I-P\}}f\right\|_{\nu}\\{}
\lesssim&\mathcal{\overline{E}}_{2,2,-\gamma}(t)\left\|\Lambda^{\frac32-\varrho}f\right\|^2+\varepsilon\|\Lambda^{-\varrho}{\bf\{I-P\}}f\|^2_{\nu}\\{}
\end{aligned}
\end{equation*}
Substituting the estimates on $I_i (i=1,2,3,4)$ into (\ref{4.6}) yields
\begin{equation*}
\begin{aligned}
&\frac{d}{dt}\left(\left\|\Lambda^{-\varrho}f\right\|^2+\left\|\Lambda^{-\varrho}(E,B)\right\|^2\right)
+\left\|\Lambda^{-\varrho}\{{\bf I-P}\}f\right\|_{\nu}^2\\{}
\lesssim&{\left(\mathcal{\overline{E}}_{0,0,-\gamma}(t)\right)^{1/2}\left(\left\|\Lambda^{\frac34-\frac \varrho2}(E,B)\right\|^2
+\left\|\Lambda^{\frac34-\frac \varrho2}f\right\|^2_{\nu}\right)+\mathcal{\overline{E}}_{2,\frac 32-\frac{1}{\gamma},-\gamma}(t)\left\|\Lambda^{\frac32-\varrho}(f,E,B)\right\|^2}.
\end{aligned}
\end{equation*}
Thus we complete the proof of Lemma $\ref{LemmaF}$.
\end{proof}

Applying the argument of Lemma 3.2 in \cite[page 3727]{Lei-Zhao-JFA-2014} and Lemma 3.3 in \cite[page 3731]{Lei-Zhao-JFA-2014}, we easily have the following lemma:

\begin{lemma}
Let $\varrho\in [\frac12, \frac32)$, there exists an interactive functional $G_{E,B}(t)$ satisfying
\begin{equation*}
{|G^{-\varrho}_{f}(t)|}\lesssim \left\|\Lambda^{1-\varrho}(f,E,B)\right\|^2+\left\|\Lambda^{-\varrho}(f,E,B)\right\|^2+\|\Lambda^{2-\varrho}E\|^2
\end{equation*}
such that
\begin{equation}\label{G_{E,B}}
\begin{aligned}
&\frac{d}{dt}{G^{-\varrho}_{f}(t)}
+\left\|\Lambda^{1-\varrho}(E,B)\right\|^2+\left\|\Lambda^{-\varrho}E\right\|^2+\left\|\Lambda^{-\varrho}(a_+-a_-)\right\|^2_{H^1}\\
\lesssim& \left\|\Lambda^{-\varrho}\{{\bf I-P}\}f\right\|^2_{\nu}+\left\|\Lambda^{1-\varrho}\{{\bf I-P}\}f\right\|^2_{\nu}
+\left\|\Lambda^{2-\varrho}\{{\bf I-P}\}f\right\|^2_{\nu}+\mathcal{\overline{E}}_{1,0,-\gamma}(t)\mathcal{\overline{D}}_{2,0,-\gamma}(t)
\end{aligned}
\end{equation}
holds for any $0\leq t\leq T$.
\end{lemma}


\noindent 4.3 {\bf The proof of Lemmas \ref{lemma4}, \ref{lemma5} and \ref{lemma6}.} We first give the proof of Lemma \ref{lemma6} as it is the most difficult one among those three lemmas.  The standard energy estimate on $\partial^\alpha f$ with $1\leq|\alpha|\leq N-1$ weighted by the time-velocity dependent function $w_{\ell,-\gamma}=w_{\ell,-\gamma}(t,v)$ gives
\begin{equation*}
\begin{aligned}
&\frac{d}{dt}\sum_{1\leq|\alpha|\leq N-1}\left\|w_{\ell,-\gamma}\partial^\alpha f\right\|^2+\sum_{1\leq|\alpha|\leq N-1}\left\|w_{\ell,-\gamma}\partial^\alpha f\right\|^2_{\nu}
+\frac{1}{(1+t)^{1+\vartheta}}\left\| w_{\ell,-\gamma}\partial^\alpha f\langle v\rangle\right\|^2\\{}
\lesssim&\sum_{1\leq|\alpha|\leq N-1}\left\|\partial^\alpha f\right\|_{\nu}^2+\sum_{1\leq|\alpha|\leq N-1}
\left\|\partial^\alpha E\right\|\left\|\mu^\delta\partial^\alpha f\right\|
+\underbrace{\sum_{1\leq|\alpha|\leq N-1}\left|\left(\partial^\alpha (v\times B)\cdot\nabla_v f),w^2_{\ell,-\gamma}\partial^\alpha f\right)\right|}_{J_1}\\{}
&+\underbrace{\sum_{1\leq|\alpha|\leq N-1}\left|\left(\partial^\alpha( E\cdot vf+E\cdot\nabla_v f),w^2_{\ell,-\gamma}\partial^\alpha f\right)\right|}_{J_2}+\underbrace{\sum_{1\leq|\alpha|\leq N-1}\left|\left(\partial^\alpha \Gamma(f,f),w^2_{\ell,-\gamma}\partial^\alpha f\right)\right|}_{J_{3}}.
\end{aligned}
\end{equation*}
{For $J_1$, we deduce
that
\begin{equation*}
\begin{aligned}
J_1
\lesssim&\underbrace{\left\|\partial^{\alpha}B\right\|
\left\|w_{\ell-1,-\gamma}\nabla_vf\langle v\rangle^{1-\frac{3\gamma}2}\right\|_{L^2_vL^\infty_x}
\left\|w_{\ell,-\gamma}\partial^\alpha f\right\|_\nu}_{J_{1,1}}\\
&+\underbrace{\sum_{1\leq|\alpha-\alpha_1|\leq N_0-2}\left\|\partial^{\alpha_1}B\right\|_{L^6_x}
\left\|w_{\ell-1,-\gamma}\nabla_v\partial^{\alpha-\alpha_1}f\langle v\rangle^{1-\frac{3\gamma}2}\right\|_{L^2_vL^3_x}
\left\|w_{\ell,-\gamma}\partial^\alpha f\right\|_{\nu}}_{J_{1,2}}\\{}
&+\underbrace{\sum_{|\alpha-\alpha_1|= N_0-1}\left\|\partial^{\alpha_1}B\right\|_{L^\infty_x}
\left\|w_{\ell-1,-\gamma}\nabla_v\partial^{\alpha-\alpha_1}f\langle v\rangle^{1-\frac{3\gamma}2}\right\|
\left\|w_{\ell,-\gamma}\partial^\alpha f\right\|_{\nu}}_{J_{1,3}}\\{}
&+\underbrace{\sum_{|\alpha-\alpha_1|\geq N_0,\atop
|\alpha_1|=N_0-1}\left\|\partial^{\alpha_1}B\right\|_{L^6_x}
\left\|w_{\ell-1,-\gamma}\nabla_v\partial^{\alpha-\alpha_1}\{{\bf I-P}\}f\langle v\rangle^{\frac12-\frac{\gamma}2}\right\|_{L^2_vL^3_x}
\left\|\langle v\rangle^{\frac12-\frac{\gamma}2}w_{\ell,-\gamma}\partial^\alpha f\right\|}_{J_{1,4}}\\{}
&+\underbrace{\sum_{|\alpha-\alpha_1|\geq N_0, \atop2\leq|\alpha_1|\leq N_0-1}\int_{\mathbb{R}^3}\left|\partial^{\alpha_1}B\right|
\left|w_{\ell-1,-\gamma}\nabla_v\partial^{\alpha-\alpha_1}\{{\bf I-P}\}f\langle v\rangle^{\frac12-\frac{\gamma}2}\right|
\left|w_{\ell,-\gamma}\partial^\alpha f\langle v\rangle^{\frac12-\frac{\gamma}2}\right|dx}_{J_{1,5}}\\{}
&+\underbrace{\sum_{|\alpha-\alpha_1|\geq N_0}\int_{\mathbb{R}^3}\left|\partial^{\alpha_1}B\right|
\left|w_{\ell-1,-\gamma}\nabla_v\partial^{\alpha-\alpha_1}{\bf P}f\langle v\rangle^{1-\frac{3\gamma}2}\right|
\left|w_{\ell,-\gamma}\partial^\alpha f\right|_{L^2_\nu}dx}_{J_{1,6}}\\{}
\end{aligned}
\end{equation*}
The first three term and the last term can be bounded by $$\sum_{i=1}^3J_{1,i}+J_{1,6}\lesssim\mathcal{E}_N(t)\mathcal{E}^1_{N_0,l_0,-\gamma}(t)+\mathcal{E}_{N_0}(t)\mathcal{D}_{N-1}(t)+\varepsilon\mathcal{D}_{N-1,\ell,-\gamma}(t)$$
where we take $l_0\geq\ell+\frac32-\frac1\gamma$ such that $w_{\ell,-\gamma}\langle v\rangle^{1-\frac{3\gamma}2}\leq w_{l_0,-\gamma}$.

As for the last two terms ${J_{1,4}}$ and ${J_{1,5}}$, we only estimate
${J_{1,4}}$ since ${J_{1,5}}$ can be obtained in a similar way,
\begin{eqnarray*}
J_{1,4}&\lesssim&\sum_{|\alpha-\alpha_1|\geq N_0,\atop|\alpha_1|=N_0-1}\left\|\partial^{\alpha_1} B\right\|_{L^6_x} \left\| w_{\ell-1,-\gamma}\partial^{\alpha-\alpha_1}\nabla_v\{{\bf I-P}\}f\langle v\rangle^{\frac\gamma2} \right\|^{1-\theta_3}_{L^2_vL^3_x}\\
&&\times\left\| w_{\ell-1,-\gamma}\partial^{\alpha-\alpha_1}\nabla_v\{{\bf I-P}\}f\langle v\rangle^{\widetilde{\ell}_3} \right\|^{\theta_3}_{L^2_vL^3_x}\left\| w_{\ell,-\gamma}\partial^{\alpha}f\langle v\rangle^{\frac\gamma{2}} \right\|^{1-\theta_3}\left\| w_{\ell,-\gamma}\partial^{\alpha}f\langle v\rangle^{\widetilde{\ell}_3}\right\|^{\theta_3}\\
&\lesssim&\sum_{|\alpha-\alpha_1|\geq N_0,\atop|\alpha_1|=N_0-1}\left\|\nabla^{N_0} B\right\|^{\frac1{\theta_3}} \left\| w_{\ell-1,-\gamma}\partial^{\alpha-\alpha_1}\nabla_v\{{\bf I-P}\}f\langle v\rangle^{\widetilde{\ell}_3} \right\|_{L^2_vL^3_x}
\left\| w_{\ell,-\gamma}\partial^{\alpha}f\langle v\rangle^{\widetilde{\ell}_3} \right\|\\
&&+\sum_{|\alpha-\alpha_1|\geq N_0,\atop|\alpha_1|=N_0-1}\varepsilon\left(\left\|w_{\ell,-\gamma}\partial^{\alpha}f\right\|_\nu^2
+\left\| w_{\ell-1,-\gamma}\partial^{\alpha-\alpha_1}\nabla_v\{{\bf I-P}\}f\right\|^2_{L^2_\nu L^3_x}\right)\\
&\lesssim&\left\|\nabla^{N_0}B\right\|^{\frac1{\theta_3}}\mathcal{\widetilde{D}}_{N-1,l_1^*,1}(t)
+\varepsilon\mathcal{D}_{{N-1},\ell,-\gamma}(t).
\end{eqnarray*}
where $\mathcal{\widetilde{D}}_{N-1,l_1^*,1}(t)$ is given in \eqref{D_{m,l^*_1,1}} with $m=N-1$ and $\theta_3$ satisfies that
$\frac12-\frac12\gamma=\frac\gamma2(1-\theta_3)+\widetilde{\ell}_3\theta_3$ which yields that $\theta_3=\frac{1-2\gamma}{2\widetilde{\ell}_3-\gamma}$.
Meanwhile $l_1^*$ satisfies that
$\widetilde{\ell}_3+(-\gamma)(\ell-1)\leq l_1^*-1+\frac\gamma2
$ which deduce that $l_1^*\geq \widetilde{\ell}_3+(-\gamma)(\ell-1)+1-\frac\gamma2$. Notice that $\gamma\in(-3,-1)$, we can take $l_1^*\geq \widetilde{\ell}_3-\frac\gamma2-\gamma\ell$.
Consequently, $$J_1\lesssim\left\|\nabla^2B\right\|_{H^{ N_0-2}_x}^{\frac1{\theta_3}}\mathcal{\widetilde{D}}_{N-1,l_1^*,1}(t)
+\mathcal{E}_N(t)\mathcal{E}^1_{N_0,l_0,-\gamma}(t)+\mathcal{E}_{N_0}(t)\mathcal{D}_{N-1}(t)+\varepsilon\mathcal{D}_{N-1,\ell,-\gamma}(t).$$
By the virtue of the estimates on $J_1$, we also have
\begin{equation*}
\begin{aligned}
J_2\lesssim&\left\|\nabla^2E\right\|_{H^{ N_0-2}_x}^{\frac1{\theta_3}}\mathcal{\widetilde{D}}_{N-1,l_1^*,1}(t)
+\mathcal{E}_N(t)\mathcal{E}^1_{N_0,l_0,-\gamma}(t)+\mathcal{E}_{N_0}(t)\mathcal{D}_{N-1}(t)\\{}
&
+\|E\|_{L^\infty_x}^{\frac{2-\gamma}{1-\gamma}}\sum_{1\leq|\alpha|\leq N-1}\left\|w_{\ell,-\gamma}\partial^{\alpha}\{{\bf I-P}\}f\langle v\rangle\right\|^2+\varepsilon\mathcal{D}_{N-1,\ell,-\gamma}(t)
\end{aligned}
\end{equation*}
Applying Lemma \ref{lemma-nonlinear} gives
\begin{equation*}
\begin{aligned}
J_{3}\lesssim&(\mathcal{E}_{N-1,\ell,-\gamma}(t)+\varepsilon)\mathcal{D}_{N-1,\ell,-\gamma}(t).
\end{aligned}
\end{equation*}
}

Collecting the above estimates gives the desired weighted energy type estimates on the derivatives of $f(t,x,v)$ with respect to the $x-$variables only as follows
\begin{equation}\label{E-1}
\begin{aligned}
&\frac{d}{dt}\sum_{1\leq|\alpha|\leq N-1}\left\|w_{\ell,-\gamma}\partial^\alpha f\right\|^2+\sum_{1\leq|\alpha|\leq N-1}\left\|w_{\ell,-\gamma}\partial^\alpha f\right\|^2_{\nu}
+\frac{1}{(1+t)^{1+\vartheta}}\left\| w_{\ell,-\gamma}\partial^\alpha f\langle v\rangle\right\|^2_\nu\\{}
\lesssim&\sum_{1\leq|\alpha|\leq N-1}\left\|\partial^\alpha f\right\|_{\nu}^2+\sum_{1\leq|\alpha|\leq N-1}
\left\|\partial^\alpha E\right\|\left\|\mu^\delta\partial^\alpha f\right\|+\{\mathcal{E}_{N-1,\ell,-\gamma}(t)+\varepsilon\}\mathcal{D}_{N-1,\ell,-\gamma}(t)\\
&+\|E\|_{L^\infty_x}^{\frac{2-\gamma}{1-\gamma}}\sum_{1\leq|\alpha|\leq N-1}\left\|w_{\ell,-\gamma}\partial^{\alpha}\{{\bf I-P}\}f\langle v\rangle\right\|^2+\left\|\nabla^2(E,B)\right\|_{H^{ N_0-2}_x}^{\frac1{\theta_3}}\mathcal{\widetilde{D}}_{N-1,l_1^*,1}(t)
+\mathcal{E}_N(t)\mathcal{E}^1_{N_0,l_0,-\gamma}(t).
\end{aligned}
\end{equation}
After applying $\{\bf I-P\}$ to the equation \eqref{f gn}, one can get the weighted energy estimate on $\{{\bf I-P}\}f$
\begin{equation}\label{E-2}
\begin{aligned}
&\frac{d}{dt}\left\|w_{\ell,-\gamma}\{{\bf I-P}\}f\right\|^2+\left\|w_{\ell,-\gamma}\{{\bf I-P}\}f\right\|_{\nu}^2+\frac{1}{(1+t)^{1+\vartheta}}\left\| w_{\ell,-\gamma}\{{\bf I-P}\}f\langle v\rangle\right\|^2\\{}
\lesssim&\|\{{\bf I-P}\}f\|_{\nu}^2+\|E\|_{L^\infty_x}^{\frac{2-\gamma}{1-\gamma}}\left\|w_{\ell,-\gamma}\{{\bf I-P}\}f\langle v\rangle\right\|^2+\|E\|^2+\left\|\nabla_x f\right\|_{\nu}^2
+\left(\mathcal{E}_{N-1,\ell,-\gamma}(t)
+\varepsilon\right)\mathcal{D}_{N-1,\ell,-\gamma}(t).
\end{aligned}
\end{equation}
As to the weighted energy estimate on $\{{\bf I-P}\}\partial^\alpha_\beta f$ with $|\alpha|+|\beta|\leq N-1,|\beta|\geq1$, applying the similar trick as \eqref{E-1}, we also deduce
\begin{equation}\label{E-3}
\begin{aligned}
&\frac{d}{dt}\sum_{m=1}^{N-1}C_m\sum_{\substack{|\beta|=m,\\|\alpha|+|\beta|\leq N-1}}\left\|w_{\ell-|\beta|,-\gamma}\partial^\alpha_\beta\{{\bf I-P}\}f\right\|^2
+\sum_{\substack{|\alpha|+|\beta|\leq N-1,\\|\beta|\geq 1}}\bigg\{\left\| w_{\ell-|\beta|,-\gamma}\partial^\alpha_\beta\{{\bf I-P}\}f\right\|_{\nu}^2\\{}
&+\frac{1}{(1+t)^{1+\vartheta}}\left\| w_{\ell-|\beta|,-\gamma}\partial^\alpha_\beta\{{\bf I-P}\}f\langle v\rangle \right\|^2\bigg\}\\{}
\lesssim
&\sum_{|\alpha|\leq N-2 }\left(\left\|\nabla^{|\alpha|+1}_xf\right\|_\nu^2+\|\{{\bf I-P}\}f\|^2_\nu+\left\|\partial^\alpha E\right\|^2\right)+\|E\|_{L^\infty_x}^{\frac{2-\gamma}{1-\gamma}}\sum_{|\alpha|+|\beta|\leq N-1}\left\|w_{\ell,-\gamma}\partial^{\alpha}_\beta\{{\bf I-P}\}f\langle v\rangle\right\|^2\\{}
&
+\left\|\nabla^2(E,B)\right\|_{H^{ N_0-2}_x}^{\frac1{\theta_3}}\mathcal{\widetilde{D}}_{N-1,l_1^*,1}(t)+\mathcal{E}_N(t)\mathcal{E}^1_{N_0,l_0,-\gamma}(t)+\left(\mathcal{E}_{N-1,\ell,-\gamma}(t)
+\varepsilon\right)\mathcal{D}_{N-1,\ell,-\gamma}(t).
\end{aligned}
\end{equation}
Here we used the fact that $\left((v\times B)\cdot\partial^\alpha_\beta\nabla_v \{{\bf I-P} \}f,w^2_{\ell-|\beta|,-\gamma}\partial_\beta^\alpha \{{\bf I-P} \}f\right)=0.$

Therefore, recalling \eqref{D_{m,l^*_1,1}}, a proper linear combination of (\ref{lemma5-1}), (\ref{E-1}), (\ref{E-2}) and (\ref{E-3}) gives \eqref{lemma6-1} by taking
$l_0\geq l_1+\frac52$, $\theta_3=\frac{1-2\gamma}{2\widetilde{\ell}_3-\gamma}$ , $l_1^*\geq \widetilde{\ell}_3-\frac\gamma2-\gamma\ell$ and further by replacing $\ell$ with $l_1\geq N$.
This proves Lemma \ref{lemma6}.

\medskip

Now, for brevity let us modify the proof of Lemma \ref{lemma6} above so as to obtain Lemma  \ref{lemma4} and Lemma  \ref{lemma5}. To prove Lemma \ref{lemma4}, similarly for deducing \eqref{E-1}, \eqref{E-2} and \eqref{E-3}, one can get
 \begin{equation}\label{C-1}
\begin{aligned}
&\frac{d}{dt}\sum_{1\leq|\alpha|\leq N_0}\left\|w_{\ell,-\gamma}\partial^\alpha f\right\|^2 +\sum_{1\leq|\alpha|\leq N_0}\left\|w_{\ell,-\gamma}\partial^\alpha f\right\|^2_{\nu}+\frac{1}{(1+t)^{1+\vartheta}}\left\| w_{\ell,-\gamma}\partial^\alpha f\langle v\rangle\right\|^2\\
\lesssim&\sum_{1\leq|\alpha|\leq N_0}\left\|\partial^\alpha f\right\|_{\nu}^2 +\sum_{1\leq|\alpha|\leq N_0}\left\|\partial^\alpha E\right\|
\left\|\mu^\delta\partial^\alpha f\right\|+\|E\|_{L^\infty_x}^{\frac{2-\gamma}{1-\gamma}}\sum_{1\leq|\alpha|\leq N_0}\left\|w_{\ell,-\gamma}\partial^{\alpha}\{{\bf I-P}\}f\langle v\rangle\right\|^2\\
&+\left\|\nabla^2(E,B)\right\|_{H^{ N_0-2}_x}^{\frac1{\theta_1}}\mathcal{\widetilde{D}}_{N_0,l_0^*,1}(t)
+\left(\mathcal{E}_{N_0,\ell,-\gamma}(t)+\varepsilon\right)
\mathcal{D}_{{N_0},\ell,-\gamma}(t),
\end{aligned}
\end{equation}
and
\begin{equation}\label{C-2}
\begin{aligned}
&\frac{d}{dt}\|w_{\ell,-\gamma}\{{\bf I-P}\}f\|^2+\|w_{\ell,-\gamma}\{{\bf I-P}\}f\|_{\nu}^2+\frac{1}{(1+t)^{1+\vartheta}}\| w_{\ell,-\gamma}\{{\bf I-P}\}f\langle v\rangle \|^2\\{}
\lesssim&\|\{{\bf I-P}\}f\|_{\nu}^2+\|E\|_{L^\infty_x}^{\frac{2-\gamma}{1-\gamma}}\left\|w_{\ell,-\gamma}\{{\bf I-P}\}f\langle v\rangle\right\|^2+\|E\|^2+\|\nabla f\|_{\nu}^2
+(\mathcal{E}_{N_0,\ell,-\gamma}(t)+\varepsilon)
\mathcal{D}_{{N_0},\ell,-\gamma}(t),
\end{aligned}
\end{equation}
and
\begin{equation}\label{C-3}
\begin{aligned}
&\frac{d}{dt}\sum_{m=1}^{N_0}C_m\sum_{\substack{|\beta|=m,\\|\alpha|+|\beta|\leq N_0}}\left\|w_{\ell-|\beta|,-\gamma}\partial^\alpha_\beta\{{\bf I-P}\}f\right\|^2
+\sum_{\substack{|\alpha|+|\beta|\leq N_0,\\|\beta|\geq 1}}\bigg(\left\| w_{\ell-|\beta|,-\gamma}\partial^\alpha_\beta\{{\bf I-P}\}f\right\|_{\nu}^2\\{}
&+\frac{1}{(1+t)^{1+\vartheta}}\left\| w_{\ell-|\beta|,-\gamma}\partial^\alpha_\beta\{{\bf I-P}\}f\langle v\rangle\right\|^2\bigg)\\{}
\lesssim&\sum_{|\alpha|\leq N_0-1 }\left(\left\|\nabla^{|\alpha|+1}f\right\|_\nu^2+\|\{{\bf I-P}\}f\|_\nu^2+\left\|\partial^\alpha E\right\|^2\right)+\|E\|_{L^\infty_x}^{\frac{2-\gamma}{1-\gamma}}\sum_{|\alpha|+|\beta|\leq N_0}\left\|w_{\ell-|\beta|,-\gamma}\partial^{\alpha}_\beta\{{\bf I-P}\}f\langle v\rangle\right\|^2\\{}
&+\left\|\nabla^2(E,B)\right\|_{H^{ N_0-2}_x}^{\frac1{\theta_1}}\mathcal{\widetilde{D}}_{N_0,l_0^*,1}(t)
+\left(\mathcal{E}_{N_0,\ell,-\gamma}(t)+\varepsilon\right)
\mathcal{D}_{{N_0},\ell,-\gamma}(t),
\end{aligned}
\end{equation}
where $\mathcal{\widetilde{D}}_{{N_0},\ell,-\gamma}(t)$ is given in \eqref{D_{N_0,l^*_0,1}}.
Therefore, recalling \eqref{D_{N_0,l^*_0,1}}, a proper linear combination of (\ref{Lemma1-1}), \eqref{Lemma4.1-1}, \eqref{G_{E,B}}, (\ref{C-1}), (\ref{C-2}) and (\ref{C-3}) yields the desired estimate \eqref{lemma4-1} by taking
$\theta_1=\frac{1-2\gamma}{2\widetilde{l}_1-\gamma}$and $l_0^*\geq \widetilde{\ell}_1-\frac\gamma2-\gamma\ell$. This proves Lemma \ref{lemma4}.  Finally, Lemma \ref{lemma5} follows from modifying the proof of Lemma \ref{lemma6} in a straightforward way without considering any weight function; the details of the proof are omitted for brevity.  \qed


\medskip
\noindent 4.4 {\bf The proof of Lemma \ref{lemma7}.}
To prove Lemma \ref{lemma7}, we firstly estimate the highest $N-$th order norm as follows:
\begin{eqnarray}\label{G-1}
&&\frac{d}{dt}\sum_{|\alpha|=N}\left\|w_{l_1^*,1}\partial^\alpha f\right\|^2
+\sum_{|\alpha|=N}(1+t)^{-1-\vartheta}\left\|w_{l_1^*,1}\partial^\alpha f\langle v\rangle\right\|^2
+\sum_{|\alpha|=N}\left\|w_{l_1^*,1}\partial^\alpha f\right\|_\nu^2\\ \nonumber
&\lesssim&\sum_{|\alpha|=N}\left\|\partial^\alpha f\right\|_\nu^2
+\sum_{|\alpha|=N}\left\|\partial^\alpha f\right\|_\nu\left\|\partial^\alpha E\right\|
+\underbrace{\sum_{|\alpha|=N}\left(\partial^\alpha(v\times B\cdot\nabla_v f),w^2_{l_1^*,1}\partial^\alpha f\right)}_{R_1}\\ \nonumber
&&
+\underbrace{\sum_{|\alpha|=N}\left(\partial^\alpha(E\cdot\nabla_v f),w^2_{l_1^*,1}\partial^\alpha f\right)}_{R_2}+\underbrace{\sum_{|\alpha|=N}\left(\partial^\alpha(v\cdot E f),w^2_{l_1^*,1}\partial^\alpha f\right)}_{R_3}+\underbrace{\sum_{|\alpha|=N}\left(\partial^\alpha\Gamma(f,f),w^2_{l_1^*,1}\partial^\alpha f\right)}_{R_4}.
\end{eqnarray}

Applying macro-micro decomposition and Holder inequalities gives
\begin{eqnarray*}
R_1&=&\underbrace{\sum_{|\alpha-\alpha_1|\leq N_0-2}\left\|\partial^{\alpha_1} B\right\|_{L^3_x}
\left\|w_{l_1^*,1}\nabla_v \partial^{\alpha-\alpha_1}{\bf \{I-P\}}f\langle v\rangle^{-\frac\gamma2+1}\right\|_{L^2_vL^6_x}
\left\|w_{l_1^*,1}\partial^\alpha f\langle v\rangle^\frac\gamma2\right\|}_{R_{1,1}}\\
&&+\underbrace{\sum_{|\alpha-\alpha_1|= N_0-1}\left\|\partial^{\alpha_1} B\right\|_{L^\infty_x}
\left\|w_{l_1^*,1}\nabla_v \partial^{\alpha-\alpha_1}{\bf \{I-P\}}f\langle v\rangle^{-\frac\gamma2+1}\right\|
\left\|w_{l_1^*,1}\partial^\alpha f\langle v\rangle^{\frac\gamma2}\right\|}_{R_{1,2}}\\
&&+\underbrace{\sum_{|\alpha-\alpha_1|+1\geq N_0+1, \atop 1\leq|\alpha_1|\leq N_0-2}
\left\|\partial^{\alpha_1} B\right\|_{L^\infty_x}\left\|w_{l_1^*-1,1}\nabla_v \partial^{\alpha-\alpha_1}{\bf \{I-P\}}f\langle v\rangle\right\|
\left\|w_{l_1^*,1}\partial^\alpha f\langle v\rangle\right\|}_{R_{1,3}}\\
&&+\underbrace{\sum_{|\alpha-\alpha_1|+1\geq N_0+1,\atop N_0-1\leq|\alpha_1|\leq N_0}
\left\|\partial^{\alpha_1} B\right\|\left\|w_{l_1^*-1,1}\nabla_v \partial^{\alpha-\alpha_1}{\bf \{I-P\}}f\langle v\rangle\right\|_{L^2_vL^\infty_x}
\left\|w_{l_1^*,1}\partial^\alpha f\langle v\rangle\right\|}_{R_{1,4}}\\ \nonumber
&&+\underbrace{\sum_{\alpha_1\leq\alpha}\int_{\mathbb{R}^3_x}\left|\partial^{\alpha_1}B\right|\left|\mu^\delta\partial^{\alpha-\alpha_1}f\right|
\left|\mu^\delta\partial^{\alpha}f\right|dx}_{R_{1,5}}.
\end{eqnarray*}
It is straightforward to compute that
\begin{eqnarray*}
R_{1,1}+R_{1,2}+R_{1,5}\lesssim
\mathcal{E}_{N}(t)\mathcal{E}^1_{N_0,l_0,-\gamma}(t)
+\mathcal{E}_{N}(t)\mathcal{D}_{N}(t)+\varepsilon\sum_{|\alpha|=N}\left\|w_{l_1^*,1}\partial^\alpha f\right\|_\nu^2,
\end{eqnarray*}
{where $l_0$ satisfies that $l_1^*-|\beta|+1-\frac\gamma2\leq -\gamma (l_0-|\beta|-1)$ so $l_0\geq\frac{l_1^*}{-\gamma}+\frac32+|\beta|+\frac{|\beta|-1}\gamma$.
Noticing that $\gamma\in(-3,-1)$, so we take $l_0\geq l_1^*+\frac{5}{2}$.}
As to $R_{1,3}$ and $R_{1,4}$, one deduce by Cauchy's inequality
\begin{eqnarray*}
R_{1,3}+R_{1,4}&\lesssim&\sum_{|\alpha-\alpha_1|+1\geq N_0+1,\atop 1\leq|\alpha_1|\leq N_0-2}(1+t)^{1+\vartheta}
\left\|\partial^{\alpha_1} B\right\|^2_{L^\infty_x}\left\|w_{l_1^*-1,1}\nabla_v \partial^{\alpha-\alpha_1}{\bf \{I-P\}}f\langle v\rangle\right\|^2\\
&&+\sum_{|\alpha-\alpha_1|+1\geq N_0+1, \atop N_0-1\leq|\alpha_1|\leq N_0}
(1+t)^{1+\vartheta}\left\|\partial^{\alpha_1} B\right\|^2
\left\|w_{l_1^*-1,1}\nabla_v \partial^{\alpha-\alpha_1}{\bf \{I-P\}}f\langle v\rangle\right\|^2
_{L^2_vL^\infty_x}\\
&&+\varepsilon(1+t)^{-1-\vartheta}\sum_{|\alpha|=N}\left\|w_{l_1^*,1}\partial^\alpha f\langle v\rangle\right\|^2.
\end{eqnarray*}
Consequently
\begin{eqnarray*}
R_1&\lesssim&\sum_{|\alpha-\alpha_1|+1\geq N_0+1,\atop 1\leq|\alpha_1|\leq N_0-2}(1+t)^{1+\vartheta}
\left\|\partial^{\alpha_1} B\right\|^2_{L^\infty_x}
\left\|w_{l_{l_1^*-1},1}\nabla_v \partial^{\alpha-\alpha_1}{\bf \{I-P\}}f\langle v\rangle\right\|^2\\
&&+\sum_{|\alpha-\alpha_1|+1\geq N_0+1,\atop N_0-1\leq|\alpha_1|\leq N_0}
(1+t)^{1+\vartheta}\left\|\partial^{\alpha_1} B\right\|^2
\left\|w_{l_1^*-1,1}\nabla_v \partial^{\alpha-\alpha_1}{\bf \{I-P\}}f\langle v\rangle\right\|^2
_{L^2_vL^\infty_x}\\
&&+\mathcal{E}_N(t)\mathcal{D}_N(t)+\mathcal{E}_{N}(t)\mathcal{E}^1_{N_0,l_0,-\gamma}(t)
+\varepsilon\sum_{|\alpha|=N}\left\{(1+t)^{-1-\vartheta}
\left\|w_{l_1^*,1}\partial^\alpha f\langle v\rangle\right\|^2+\left\|w_{l_1^*,1}\partial^\alpha f\right\|_\nu^2\right\}.
\end{eqnarray*}
By exploiting the same argument used to estimate $R_1$, one also deduces
\begin{eqnarray*}
R_2+R_3&\lesssim&\sum_{|\alpha-\alpha_1|+m\geq N_0+1,\atop { 1\leq|\alpha_1|\leq N_0-2},m\leq 1}(1+t)^{1+\vartheta}\|\partial^{\alpha_1} E\|^2_{L^\infty_x}\|w_{l_1^*-m,1}\nabla^m_v \partial^{\alpha-\alpha_1}{\bf \{I-P\}}f\|^2\\
&&+\sum_{|\alpha-\alpha_1|+m\geq N_0+1,\atop{ N_0-1\leq|\alpha_1|\leq N_0}, m\leq1}
(1+t)^{1+\vartheta}\|\partial^{\alpha_1} E\|^2\|w_{l_1^*-m,1}
\nabla^m_v \partial^{\alpha-\alpha_1}
{\bf \{I-P\}}f\|^2
_{L^2_vL^\infty_x}\\
&&+\sum_{|\alpha|=N}\|E\|_{L^\infty_x}^{\frac{2-\gamma}{1-\gamma}}\left\|w_{l_1^*,1}\partial^\alpha \{{\bf I-P}\}f\langle v\rangle\right\|^2 +\mathcal{E}_N(t)\mathcal{D}_N(t)+\mathcal{E}_{N}(t)\mathcal{E}^1_{N_0,l_0,-\gamma}(t)
\\{}
&&+\varepsilon\sum_{|\alpha|=N}\left\{(1+t)^{-1-\vartheta}\left\|w_{l_1^*,1}\partial^\alpha f\langle v\rangle\right\|^2
+\left\|w_{l_1^*,1}\partial^\alpha f\right\|_\nu^2\right\}.
\end{eqnarray*}
For $R_4$, Lemma \ref{lemma-nonlinear} tells us that
{\begin{eqnarray*}
R_4&\lesssim&
\underbrace{\sum_{|\alpha_1|\leq N_0-4,m\leq2}\int_{\mathbb{R}^3}\left|\nabla^m_v\left(\mu^{\delta}\partial^{\alpha_1}f\right)\right|
\left|w_{l_1^*,1}\partial^{\alpha-\alpha_1}f\right|_\nu\left|w_{l_1^*,1}\partial^{\alpha}f\right|_\nu dx}_{R_{4,1}}\\
&&+\underbrace{\sum_{|\alpha_1|=N_0-3\ or\ N_0-2 ,m\leq2}\int_{\mathbb{R}^3}\left|\nabla^m_v\left(\mu^{\delta}\partial^{\alpha_1}f\right)\right|
\left|w_{l_1^*,1}\partial^{\alpha-\alpha_1}f\right|_\nu\left|w_{l_1^*,1}\partial^{\alpha}f\right|_\nu dx}_{R_{4,2}}\\
&&+\underbrace{\sum_{|\alpha_1|\geq N_0-2,m\leq2}\int_{\mathbb{R}^3}\left|w_{l_1^*,1}\partial^{\alpha_1}f\right|_\nu
\left|\nabla^m_v\left(\mu^{\delta}\partial^{\alpha-\alpha_1}f\right)\right|\left|w_{l_1^*,1}\partial^{\alpha}f\right|_\nu dx}_{R_{4,3}}\\
&&+\underbrace{\sum_{|\alpha_1|\leq N_0}\int_{\mathbb{R}^3}\left|w_{l_1^*,1}\partial^{\alpha_1}f\right|
\left|w_{l_1^*,1}\partial^{\alpha-\alpha_1}f\right|_\nu\left|w_{l_1^*,1}\partial^{\alpha}f\right|_\nu dx}_{R_{4,4}}\\
&&+\underbrace{\sum_{N_0+1\leq|\alpha_1|\leq |\alpha|-1}\int_{\mathbb{R}^3}\left|w_{l_1^*,1}\partial^{\alpha_1}f\right|_\nu
\left|w_{l_1^*,1}\partial^{\alpha-\alpha_1}f\right|\left|w_{l_1^*,1}\partial^{\alpha}f\right|_\nu dx}_{R_{4,5}}\\
&&+\underbrace{\sum_{\alpha_1= \alpha}\int_{\mathbb{R}^3}\left|w_{l_1^*,1}\partial^{\alpha}f\right|_\nu
\left|w_{l_1^*,1}f\right|\left|w_{l_1^*,1}\partial^{\alpha}f\right|_\nu dx}_{R_{4,6}}.
\end{eqnarray*}
For $R_{4,1}$, one obtains by the H$\ddot{o}$lder inequality that
\begin{eqnarray*}
R_{4,1}&\lesssim&
\sum_{|\alpha_1|\leq N_0-4,\atop m\leq2}\left\|\nabla^m_v\left(\mu^{\delta}\partial^{\alpha_1}f\right)\right\|^2_{L^\infty_xL^2_v}
\left\|w_{l_1^*,1}\partial^{\alpha-\alpha_1}f\right\|^2_{\nu}+\varepsilon\left\|w_{l_1^*,1}\partial^{\alpha}f\right\|_\nu\\
&\lesssim&\mathcal{E}_{N_0,l_0,-\gamma}(t)
\sum_{|\alpha_1|\leq N_0-4}
\left\|w_{l_1^*,1}\partial^{\alpha-\alpha_1}f\right\|^2_{\nu }+\varepsilon\left\|w_{l_1^*,1}\partial^{\alpha}f\right\|^2_\nu.
\end{eqnarray*}
By using $L^2-L^\infty-L^2$, $L^\infty-L^2-L^2$, $L^6-L^3-L^2$  or $L^3-L^6-L^2$ type inequalities with respect to space derivative $x$,  one also has
\begin{eqnarray*}
\sum_{i=2}^{6} R_{4,i}\lesssim\max\left\{\mathcal{E}_{N_0,l_0,-\gamma}(t),
\mathcal{E}_{N-1,N-1,-\gamma}(t)\right\}
\sum_{1\leq |\alpha_1|\leq |\alpha|}\left\|w_{l_1^*,1}\partial^{\alpha_1}f\right\|^2_{\nu}
+\varepsilon\left\|w_{l_1^*,1}\partial^{\alpha}f\right\|^2_\nu.
\end{eqnarray*}}
Consequently
\begin{eqnarray*}
R_4&\lesssim&\max\{\mathcal{E}_{N_0,l_0,-\gamma}(t),
\mathcal{E}_{N-1,N-1,-\gamma}(t)\}
\sum_{1\leq |\alpha_1|\leq |\alpha|}\left\|w_{l_1^*,1}\partial^{\alpha_1}f\right\|^2_{\nu }
+\varepsilon\left\|w_{l_1^*,1}\partial^{\alpha}f\right\|^2_\nu.
\end{eqnarray*}
Recalling \eqref{Enj} in the case when $n=N$ and $j=0$, we refer to
\begin{eqnarray*}
\sum_{i=1}^4R_i&\lesssim& E^N_{tri,0}(t)+\mathcal{E}_{N}(t)\mathcal{E}^1_{N_0,l_0,-\gamma}(t)
+\mathcal{E}_{N}(t)\mathcal{D}_{N}(t)\\
&&+\sum_{|\alpha|=N}\left\|\partial^\alpha f\right\|_\nu\left\|\partial^\alpha E\right\|+\varepsilon\sum_{|\alpha|=N}\left\{(1+t)^{-1-\vartheta}
\left\|w_{l_1^*,1}\partial^\alpha f\langle v\rangle\right\|^2+\left\|w_{l_1^*,1}\partial^\alpha f\right\|_\nu^2\right\}.
\end{eqnarray*}
Collecting the above estimates into (\ref{G-1}) yields
\begin{eqnarray*}
&&\frac{d}{dt}\sum_{|\alpha|=N}\left\|w_{l_1^*,1}\partial^\alpha f\right\|^2+\sum_{|\alpha|=N}(1+t)^{-1-\vartheta}
\left\|w_{l_1^*,1}\partial^\alpha f\langle v\rangle\right\|^2+\sum_{|\alpha|=N}
\left\|w_{l_1^*,1}\partial^\alpha f\right\|_\nu^2\\
&\lesssim&\sum_{|\alpha|=N}\left\|\partial^\alpha f\right\|_\nu^2+E^N_{tri,0}(t)+\mathcal{E}_{N}(t)\mathcal{E}^1_{N_0,l_0,-\gamma}(t)
+\mathcal{E}_{N}(t)\mathcal{D}_{N}(t)+\sum_{|\alpha|=N}\left\|\partial^\alpha f\right\|_\nu\left\|\partial^\alpha E\right\|
.
\end{eqnarray*}
When $|\alpha|+|\beta|=N,|\beta|=1$, one has
\begin{eqnarray}\label{G-2}
&&\frac{d}{dt}\sum_{|\alpha|+|\beta|=N,|\beta|=1}\left\|w_{l_1^*-1,1}\partial_\beta^\alpha\{{\bf I-P}\} f\right\|^2
+\sum_{|\alpha|+|\beta|=N,|\beta|=1}\left\|w_{l_1^*-1,1}\partial_\beta^\alpha\{{\bf I-P}\} f\right\|^2_\nu
\\ \nonumber
&&+(1+t)^{-1-\vartheta}\sum_{|\alpha|+|\beta|=N,|\beta|=1}\left\|w_{l_1^*-1,1}\partial_\beta^\alpha\{{\bf I-P}\} f\langle v\rangle\right\|^2\\ \nonumber
&\lesssim&\sum_{|\alpha|+|\beta|=N,|\beta|=1}\left\{\eta\left\|w_{l_1^*,1}\partial^\alpha\{{\bf I-P}\} f\right\|^2_\nu
+\left\|\partial^{\alpha} \{{\bf I-P}\} f\right\|_\nu^2\right\}\\ \nonumber
&&+\underbrace{\sum_{|\alpha|+|\beta|=N,|\beta|=1}\left(\partial_\beta^\alpha(v\cdot \{{\bf I-P}\} f),w^2_{l_1^*-1,1}\partial_\beta^\alpha\{{\bf I-P}\} f\right)}_{R_5}\\ \nonumber
&&+\underbrace{\sum_{|\alpha|+|\beta|=N,|\beta|=1}\left(\partial_\beta^\alpha((v\times B)\cdot\nabla_v \{{\bf I-P}\} f),w^2_{l_1^*-1,1}\partial_\beta^\alpha\{{\bf I-P}\} f\right)}_{R_6}
\\ \nonumber \nonumber&&+\underbrace{\sum_{|\alpha|+|\beta|=N,|\beta|=1}\left(\partial_\beta^\alpha(E\cdot\nabla_v \{{\bf I-P}\} f),w^2_{l_1^*-1,1}\partial_\beta^\alpha\{{\bf I-P}\} f\right)}_{R_7}\\ \nonumber
&&+\underbrace{\sum_{|\alpha|+|\beta|=N,|\beta|=1}\left(\partial_\beta^\alpha(v\cdot E \{{\bf I-P}\} f),w^2_{l_1^*-1,1}\partial_\beta^\alpha\{{\bf I-P}\} f\right)}_{R_8}
\\ \nonumber &&+\underbrace{\sum_{|\alpha|+|\beta|=N,|\beta|=1}\left(\partial_\beta^\alpha\Gamma(f,f),w^2_{l_1^*-1,1}\partial_\beta^\alpha\{{\bf I-P}\} f\right)}_{R_9}
\\ \nonumber
&&+\underbrace{\sum_{|\alpha|+|\beta|=N,|\beta|=1}\left(\partial_\beta^\alpha I_{mac}(t),w^2_{l_1^*-1,1}\partial_\beta^\alpha\{{\bf I-P}\} f\right)}_{R_{10}},
\end{eqnarray}
{where $I_{mac}(t)$ is defined by
\begin{eqnarray}\label{I-mac}
I_{mac}(t)&=&{\{{\bf I-P}\}}q_1E\cdot v\mu^{1/2}+{\bf P}\left\{v\cdot\nabla_x f+q_0(E+v\times B)\cdot\nabla_v f-\frac{q_0}{2}E\cdot vf\right\}\\ \nonumber
&&-v\cdot\nabla_x {\bf P}f-q_0(E+v\times B)\cdot\nabla_v {\bf P}f+\frac{q_0}{2}E\cdot v{\bf P}f.
\end{eqnarray}
}
Unlike the corresponding linear term for the weight $w_{\ell,-\gamma}$,
here $R_5$ can be dominated by
\begin{eqnarray*}
R_5&\lesssim&\sum_{|\alpha|+|\beta|=N,|\beta|=1}\left(\partial^{\alpha+e_i} \{{\bf I-P}\} f
,w^2_{l_1^*-1,1}\partial_\beta^\alpha\{{\bf I-P}\} f\right)\\ \nonumber
&\lesssim&\sum_{|\alpha|+|\beta|=N,|\beta|=1}
\left\|w_{l_1^*,1}\partial^{\alpha+e_i} \{{\bf I-P}\} f\langle v\rangle^{-\frac\gamma2-1}\right\|
\left\|w_{l_1^*-1,1}\partial_\beta^\alpha\{{\bf I-P}\} f\langle v\rangle^{\frac\gamma2}\right\|\\ \nonumber
&\lesssim&\sum_{|\alpha|+|\beta|=N,|\beta|=1}
\left\|w_{l_1^*,1}\partial^{\alpha+e_i} \{{\bf I-P}\} f\langle v\rangle^{-\frac\gamma2-1}\right\|
^2+\varepsilon\sum_{|\alpha|+|\beta|=N,|\beta|=1}
\left\|w_{l_1^*-1,1}\partial_\beta^\alpha\{{\bf I-P}\} f\right\|^2_\nu
\end{eqnarray*}
As for $R_{10}$, it is straightforward to compute that
\[R_{10}\lesssim \|\nabla^{|\alpha|+1}f\|^2_\nu
+\|\nabla^{|\alpha|}E\|^2+\varepsilon
\left\|\nabla^{|\alpha|}\{{\bf I-P}\} f\right\|^2_\nu+\mathcal{E}_N(t)\mathcal{D}_N(t).\]
Applying the similar trick as $R_1\sim R_4 $  gives
\begin{eqnarray*}
\sum_{i=6}^9R_i&\lesssim & E^N_{tri,1}(t)
+\varepsilon\sum_{|\alpha|+|\beta|=N,|\beta|=1}\left\{(1+t)^{-1-\vartheta}\left\|w_{l_1^*-1,1}\partial_\beta^\alpha\{{\bf I-P}\} f\langle v\rangle\right\|^2+
\left\|w_{l_1^*-1,1}\partial_\beta^\alpha\{{\bf I-P}\} f\right\|^2_\nu\right\},
\end{eqnarray*}
where $E^{N}_{tri,1}(t)$ is given in \eqref{Enj} with $n=N$ and $j=1$.
Thus plugging the estimates on $R_5\sim R_{10}$ into (\ref{G-2}) yields
\begin{eqnarray*}
&&\frac{d}{dt}\sum_{|\alpha|+|\beta|=N,|\beta|=1}\left\|w_{l_1^*-1,1}\partial_\beta^\alpha\{{\bf I-P}\} f\right\|^2
+\sum_{|\alpha|+|\beta|=N,|\beta|=1}\left\|w_{l_1^*-1,1}\partial_\beta^\alpha\{{\bf I-P}\} f\right\|^2_\nu\\
&&+(1+t)^{-1-\vartheta}\sum_{|\alpha|+|\beta|=N,|\beta|=1}\left\|w_{l_1^*-1,1}\partial_\beta^\alpha\{{\bf I-P}\} f\langle v\rangle\right\|^2\\
&\lesssim&{\sum_{|\alpha|+|\beta|=N,|\beta|=1}
\left\|w_{l_1^*,1}\partial^{\alpha+e_i} \{{\bf I-P}\} f\langle v\rangle^{-\frac\gamma2-1}\right\|^2+\left\|\nabla^{|\alpha|+1}f\right\|^2_\nu+\left\|\nabla^{|\alpha|}E\right\|^2
+E^N_{tri,1}(t)}\\
&&{+\mathcal{E}_{N}(t)\mathcal{E}^1_{N_0,l_0,-\gamma}(t)
+\mathcal{E}_{N}(t)\mathcal{D}_{N}(t)+\sum_{|\alpha|+|\beta|=N,|\beta|=1}\left\{\eta\left\|w_{l_1^*,1}\partial^\alpha\{{\bf I-P}\} f\right\|^2_\nu
+\left\|\partial^{\alpha} \{{\bf I-P}\} f\right\|_\nu^2\right\}}.
\end{eqnarray*}
Similarly, we can obtain that
\begin{eqnarray*}
&&\frac{d}{dt}\sum_{|\alpha|+|\beta|=N,|\beta|=j}\left\|w_{l_1^*-j,1}\partial_\beta^\alpha\{{\bf I-P}\} f\right\|^2
+\sum_{|\alpha|+|\beta|=N,|\beta|=j}\left\|w_{l_1^*-1,1}\partial_\beta^\alpha\{{\bf I-P}\} f\right\|^2_\nu\\
&&+(1+t)^{-1-\vartheta}\sum_{|\alpha|+|\beta|=N,|\beta|=j}\left\|w_{l_1^*-1,1}\partial_\beta^\alpha\{{\bf I-P}\} f\langle v\rangle\right\|^2\\
&\lesssim&\sum_{|\alpha|+|\beta|=N,|\beta|=j}
\left\|w_{l_1^*-j+1,1}\partial^{\alpha+e_i}_{\beta-e_i} \{{\bf I-P}\} f\langle v\rangle^{-\frac\gamma2-1}\right\|^2+\left\|\nabla^{|\alpha|+1}f\right\|^2_\nu+\left\|\nabla^{|\alpha|}E\right\|^2
+E^N_{tri,j}(t)\\
&&
+\mathcal{E}_{N}(t)\mathcal{E}^1_{N_0,l_0,-\gamma}(t)
+\mathcal{E}_{N}(t)\mathcal{D}_{N}(t)+\sum_{|\alpha|+|\beta|=N,|\beta'|<j}\left\{\eta\left\|w_{l_1^*-|\beta'|,1}\partial^\alpha_{\beta'}\{{\bf I-P}\} f\right\|^2_\nu
+\left\|\partial^{\alpha} \{{\bf I-P}\} f\right\|_\nu^2\right\},
\end{eqnarray*}
for $2\leq j\leq N$, where $E^N_{tri,j}(t)$ is defined in \eqref{Enj} with $n=N$.
Taking summation over $0\leq j\leq N$, one deduces
\begin{eqnarray*}
&&\frac{d}{dt}\left\{\sum_{|\alpha|+|\beta|=N,\atop|\beta|=j,1\leq j\leq N}(1+t)^{-\sigma_{N,j}}
\left\|w_{l_1^*-j,1}\partial_\beta^\alpha\{{\bf I-P}\} f\right\|^2+\sum_{|\alpha|=N}(1+t)^{-\sigma_{N,0}}
\left\|w_{l_1^*,1}\partial^\alpha f\right\|^2\right\}\\
&&
+\sum_{|\alpha|+|\beta|=N,\atop|\beta|=j,1\leq j\leq N}(1+t)^{-\sigma_{N,j}}
\left\|w_{l_1^*-j,1}\partial_\beta^\alpha\{{\bf I-P}\} f\right\|^2_\nu
+\sum_{|\alpha|=N}(1+t)^{-\sigma_{N,0}}\left\|w_{l_1^*,1}\partial^\alpha f\right\|^2_\nu\\
&&+\sum_{|\alpha|+|\beta|=N,\atop|\beta|=j,1\leq j\leq N}(1+t)^{-1-\vartheta-\sigma_{N,j}}
\left\|w_{l_1^*-j,1}\partial_\beta^\alpha\{{\bf I-P}\} f\langle v\rangle\right\|^2
+\sum_{|\alpha|=N}(1+t)^{-1-\vartheta-\sigma_{N,0}}\left\|w_{l_1^*,1}\partial^\alpha f\langle v\rangle\right\|^2\\
&\lesssim&{\sum_{|\alpha|\leq N-1}\left\{\left\|\nabla^{|\alpha|+1}f\right\|^2_\nu+\|\{{\bf I-P}\} f\|^2_\nu+
\left\|\nabla^{|\alpha|}E\right\|^2\right\}+(1+t)^{-2\sigma_{N,0}}\left\|\nabla^NE\right\|^2+{\sum_{0\leq j\leq N}}(1+t)^{-\sigma_{N,j}}E^N_{tri,j}(t)}\\
&&{+\mathcal{E}_{N}(t)\mathcal{E}^1_{N_0,l_0,-\gamma}(t)
+\mathcal{E}_{N}(t)\mathcal{D}_{N}(t)+\eta\sum_{|\alpha|+|\beta|=N,\atop|\beta|=j,1\leq j\leq N,|\beta'|<j}(1+t)^{-\sigma_{N,j}}
\left\|w_{l_1^*-|\beta'|,1}\partial^\alpha_{\beta'}\{{\bf I-P}\} f\right\|^2_\nu,}
\end{eqnarray*}
where we have used the fact that
\begin{eqnarray}\label{sigma-d}
&&\sum_{|\alpha|+|\beta|=N,\atop|\beta|=j,1\leq j\leq N}(1+t)^{-\sigma_{N,j}}
\left\|w_{l_1^*-j+1,1}\partial^{\alpha+e_i}_{\beta-e_i} \{{\bf I-P}\} f\langle v\rangle^{-\frac\gamma2-1}\right\|
^2\\ \nonumber
&\lesssim&\sum_{|\alpha|+|\beta|=N,\atop|\beta|=j,1\leq j\leq N}(1+t)^{-\sigma_{N,j-1}-\frac{2(1+\gamma)}{\gamma-2}(1+\vartheta)}
\left\|w_{l_1^*-j+1,1}\partial^{\alpha+e_i}_{\beta-e_i} \{{\bf I-P}\} f\langle v\rangle\right\|^{\frac{4(1+\gamma)}{\gamma-2}}
\left\|w_{l_1^*-j+1,1}\partial^{\alpha+e_i}_{\beta-e_i} \{{\bf I-P}\} f\langle v\rangle^{\frac{\gamma}2}\right\|^{\frac{-2\gamma-8}{\gamma-2}}\\ \nonumber
&\lesssim&\sum_{|\alpha|+|\beta|=N,\atop|\beta|=j,1\leq j\leq N}\left\{(1+t)^{-\sigma_{N,j-1}-1-\vartheta}
\left\|w_{l_1^*-j+1,1}\partial^{\alpha+e_i}_{\beta-e_i} \{{\bf I-P}\} f\langle v\rangle\right\|^{2}
+(1+t)^{-\sigma_{N,j-1}}\left\|w_{l_1^*-j+1,1}\partial^{\alpha+e_i}_{\beta-e_i} \{{\bf I-P}\} f\langle v\rangle^{\frac{\gamma}2}\right\|^2\right\},
\end{eqnarray}
where follows from the fact that  $\sigma_{N,j}-\sigma_{N,j-1}=\frac{2(1+\gamma)}{\gamma-2}(1+\vartheta)$.

When $N_0+1\leq n\leq N-1$, we can deduce similarly that
\begin{eqnarray*}
&&\frac{d}{dt}\left\{\sum_{|\alpha|+|\beta|=n,\atop|\beta|=j,1\leq j\leq n}(1+t)^{-\sigma_{n,j}}
\left\|w_{l_1^*-j,1}\partial_\beta^\alpha\{{\bf I-P}\} f\right\|^2+\sum_{|\alpha|=n}(1+t)^{-\sigma_{n,0}}
\left\|w_{l_1^*,1}\partial^\alpha f\right\|^2\right\}\\
&&
+\sum_{|\alpha|+|\beta|=n,\atop|\beta|=j,1\leq j\leq n}(1+t)^{-\sigma_{n,j}}
\left\|w_{l_1^*-j,1}\partial_\beta^\alpha\{{\bf I-P}\} f\right\|^2_\nu
+\sum_{|\alpha|=n}(1+t)^{-\sigma_{n,0}}\left\|w_{l_1^*,1}\partial^\alpha f\right\|^2_\nu\\
&&+\sum_{|\alpha|+|\beta|=n,\atop|\beta|=j,1\leq j\leq n}(1+t)^{-1-\vartheta-\sigma_{n,j}}
\left\|w_{l_1^*-j,1}\partial_\beta^\alpha\{{\bf I-P}\} f\langle v\rangle\right\|^2
+\sum_{|\alpha|=n}(1+t)^{-1-\vartheta-\sigma_{n,0}}\left\|w_{l_1^*,1}\partial^\alpha f\langle v\rangle\right\|^2\\
&\lesssim&{\sum_{|\alpha|\leq n-1}\left\{\left\|\nabla^{|\alpha|+1}f\right\|^2_\nu
+\left\|\{{\bf I-P}\} f\right\|^2_\nu+\left\|\nabla^{|\alpha|}E\right\|^2\right\}
+(1+t)^{-2\sigma_{n,0}}\left\|\nabla^nE\right\|^2+\mathcal{E}_{N}(t)\mathcal{E}^1_{N_0,l_0,-\gamma}(t)}\\
&&+{{\sum_{0\leq j\leq n}}(1+t)^{-\sigma_{n,j}}E^n_{tri,j}(t)
+\mathcal{E}_{N}(t)\mathcal{D}_{N}(t)+\eta\sum_{|\alpha|+|\beta|=n,\atop|\beta|=j,1\leq j\leq n,|\beta'|<j}(1+t)^{-\sigma_{n,j}}
\left\|w_{l_1^*-|\beta'|,1}\partial^\alpha_{\beta'}\{{\bf I-P}\} f\right\|^2_\nu.}
\end{eqnarray*}
Here we have used $\sigma_{n,j}-\sigma_{n,j-1}=\frac{2(1+\gamma)}{\gamma-2}(1+\vartheta)$ and recall that $E^n_{tri,j}(t)$ is given in \eqref{Enj}.
%
Taking summation over $N_0+1\leq n\leq N$ gives \eqref{lemma7-1}. This completes the proof of Lemma \ref{lemma7}.\qed

\medskip
\noindent 4.5 {\bf  The proof of Lemma \ref{lemma8}.}
Similar to the proof of Lemma \ref{lemma7}, we have firstly that
\begin{eqnarray}\label{H-1}
&&\frac{d}{dt}\sum_{|\alpha|=N_0}\left\|w_{l_0^*,1}\partial^\alpha f\right\|^2+\sum_{|\alpha|=N_0}(1+t)^{-1-\vartheta}
\left\|w_{l_0^*,1}\partial^\alpha f\langle v\rangle\right\|^2+\sum_{|\alpha|=N_0}
\left\|w_{l_0^*,1}\partial^\alpha f\right\|_\nu^2\\ \nonumber
&\lesssim&\sum_{|\alpha|=N_0}\left\|\partial^\alpha f\right\|_\nu^2
+\sum_{|\alpha|=N_0}\left\|\partial^\alpha f\right\|_\nu\left\|\partial^\alpha E\right\|
+\underbrace{\sum_{|\alpha|=N_0}\left(\partial^\alpha((v\times B)\cdot\nabla_v f),w^2_{l_0^*,1}\partial^\alpha f\right)}_{H_1}\\ \nonumber
&&
+\underbrace{\sum_{|\alpha|=N_0}\left(\partial^\alpha(E\cdot\nabla_v f),w^2_{l_0^*,1}\partial^\alpha f\right)}_{H_2}+\underbrace{\sum_{|\alpha|=N_0}\left(\partial^\alpha(v\cdot E f),w^2_{l_0^*,1}\partial^\alpha f\right)}_{H_3}\\
&&+\underbrace{\sum_{|\alpha|=N_0}\left(\partial^\alpha\Gamma(f,f),w^2_{l_0^*,1}\partial^\alpha f\right)}_{H_4}.\nonumber
\end{eqnarray}
Applying the H$\ddot{o}$lder inequality and the Sobolev inequality, one has
\begin{eqnarray*}
H_1&\lesssim&
\sum_{1\leq|\alpha_1|\leq N_0-2}\left\|\partial^{\alpha_1} B\right\|_{L^\infty_x}
\left\|w_{l_0^*-1,1}\nabla_v \partial^{\alpha-\alpha_1}{\bf \{I-P\}}f\langle v\rangle\right\|
\left\|w_{l_0^*,1}\partial^\alpha f\langle v\rangle\right\|\\
&&+\sum_{N_0-1\leq|\alpha_1|\leq N_0}\left\|\partial^{\alpha_1} B\right\|
\left\|w_{l_0^*-1,1}\nabla_v \partial^{\alpha-\alpha_1}{\bf \{I-P\}}f\langle v\rangle\right\|_{L^2_vL^\infty_x}
\left\|w_{l_0^*,1}\partial^\alpha f\langle v\rangle\right\|\\
&&+\sum_{\alpha_1\leq\alpha}\int_{\mathbb{R}^3_x}
\left|\partial^{\alpha_1}B\right|\left|\mu^{\delta}\partial^{\alpha-\alpha_1}f\right|_{L^2_v}
\left|\mu^{\delta}\partial^{\alpha}f\right|_{L^2_v}dx\\
&\lesssim&\sum_{1\leq|\alpha_1|\leq N_0-2}(1+t)^{1+\vartheta}
\left\|\partial^{\alpha_1} B\right\|^2_{L^\infty_x}
\left\|w_{l_0^*-1,1}\nabla_v \partial^{\alpha-\alpha_1}{\bf \{I-P\}}f\langle v\rangle\right\|^2\\
&&+\sum_{N_0-1\leq|\alpha_1|\leq N_0}(1+t)^{1+\vartheta}
\left\|\partial^{\alpha_1} B\right\|^2
\left\|w_{l_0^*-1,1}\nabla_v \partial^{\alpha-\alpha_1}{\bf \{I-P\}}f\langle v\rangle\right\|^2_{L^2_vL^\infty_x}\\
&&{+\mathcal{E}_{N_0}(t)\mathcal{D}_{N_0}(t)+\varepsilon(1+t)^{-1-\vartheta}\left\|w_{l_0^*,1}\partial^\alpha f\langle v\rangle\right\|^2
+\varepsilon\left\|\mu^{\delta}\partial^\alpha f\right\|^2}.
\end{eqnarray*}
In a similar way, we can also get that
\begin{eqnarray*}
H_2+H_3&\lesssim&\sum_{1\leq|\alpha_1|\leq N_0-2,m\leq1}(1+t)^{1+\vartheta}
\left\|\partial^{\alpha_1} E\right\|^2_{L^\infty_x}
\left\|w_{l_0^*-m,1}\nabla^m_v \partial^{\alpha-\alpha_1}{\bf \{I-P\}}f\right\|^2\\
&&+\sum_{N_0-1\leq|\alpha_1|\leq N_0,m\leq1}(1+t)^{1+\vartheta}
\left\|\partial^{\alpha_1} E\right\|^2
\left\|w_{l_0^*-m,1}\nabla^m_v \partial^{\alpha-\alpha_1}{\bf \{I-P\}}f\right\|^2_{L^2_vL^\infty_x}\\
&&+\sum_{|\alpha|=N_0}\|E\|_{L^\infty}^{\frac{2-\gamma}{1-\gamma}}
\left\|w_{l_0^*,1}\partial^\alpha \{{\bf I-P}\}f\langle v\rangle\right\|^2+\varepsilon\left\|w_{l_0^*,1}\partial^\alpha f\right\|^2_\nu\\
&&+\varepsilon(1+t)^{-1-\vartheta}\left\|w_{l_0^*,1}\partial^\alpha \{{\bf I-P}\}f\langle v\rangle\right\|^2+\mathcal{E}_{N_0}(t)\mathcal{D}_{N_0}(t).
\end{eqnarray*}
As for $H_4$, Lemma \ref{lemma-nonlinear} suggests that
\begin{eqnarray*}
H_4&\lesssim&
\underbrace{\sum_{|\alpha_1|= 0,m\leq2}\int_{\mathbb{R}^3}\left|\nabla^m_v\left(\mu^{\delta}f\right)\right|
\left|w_{l_0^*,1}\partial^{\alpha}f\right|_\nu\left|w_{l_0^*,1}\partial^{\alpha}f\right|_\nu dx}_{H_{4,1}}\\
&&+\underbrace{\sum_{1\leq|\alpha_1|\leq N_0-3,m\leq2}\int_{\mathbb{R}^3}\left|\nabla^m_v\left(\mu^{\delta}\partial^{\alpha_1}f\right)\right|
\left|w_{l_0^*,1}\partial^{\alpha-\alpha_1}f\right|_\nu\left|w_{l_0^*,1}\partial^{\alpha}f\right|_\nu dx}_{H_{4,2}}\\
&&+\underbrace{\sum_{|\alpha_1|=N_0-2 \ or\  N_0-1 ,m\leq2}\int_{\mathbb{R}^3}\left|w_{l_0^*,1}\partial^{\alpha_1}f\right|_\nu
\left|\nabla^m_v\left(\mu^{\delta}\partial^{\alpha-\alpha_1}f\right)\right|\left|w_{l_0^*,1}\partial^{\alpha}f\right|_\nu dx}_{H_{4,3}}\\
&&+\underbrace{\sum_{\alpha_1=\alpha,m\leq2}\int_{\mathbb{R}^3}\left|w_{l_0^*,1}\partial^{\alpha}f\right|_\nu
\left|\nabla^m_v\left(\mu^{\delta}f\right)\right|\left|w_{l_0^*,1}\partial^{\alpha}f\right|_\nu dx}_{H_{4,4}}\\
&&+\underbrace{\sum_{|\alpha_1|= 0}\int_{\mathbb{R}^3}\left|w_{l_0^*,1}f\right|
\left|w_{l_0^*,1}\partial^{\alpha}f\right|_\nu\left|w_{l_0^*,1}\partial^{\alpha}f\right|_\nu dx}_{H_{4,5}}\\
&&+\underbrace{\sum_{1\leq|\alpha_1|\leq N_0-1}\int_{\mathbb{R}^3}\left|w_{l_0^*,1}\partial^{\alpha_1}f\right|
\left|w_{l_0^*,1}\partial^{\alpha-\alpha_1}f\right|_\nu\left|w_{l_0^*,1}\partial^{\alpha}f\right|_\nu dx}_{H_{4,6}}\\
&&+\underbrace{\sum_{\alpha_1= \alpha}\int_{\mathbb{R}^3}\left|w_{l_0^*,1}f\right|
\left|w_{l_0^*,1}\partial^{\alpha}f\right|_\nu\left|w_{l_0^*,1}\partial^{\alpha}f\right|_\nu dx}_{H_{4,7}}.
\end{eqnarray*}
By using $L^2-L^\infty-L^2$ or $L^\infty-L^2-L^2$ type inequalities with respect to space derivative $x$, one has
\begin{eqnarray*}
H_{4,1}+H_{4,4}+H_{4,5}+H_{4,7}&\lesssim&
\left\{\mathcal{E}_{N_0,0}(t)+\left\|w_{l_0^*,1}f\right\|_{L^2_vL^\infty_x}^2\right\}\left\|w_{l_0^*,1}\partial^{\alpha}f\right\|^2_\nu
+\varepsilon\left\|w_{l_0^*,1}\partial^{\alpha}f\right\|^2_\nu,
\end{eqnarray*}
while employing $L^3-L^6-L^2$ or $L^6-L^3-L^2$ type inequalities gives
\begin{eqnarray*}
H_{4,2}+H_{4,3}+H_{4,6}
&\lesssim&
\sum_{1\leq|\alpha_1|\leq N_0-3,m\leq2}\left\|\nabla^m_v\left(\mu^{\delta}\partial^{\alpha_1}f\right)\right\|_{L^3_x}
\left\|w_{l_0^*,1}\partial^{\alpha-\alpha_1}f\right\|^2_{L^2_\nu L^6_x}\\
&&+\sum_{|\alpha_1|=N_0-2,m\leq2}\left\|w_{l_0^*,1}\partial^{\alpha_1}f\right\|^2_{L^2_\nu L^\infty_x}
\left\|\nabla^m_v\left(\mu^{\delta}\partial^{\alpha-\alpha_1}f\right)\right\|^2\\
&&+\sum_{|\alpha_1|=N_0-1,m\leq2}\left\|w_{l_0^*,1}\partial^{\alpha_1}f\right\|^2_{L^2_\nu L^6_x}
\left\|\nabla^m_v\left(\mu^{\delta}\partial^{\alpha-\alpha_1}f\right)\right\|^2_{L^2_vL^3_x}\\
&&+\sum_{1\leq|\alpha_1|\leq N_0-1}\left\|w_{l_0^*,1}\partial^{\alpha_1}f\right\|^2_{L^2_vL^3_x}
\left\|w_{l_0^*,1}\partial^{\alpha-\alpha_1}f\right\|^2_{L^2_\nu L^6_x}+\varepsilon\left\|w_{l_0^*,1}\partial^{\alpha}f\right\|^2_\nu.
\end{eqnarray*}
Consequently
\begin{eqnarray*}
H_4&\lesssim& \left\{\mathcal{E}_{N_0,0}(t)+\left\|w_{l_0^*,1}f\right\|_{L^2_vL^\infty_x}^2\right\}
\sum_{1\leq|\alpha_1|\leq N_0}\left\|w_{l_0^*,1}\partial^{\alpha_1}f\right\|^2_\nu\\
&&
+\sum_{1\leq|\alpha_1|\leq N_0-1}\left\|w_{l_0^*,1}\partial^{\alpha_1}f\right\|^2_{L^3_x}
\left\|w_{l_0^*,1}\partial^{\alpha-\alpha_1}f\right\|^2_{L^2_\nu L^6_x}+\varepsilon\left\|w_{l_0^*,1}\partial^{\alpha}f\right\|^2_\nu.
\end{eqnarray*}
Recalling \eqref{Fnj} with $n=N_0$, $j=0$ for $F^{N_0}_{tri,0}(t)$,
it follows by inserting the estimates on $H_1\sim H_4$ into (\ref{H-1}) that
\begin{eqnarray*}
&&\frac{d}{dt}\sum_{|\alpha|=N_0}\left\|w_{l_0^*,1}\partial^\alpha f\right\|^2
+\sum_{|\alpha|=N_0}\left\|w_{l_0^*,1}\partial^\alpha f\right\|_\nu^2+\sum_{|\alpha|=N_0}(1+t)^{-1-\vartheta}
\left\|w_{l_0^*,1}\partial^\alpha f\langle v\rangle\right\|^2\\ \nonumber
&\lesssim&\sum_{|\alpha|=N_0}\left\|\partial^\alpha f\right\|_\nu^2
+F^{N_0}_{tri,0}(t)+\|\nabla^{N_0}E\|^2
+\mathcal{E}_{N_0}(t)\mathcal{D}_{N_0}(t).\nonumber
\end{eqnarray*}
When $|\alpha|+|\beta|=N_0,|\beta|=1$, one has
\begin{eqnarray}\label{H-2}
&&\frac{d}{dt}\sum_{|\alpha|+|\beta|=N_0,|\beta|=1}\left\|w_{l_0^*-1,1}\partial_\beta^\alpha\{{\bf I-P}\} f\right\|^2
+\sum_{|\alpha|+|\beta|=N_0,|\beta|=1}\left\|w_{l_0^*-1,1}\partial_\beta^\alpha\{{\bf I-P}\} f\right\|^2_\nu
\\ \nonumber
&&+(1+t)^{-1-\vartheta}\sum_{|\alpha|+|\beta|=N_0,|\beta|=1}\left\|w_{l_0^*-1,1}\partial_\beta^\alpha\{{\bf I-P}\} f\langle v\rangle\right\|^2\\ \nonumber
&\lesssim&\sum_{|\alpha|+|\beta|=N,|\beta|=1}\left\{\eta\left\|w_{l_0^*,1}\partial^\alpha\{{\bf I-P}\} f\right\|^2_\nu
+\left\|\partial^{\alpha} \{{\bf I-P}\} f\right\|_\nu^2\right\}\\ \nonumber
&&+\underbrace{\sum_{|\alpha|+|\beta|=N_0,|\beta|=1}\left(\partial_\beta^\alpha(v\cdot \{{\bf I-P}\} f),w^2_{l_0^*-1,1}\partial_\beta^\alpha\{{\bf I-P}\} f\right)}_{H_5}\\ \nonumber
&&+\underbrace{\sum_{|\alpha|+|\beta|=N_0,|\beta|=1}\left(\partial_\beta^\alpha((v\times B)\cdot\nabla_v \{{\bf I-P}\} f),w^2_{l_0^*-1,1}\partial_\beta^\alpha\{{\bf I-P}\} f\right)}_{H_6}
\\ \nonumber&&+\underbrace{\sum_{|\alpha|+|\beta|=N_0,|\beta|=1}\left(\partial_\beta^\alpha(E\cdot\nabla_v \{{\bf I-P}\} f),w^2_{l_0^*-1,1}\partial_\beta^\alpha\{{\bf I-P}\} f\right)}_{H_7}\\ \nonumber
&&+\underbrace{\sum_{|\alpha|+|\beta|=N_0,|\beta|=1}\left(\partial_\beta^\alpha(v\cdot E \{{\bf I-P}\} f),w^2_{l_0^*-1,1}\partial_\beta^\alpha\{{\bf I-P}\} f\right)}_{H_8}
\\ \nonumber &&+\underbrace{\sum_{|\alpha|+|\beta|=N_0,|\beta|=1}\left(\partial_\beta^\alpha\Gamma(f,f),w^2_{l_0^*-1,1}\partial_\beta^\alpha\{{\bf I-P}\} f\right)}_{H_9}
\\ \nonumber
&&+\underbrace{\sum_{|\alpha|+|\beta|=N_0,|\beta|=1}\left(\partial_\beta^\alpha I_{mac}(t),w^2_{l_0^*-1,1}\partial_\beta^\alpha\{{\bf I-P}\} f\right)}_{H_{10}},
\end{eqnarray}
{where $I_{mac}(t)$ is given in \eqref{I-mac}.}
$H_5$ can be dominated by
\begin{eqnarray*}
H_5&\lesssim&\sum_{|\alpha|+|\beta|=N_0,|\beta|=1}\left(\partial^{\alpha+e_i} \{{\bf I-P}\} f)
,w^2_{l_0^*-1,1}\partial_\beta^\alpha\{{\bf I-P}\} f\right)\\
&\lesssim&\sum_{|\alpha|+|\beta|=N_0,|\beta|=1}
\left\|w_{l_0^*,1}\partial^{\alpha+e_i} \{{\bf I-P}\} f\langle v\rangle^{-\frac\gamma2-1}\right\|
\left\|w_{l_0^*-1,1}\partial_\beta^\alpha\{{\bf I-P}\} f\langle v\rangle^{\frac\gamma2}\right\|\\
&\lesssim&\sum_{|\alpha|+|\beta|=N_0,|\beta|=1}
\left\|w_{l_0^*,1}\partial^{\alpha+e_i} \{{\bf I-P}\} f\langle v\rangle^{-\frac\gamma2-1}\right\|
^2+\varepsilon\sum_{|\alpha|+|\beta|=N_0,|\beta|=1}
\left\|w_{l_0^*-1,1}\partial_\beta^\alpha\{{\bf I-P}\} f\right\|^2_\nu.
\end{eqnarray*}
Applying the same trick as $H_1\sim H_4$ suggests that
\begin{eqnarray*}
\sum_{i=6}^{10}H_i&\lesssim& F^{N_0}_{tri,1}(t)+\varepsilon(1+t)^{-1-\vartheta}\sum_{|\alpha|+|\beta|=N_0,|\beta|=1}
\left\|w_{l_0^*-1,1}\partial_\beta^\alpha\{{\bf I-P}\} f\langle v\rangle\right\|^2\\[3mm]
&&+\left\|\nabla^{|\alpha|}E\right\|^2+\left\|\nabla^{|\alpha|+1}f\right\|^2_\nu+\varepsilon\sum_{|\alpha|+|\beta|=N_0,|\beta|=1}
\left\|w_{l_0^*-1,1}\partial_\beta^\alpha\{{\bf I-P}\} f\right\|^2_\nu,
\end{eqnarray*}
where $F^{N_0}_{tri,1}(t)$ is given in \eqref{Fnj} with $n=N_0$ and $j=1$.
Thus plugging the estimates on $H_5\sim H_{10}$ into (\ref{H-2}) yields
\begin{eqnarray*}
&&\frac{d}{dt}\sum_{|\alpha|+|\beta|=N_0,|\beta|=1}\left\|w_{l_0^*-1,1}\partial_\beta^\alpha\{{\bf I-P}\} f\right\|^2
+\sum_{|\alpha|+|\beta|=N_0,|\beta|=1}\left\|w_{l_0^*-1,1}\partial_\beta^\alpha\{{\bf I-P}\} f\right\|^2_\nu
\\ \nonumber
&&+(1+t)^{-1-\vartheta}\sum_{|\alpha|+|\beta|=N_0,|\beta|=1}\left\|w_{l_0^*-1,1}\partial_\beta^\alpha\{{\bf I-P}\} f\langle v\rangle\right\|^2\\ \nonumber
&\lesssim&{\sum_{|\alpha|+|\beta|=N_0,|\beta|=1}
\left\|w_{l_0^*,1}\partial^{\alpha+e_i} \{{\bf I-P}\} f\langle v\rangle^{-\frac\gamma2-1}\right\|
^2+\left\|\nabla^{|\alpha|+1}f\right\|^2_\nu+\left\|\nabla^{|\alpha|}E\right\|^2+F^{N_0}_{tri,1}(t)
}\\
&&+{\mathcal{E}_{N_0}(t)\mathcal{D}_{N_0}(t)+\sum_{|\alpha|+|\beta|=N_0,|\beta|=1}\left\{\eta\left\|w_{l_0^*,1}\partial^\alpha\{{\bf I-P}\} f\right\|^2_\nu
+\left\|\partial^{\alpha} \{{\bf I-P}\} f\right\|_\nu^2\right\}}.
\end{eqnarray*}
Similarly, we can get for $2\leq j\leq N_0$ that
\begin{eqnarray*}
&&\frac{d}{dt}\sum_{|\alpha|+|\beta|=N_0,|\beta|=j}\left\|w_{l_0^*-j,1}\partial_\beta^\alpha\{{\bf I-P}\} f\right\|^2
+\sum_{|\alpha|+|\beta|=N_0,|\beta|=j}\left\|w_{l_0^*-1,1}\partial_\beta^\alpha\{{\bf I-P}\} f\right\|^2_\nu
\\ \nonumber
&&+(1+t)^{-1-\vartheta}\sum_{|\alpha|+|\beta|=N_0,|\beta|=j}\left\|w_{l_0^*-1,1}\partial_\beta^\alpha\{{\bf I-P}\} f\langle v\rangle\right\|^2\\ \nonumber
&\lesssim&{\sum_{|\alpha|+|\beta|=N_0,|\beta|=j}
\left\|w_{l_0^*-j+1,1}\partial^{\alpha+e_i}_{\beta-e_i} \{{\bf I-P}\} f\langle v\rangle^{-\frac\gamma2-1}\right\|
^2+\left\|\nabla^{|\alpha|+1}f\right\|^2_\nu
+\left\|\nabla^{|\alpha|}E\right\|^2+F^{N_0}_{tri,j}(t)}\\
&&+{\mathcal{E}_{N_0}(t)\mathcal{D}_{N_0}(t)+\sum_{|\alpha|+|\beta|=N_0,\atop
{|\beta|=j,|\beta'|<j}}\left\{\eta\left\|w_{l_0^*-|\beta'|,1}\partial^\alpha_{\beta'}\{{\bf I-P}\} f\right\|^2_\nu
+\left\|\partial^{\alpha} \{{\bf I-P}\} f\right\|_\nu^2\right\}},
\end{eqnarray*}
where $F^{N_0}_{tri,j}(t)$ is given in \eqref{Fnj} with $n=N_0$.
Taking summation over $0\leq j\leq N_0$, one deduces
\begin{eqnarray*}
&&\frac{d}{dt}\left\{\sum_{|\alpha|+|\beta|=N_0,\atop|\beta|=j,1\leq j\leq N_0}(1+t)^{-\sigma_{N_0,j}}
\left\|w_{l_0^*-j,1}\partial_\beta^\alpha\{{\bf I-P}\} f\right\|^2+\sum_{|\alpha|=N_0}(1+t)^{-\sigma_{N_0,0}}
\left\|w_{l_0^*,1}\partial^\alpha f\right\|^2\right\}\\ \nonumber
&&
+\sum_{|\alpha|+|\beta|=N_0,\atop|\beta|=j,1\leq j\leq N_0}(1+t)^{-\sigma_{N_0,j}}
\left\|w_{l_0^*-j,1}\partial_\beta^\alpha\{{\bf I-P}\} f\right\|^2_\nu
+\sum_{|\alpha|=N_0}(1+t)^{-\sigma_{N,0}}\left\|w_{l_0^*,1}\partial^\alpha f\right\|^2_\nu
\\ \nonumber
&&+\sum_{|\alpha|+|\beta|=N_0,\atop|\beta|=j,1\leq j\leq N_0}(1+t)^{-1-\vartheta-\sigma_{N_0,j}}
\left\|w_{l_0^*-j,1}\partial_\beta^\alpha\{{\bf I-P}\} f\langle v\rangle\right\|^2+\sum_{|\alpha|=N_0}(1+t)^{-1-\vartheta-\sigma_{N_0,0}}
\left\|w_{l_0^*,1}\partial^\alpha f\langle v\rangle\right\|^2\\ \nonumber
&\lesssim&{\sum_{|\alpha|\leq N_0-1}\left\{\left\|\nabla^{|\alpha|+1}f\right\|_\nu^2
+\left\| \{{\bf I-P}\} f\right\|_\nu^2+\left\|\nabla^{|\alpha|}E\right\|^2+\left\|\nabla^{N_0}E\right\|^2\right\}
+\mathcal{E}_{N_0}(t)\mathcal{D}_{N_0}(t)}\\ \nonumber
&&{+{\sum_{0\leq j\leq {N_0}}}(1+t)^{-\sigma_{N_0,j}}F^{N_0}_{tri,j}(t)+\eta\sum_{|\alpha|+|\beta|=N_0,\atop|\beta|=j,1\leq j\leq N_0,|\beta'|<j}(1+t)^{-\sigma_{N_0,j}}
\left\|w_{l_0^*-|\beta'|,1}\partial^\alpha_{\beta'}\{{\bf I-P}\} f\right\|^2_\nu,}
\end{eqnarray*}
where we have used the facts that
$$
\sigma_{N_0,j}-\sigma_{N_0,j-1}=\frac{2(1+\gamma)}{\gamma-2}(1+\vartheta)
$$
and {in a similar way as \eqref{sigma-d},
\begin{eqnarray*}
&&\sum_{|\alpha|+|\beta|=N_0,\atop|\beta|=j,1\leq j\leq N_0}(1+t)^{-\sigma_{N_0,j}}
\left\|w_{l_0^*-j+1,1}\partial^{\alpha+e_i}_{\beta-e_i} \{{\bf I-P}\} f\langle v\rangle^{-\frac\gamma2-1}\right\|
^2\\ \nonumber
&\lesssim&\sum_{|\alpha|+|\beta|=N_0,\atop|\beta|=j,1\leq j\leq N_0}\left\{(1+t)^{-\sigma_{N_0,j-1}-1-\vartheta}
\left\|w_{l_0^*-j+1,1}\partial^{\alpha+e_i}_{\beta-e_i} \{{\bf I-P}\} f\langle v\rangle\right\|^{2}\right.\\
&&\qquad\qquad\qquad\qquad\qquad\qquad\left.+(1+t)^{-\sigma_{N_0,j-1}}\left\|w_{l_0^*-j+1,1}\partial^{\alpha+e_i}_{\beta-e_i} \{{\bf I-P}\} f\right\|^2_\nu\right\}.
\end{eqnarray*}}
By exploiting the same argument as before, we can get for $1\leq n\leq N_0-1$ that
\begin{eqnarray*}
&&\frac{d}{dt}\Bigg\{\sum_{|\alpha|+|\beta|=n,\atop|\beta|=j,1\leq j\leq n}(1+t)^{-\sigma_{n,j}}
\left\|w_{l_0^*-j,1}\partial_\beta^\alpha\{{\bf I-P}\} f\right\|^2+\sum_{|\alpha|=n}(1+t)^{-\sigma_{n,0}}
\left\|w_{l_0^*,1}\partial^\alpha f\right\|^2\\ \nonumber
&&+(1+t)^{-\sigma_{0,0}}\left\|w_{l_0^*,1}\{{\bf I-P}\} f\right\|^2\Bigg\}
+\sum_{|\alpha|+|\beta|=n,\atop|\beta|=j,1\leq j\leq n}(1+t)^{-\sigma_{n,j}}
\left\|w_{l_0^*-j,1}\partial_\beta^\alpha\{{\bf I-P}\} f\right\|^2_\nu\\ \nonumber
&&+\sum_{|\alpha|=n}(1+t)^{-\sigma_{n,0}}\left\|w_{l_0^*,1}\partial^\alpha f\right\|_\nu^2
+(1+t)^{-\sigma_{0,0}}\left\|w_{l_0^*,1}\{{\bf I-P}\} f\right\|^2_\nu
\\ \nonumber
&&+\sum_{|\alpha|+|\beta|=n,\atop|\beta|=j,1\leq j\leq n}(1+t)^{-1-\vartheta-\sigma_{n,j}}
\left\|w_{l_0^*-j,1}\partial_\beta^\alpha\{{\bf I-P}\} f\langle v\rangle\right\|^2
\\ \nonumber
&&+\sum_{|\alpha|=n}(1+t)^{-1-\vartheta-\sigma_{n,0}}\left\|w_{l_0^*,1}\partial^\alpha f\langle v\rangle\right\|^2+(1+t)^{-1-\vartheta-\sigma_{0,0}}\left\|w_{l_0^*,1}\{{\bf I-P}\} f\langle v\rangle\right\|^2\\ \nonumber
&\lesssim&{\sum_{|\alpha|\leq n-1}\left\{\left\|\nabla^{|\alpha|+1}f\right\|_\nu^2+\left\|\{{\bf I-P}\} f\right\|_\nu^2
+\left\|\nabla^{|\alpha|}E\right\|^2+\left\|\nabla^{n}E\right\|^2\right\}+\mathcal{E}_{N_0}(t)\mathcal{D}_{N_0}(t)}
\\ \nonumber
&&+{{\sum_{0\leq j\leq n}}(1+t)^{-\sigma_{n,j}}F^n_{tri,j}(t)+\eta{\sum_{|\alpha|+|\beta|=n,\atop|\beta|=j,1\leq j\leq n,|\beta'|<j}}(1+t)^{-\sigma_{n,j}}
\left\|w_{l_0^*-|\beta'|,1}\partial^\alpha_{\beta'}\{{\bf I-P}\} f\right\|^2_\nu,}
\end{eqnarray*}
where  $F^n_{tri,j}(t)$ is given in \eqref{Fnj}.
With the above estimates in hand, \eqref{lemma8-1} follows by taking summation over $1\leq n\leq N_0$ and by using the energy estimates on $\|w_{l_0^*,1}\{{\bf I-P}\} f\|^2$. Thus we have completed the proof of Lemma \ref{lemma8}.\qed

\bigskip
\noindent {\bf Acknowledgements:}
RJD was supported by the General Research Fund (Project No. 400511) from RGC of Hong Kong. TY was supported by the General Research Fund of Hong Kong, CityU No.103412. HJZ was support by two grants from the National Natural Science Foundation of China under contracts 10925103 and 11271160 respectively. This work was also supported by a grant from the National Natural Science Foundation of China under contract 11261160485 and the Fundamental Research Funds for the Central Universities.

\end{document}